\documentclass[amscd, leqno, amstex]{amsart}
\usepackage{amscd}
\usepackage{amsmath}
\usepackage{amssymb}
\theoremstyle{plain}
\newtheorem{theorem}{Theorem}[section]
\newtheorem{proposition}[theorem]{Proposition}
\newtheorem{lemma}[theorem]{Lemma}
\newtheorem{claim}[theorem]{Claim}
\newtheorem{corollary}[theorem]{Corollary}
\theoremstyle{definition}
\newtheorem{definition}[theorem]{Definition}

\theoremstyle{remark}

\newtheorem{remark}[theorem]{Remark}

\newenvironment{pf}{\begin{proof}}{\end{proof}}
\makeatletter

\@addtoreset{equation}{section}
\makeatother

\begin{document}
 
\title
[$K3$ surfaces with involution IV]
{$K3$ surfaces with involution, equivariant analytic torsion, and automorphic forms on the moduli space IV:  the structure of invariant}
\author{Shouhei Ma}
\address{Department of Mathematics, 
Tokyo Institute of Technology, 
Tokyo 152-8551, JAPAN}
\email{ma@math.titech.ac.jp}
\author{Ken-Ichi Yoshikawa}
\address{Department of Mathematics, 
Faculty of Science,
Kyoto University,
Kyoto 606-8502, JAPAN}
\email{yosikawa@math.kyoto-u.ac.jp}

\begin{abstract}
In \cite{Yoshikawa04}, a holomorphic torsion invariant of $K3$ surfaces with involution was introduced. 
In this paper, we completely determine its structure as an automorphic function on the moduli space of such $K3$ surfaces.
On every component of the moduli space, it is expressed as the product of an explicit Borcherds lift and a classical Siegel modular form. 
We also introduce its twisted version. We prove its modularity and a certain uniqueness of the modular form corresponding to
the twisted holomorphic torsion invariant. This is used to study an equivariant analogue of Borcherds' conjecture.
\end{abstract}

\maketitle

\tableofcontents

\section*{Introduction}\label{sect:0}
In \cite{Yoshikawa04}, a holomorphic torsion invariant of $K3$ surfaces with involution was introduced by the second-named author.
The purpose of the present paper is to completely determine the structure of this invariant as a function on the moduli space of such $K3$ surfaces.
We will express it using Borcherds products and Siegel modular forms. 
Let us explain our result in detail.
\par
A pair consisting of a $K3$ surface and its anti-symplectic involution is called a $2$-elementary $K3$ surface.
By Nikulin \cite{Nikulin83}, the deformation type of a $2$-elementary $K3$ surface is determined by the isometry class of 
the invariant lattice of the induced involution on the second integral cohomology.
There exist $75$ deformation types, labeled by primitive $2$-elementary Lorentzian sublattices of the K3 lattice ${\Bbb L}_{K3}$.
Let $M\subset{\Bbb L}_{K3}$ be one such sublattice of rank $r$. 
Its orthogonal complement $\Lambda = M^{\perp}\cap{\Bbb L}_{K3}$ is a $2$-elementary lattice of signature $(2,20-r)$. 
Let $\Omega_{\Lambda}^{+}$ be the Hermitian domain of type IV of dimension $20-r$ associated to $\Lambda$
and let $O^{+}(\Lambda)$ be the index $2$ subgroup of the orthogonal group of $\Lambda$ preserving $\Omega_{\Lambda}^{+}$. 
Via the period map, the moduli space of $2$-elementary $K3$ surfaces of type $M$ is isomorphic to the quotient
$$
{\mathcal M}_{\Lambda}^{0} = (\Omega_{\Lambda}^{+}-{\mathcal D}_{\Lambda})/O^{+}(\Lambda),
$$
where ${\mathcal D}_{\Lambda}$ is the discriminant divisor.
Hence ${\mathcal M}_{\Lambda}^{0}$ is a Zariski open subset of a modular variety of orthogonal type of dimension $20-r$.
\par
In \cite{Yoshikawa04}, a holomorphic torsion invariant of $2$-elementary $K3$ surfaces was defined as follows.
Let $(X,\iota)$ be a $2$-elementary $K3$ surface of type $M$. Write $X^{\iota}$ for the set of fixed points of $\iota$. 
Depending on $M$, $X^{\iota}$ is either empty or the disjoint union of smooth curves of total genus $g=g(M)$. 
(See \S \ref{sect:3.2} for a formula for $g$.)
When $X^{\iota}$ is empty, the corresponding type is unique and is called {\em exceptional}, in which case $(X,\iota)$ is the universal covering
of an Enriques surface endowed with the non-trivial covering transformation.
Take an $\iota$-invariant Ricci-flat K\"ahler form $\gamma$ on $X$ (cf. \cite{Yau78}) and a holomorphic $2$-form $\eta\not=0$ on $X$. 
Let $\tau_{{\bf Z}_{2}}(X,\gamma)(\iota)$ be the equivariant analytic torsion of $(X,\gamma)$ with respect to the $\iota$-action
and let $\tau(X^{\iota},\gamma|_{X^{\iota}})$ be the analytic torsion of $(X^{\iota},\gamma|_{X^{\iota}})$ (cf. \cite{RaySinger73}, \cite{Bismut95}).
Then the real number
$$
\tau_{M}(X,\iota)
=
{\rm Vol}(X,\gamma)^{\frac{14-r}{4}}
\tau_{{\bf Z}_{2}}(X,\gamma)(\iota)
{\rm Vol}(X^{\iota},\gamma|_{X^{\iota}})
\tau(X^{\iota},\gamma|_{X^{\iota}})
$$
depends only on the isomorphism class of $(X,\iota)$, so that it gives rise to a function $\tau_{M}$ on ${\mathcal M}_{\Lambda}^{0}$.
The goal of this paper is to give an explicit expression of $\tau_M$ in terms of modular forms.
It turns out that $\tau_{M}$ is expressed by two types of modular forms: Borcherds products and Siegel modular forms.
Let us explain these modular forms.
\par
Let $\rho_{\Lambda}\colon{\rm Mp}_{2}({\bf Z})\to{\rm GL}({\bf C}[A_{\Lambda}])$ be the Weil representation attached to 
the discriminant group $A_{\Lambda}$ of $\Lambda$, where ${\bf C}[A_{\Lambda}]$ is the group ring of $A_{\Lambda}$. 
By Borcherds \cite{Borcherds98}, 
given an $O^{+}(\Lambda)$-invariant elliptic modular form $f$ of type $\rho_{\Lambda}$ and of weight $1-(20-r)/2$
with integral Fourier expansion, we can take its Borcherds lift $\Psi_{\Lambda}(\cdot,f)$. 
This is a (possibly meromorphic) automorphic form for $O^{+}(\Lambda)$ with infinite product expansion, 
and its Petersson norm $\|\Psi_{\Lambda}(\cdot,f)\|$ descends to a function on ${\mathcal M}_{\Lambda}^{0}$.
To express $\tau_{M}$, the Borcherds lift of the following elliptic modular form will be used.
Let $\eta(\tau)$ be the Dedekind $\eta$-function and let $\theta_{{\Bbb A}_{1}^{+}}(\tau)$ be the theta series of the $A_{1}$-lattice.
We put
$$
\phi_{\Lambda}(\tau)=\eta(\tau)^{-8}\eta(2\tau)^{8}\eta(4\tau)^{-8}\,\theta_{{\Bbb A}_{1}^{+}}(\tau)^{r-10}. 
$$ 
This induces the following modular form of type $\rho_{\Lambda}$ (cf. \cite{Borcherds00}, \cite{Scheithauer06}, \cite{Yoshikawa13}):
$$
F_{\Lambda}
=
\sum_{\gamma\in\widetilde{\Gamma}_{0}(4)\backslash{\rm Mp}_{2}({\bf Z})}
\phi_{\Lambda}|_{\gamma}\,
\rho_{\Lambda}(\gamma^{-1})\,{\bf e}_{0},
$$
where $|_{\gamma}$ is the Petersson slash operator and ${\bf e}_{0}\in{\bf C}[A_{\Lambda}]$ is the vector corresponding to $0\in A_{\Lambda}$.
Except for two types, the Borcherds lift of $(2^{g-1}+\delta_{r,10})F_{\Lambda}$ will be used in the expression of $\tau_{M}$,
where $\delta_{i,j}$ denotes the Kronecker delta.
\par
On the other hand, Siegel modular forms also yield functions on ${\mathcal M}_{\Lambda}^{0}$.
The period map for the fixed curves of $2$-elementary $K3$ surfaces induces a holomorphic map
$$
J_{M}\colon{\mathcal M}_{\Lambda}^{0}\to{\mathcal A}_{g},
$$
where ${\mathcal A}_{g}$ is the Siegel modular variety of degree $g$.
Then the pullback of the Petersson norm of a Siegel modular form by $J_{M}$ is a function on ${\mathcal M}_{\Lambda}^{0}$.
The following Siegel modular forms of weights $2^{g+1}(2^{g}+1)$ and $2(2^{g}-1)(2^{g}+2)$ will be used to express $\tau_{M}$:
$$
\chi_{g}^{8}=\prod_{(a,b)\,{\rm even}}\theta_{a,b}^{8},
\qquad
\Upsilon_{g}=\chi_{g}^{8}\sum_{(a,b)\,{\rm even}}\theta_{a,b}^{-8},
$$
where $\theta_{a,b}$ is the Riemann theta constant with even characteristic $(a, b)$.
Let $\|\chi_{g}\|^{2}$ and $\|\Upsilon_{g}\|^{2}$ be their Petersson norms. 
Hence $J_{M}^{*}\|\chi_{g}^{8}\|$ and $J_{M}^{*}\|\Upsilon_{g}\|$ are functions on ${\mathcal M}_{\Lambda}^{0}$.
For convenience, if $M$ is exceptional, we set $g=1$ and $J_{M}^{*}\|\Upsilon_{g}\|=1$.
\par
The main result of this paper is the following 
(Theorems~\ref{thm:structure:g<6:r<10}, \ref{thm:structure:r=6:delta=0}, \ref{thm:structure:r=10:delta=0}, \ref{thm:formula:Phi:U}, \ref{thm:formula:Phi:U(2)}).

\begin{theorem}
\label{thm:main:theorem:1}
Let $M\subset{\Bbb L}_{K3}$ be a primitive $2$-elementary Lorentzian sublattice of rank $r$ with orthogonal complement $\Lambda$
and let $\delta\in\{0,1\}$ be the parity of its discriminant form.
Then there exists a constant $C_{M}>0$ depending only on $M$ such that 
the following equality of functions on ${\mathcal M}_{\Lambda}^{0}$ holds:
\begin{itemize}
\item[(1)]
If $(r,\delta)\not=(2,0),(10,0)$, then
$$
\tau_{M}^{-2^{g}(2^{g}+1)}
=
C_{M}\,
\left\|
\Psi_{\Lambda}(\cdot,2^{g-1}F_{\Lambda}+f_{\Lambda})
\right\|
\cdot 
J_{M}^{*}
\left\|
\chi_{g}^{8}
\right\|.
$$
\item[(2)]
If $(r,\delta)=(2,0)$ or $(10,0)$, then
$$
\tau_{M}^{-(2^{g}-1)(2^{g}+2)}
=
C_{M}\,
\left\|
\Psi_{\Lambda}(\cdot,2^{g-1}F_{\Lambda}+f_{\Lambda})
\right\|
\cdot 
J_{M}^{*}
\left\|
\Upsilon_{g}
\right\|.
$$
\end{itemize}
Here $f_{\Lambda}$ is the elliptic modular form of type $\rho_{\Lambda}$ given by $f_{\Lambda}=\delta_{r,10}\,F_{\Lambda}$ for $r\not=2$
and by \eqref{eqn:correcting:modular:form:(r,a)=(2,0)}, \eqref{eqn:correcting:modular:form:(r,a)=(2,2)} below for $r=2$.
\end{theorem}

The majority is the case (1), which covers $67$ types. 
The case (2) covers $8$ types.
The formula of (1) does not hold in case (2) because $J_{M}^{*}\chi_{g}$ vanishes identically there. 
In \cite{Yoshikawa04}, \cite{Yoshikawa12}, the automorphy of $\tau_{M}$ was proved (cf.~\eqref{eqn:automorphic:property:sect:0} below), 
but the corresponding modular form was given explicitly only for the exceptional $M$.
In \cite{Yoshikawa13}, the elliptic modular form $F_{\Lambda}$ was introduced and (1) was proved for $r\geq10$.
Theorem~\ref{thm:main:theorem:1} completes this series of work. As a by-product of Theorem~\ref{thm:main:theorem:1}, we prove
the quasi-affinity of ${\mathcal M}_{\Lambda}^{0}$ when $r>6$ (Theorem~\ref{thm:quasiaffinity}).
\par
The invariant $\tau_{M}$ is closely related to the BCOV invariant $\tau_{\rm BCOV}$ of Calabi-Yau threefolds, which was introduced in \cite{FLY08} 
after the prediction of Bershadsky-Cecotti-Ooguri-Vafa \cite{BCOV94} on the mirror symmetry at genus one.
On the moduli space of Calabi-Yau threefolds of Borcea-Voisin associated to $2$-elementary $K3$ surfaces of type $M$ and elliptic curves, 
one has the following equality of functions (cf. \cite{Yoshikawa14})
\begin{equation}
\label{eqn:factorization:BCOV:inv}
\tau_{\rm BCOV}={\mathcal C}_{M}\,\tau_{M}^{-4}\,\|\eta^{24}\|^{2},
\end{equation}
where ${\mathcal C}_{M}$ is a constant depending only on $M$. 
By Theorem~\ref{thm:main:theorem:1} and \eqref{eqn:factorization:BCOV:inv}, 
$\tau_{\rm BCOV}$ for Borcea-Voisin threefolds of type $M$ is given by the product of the Petersson norms of the modular forms
$\Psi_{\Lambda}(\cdot,F_{\Lambda})$, $\chi_{g}^{8}$ (or $\Upsilon_{g}$) and $\eta$. 
Since 
$$
F_{1}=-\log\tau_{\rm BCOV}
$$ 
is the genus one string amplitude $F_{1}$ in B-model (cf. \cite{BCOV94}, \cite{FLY08}),
Theorem~\ref{thm:main:theorem:1} gives an exact result of $F_{1}$ in B-model for all Borcea-Voisin threefolds. 
\par
In Theorem~\ref{thm:main:theorem:1}, the automorphic form corresponding to $\tau_{M}$ splits into two factors.
It is natural to ask if this factorization is realized at the level of holomorphic torsion invariants of $2$-elementary $K3$ surfaces.
Thanks to the spin-$1/2$ bosonization formula \cite{AGMV86}, \cite{BostNelson86}, \cite{Fay92}, we have an affirmative answer
to this question. Let us introduce the following twisted version of $\tau_{M}$
$$
\tau_{M}^{\rm spin}(X,\iota)
=
\prod_{\Sigma^{2}=K_{X^{\iota}},\,h^{0}(\Sigma)=0}
{\rm Vol}(X,\gamma)^{\frac{14-r}{4}}
\tau_{{\bf Z}_{2}}(X,\gamma)(\iota)\,
\tau(X^{\iota},\Sigma;\gamma|_{X^{\iota}})^{-2},
$$
where $\Sigma$ runs over all {\em ineffective} theta characteristics on $X^{\iota}$ and 
$\tau(X^{\iota},\Sigma;\gamma|_{X^{\iota}})$ is the analytic torsion of $\Sigma$ with respect to $\gamma|_{X^{\iota}}$.
It turns out that $\tau_{M}^{\rm spin}(X,\iota)$ is independent of the choice of an $\iota$-invariant Ricci-flat K\"ahler form $\gamma$
and gives rise to a function $\tau_{M}^{\rm spin}$ on ${\mathcal M}_{\Lambda}^{0}$.
Our second result is stated as follows (cf. Theorem~\ref{thm:tau:spin}):

\begin{theorem}
\label{thm:main:theorem:2}
Let ${\mathcal H}_{\Lambda}\subset{\mathcal M}_{\Lambda}$ be the characteristic Heegner divisor (cf. Section~\ref{sect:2}).
Then there exists a constant $C'_{M}>0$ depending only on $M$ such that 
the following equality of functions on ${\mathcal M}_{\Lambda}^{0}\setminus{\mathcal H}_{\Lambda}$ holds:
$$
\tau_{M}^{\rm spin}=C'_{M}\,\|\Psi_{\Lambda}(\cdot,2^{g-1}F_{\Lambda}+f_{\Lambda})\|^{-1/2}.
$$
\end{theorem}

We remark that if ${\mathcal H}_{\Lambda}\not=\emptyset$, then $\tau_{M}^{\rm spin}$ jumps along ${\mathcal H}_{\Lambda}$
and thus $\tau_{M}^{\rm spin}$ is a {\em discontinuous} function on ${\mathcal M}_{\Lambda}^{0}$.
For an explicit relation between the constants in Theorems~\ref{thm:main:theorem:1} and \ref{thm:main:theorem:2},
see Section~\ref{sect:10.2} below. 
After Theorem~\ref{thm:main:theorem:2}, it is very likely that $C'_{M}$ can be determined up to an algebraic number 
by an Arakelov geometric study of $2$-elementary $K3$ surfaces with maximal Picard number. 
This subject will be discussed elsewhere. 
Theorem~\ref{thm:main:theorem:2} can be interpreted as a formula for the equivariant determinant of Laplacian 
(with certain correction term) on the space of invariant Ricci-flat metrics on a $K3$ surface with involution. 
In Section~\ref{sect:11}, we use this interpretation to study an equivariant analogue of Borcherds' conjecture \cite{Borcherds98}.
\par
According to Theorem~\ref{thm:main:theorem:2}, the invariant $\tau_{M}^{\rm spin}$ is elliptic modular in the sense that 
it is the Borcherds lift of an elliptic modular form. When ${\mathcal D}_{\Lambda}\not=\emptyset$, the corresponding modular form 
is determined by $\tau_{M}^{\rm spin}$ in a canonical manner (cf. Theorem~\ref{thm:converse}). In this way, there is a natural one-to-one 
correspondence between the holomorphic torsion invariants $\tau_{M}^{\rm spin}$ and the elliptic modular forms $2^{g-1}F_{\Lambda}+f_{\Lambda}$.
A conceptual account for this unexpected elliptic modularity as well as the geometric meaning of the corresponding modular forms is strongly desired. 
\par
Let us explain the outline of the proof of Theorem~\ref{thm:main:theorem:1}.
For the sake of simplicity, we explain only the major case $\delta=1$.
Since $\tau_{M}$ was determined in \cite{Yoshikawa13} when $g\leq2$, we assume $3\leq g\leq10$.
The strategy is summarized as follows:
\begin{itemize}
\item[(a)]
Reduce Theorem~\ref{thm:main:theorem:1} to determining the divisor of $J_{M}^{*}\chi_{g}^{8}$.
\item[(b)]
Determine the support of the divisor of $J_{M}^{*}\chi_{g}^{8}$ for certain key lattices $M_{g,0}$.
\item[(c)]
Determine inductively from (b) the support of the divisor of $J_{M}^{*}\chi_{g}^{8}$ for all $M$ and prove sharp estimates for its coefficients.
\item[(d)]
Deduce Theorem~\ref{thm:main:theorem:1} from the estimates in step (c).
\end{itemize}
Let us see each step in more detail.
\par
{\bf (a) }  
From the theory of Quillen metrics \cite{BGS88}, \cite{Bismut95}, \cite{MaX00}, the automorphy of $\tau_{M}$ follows (\cite{Yoshikawa04}, \cite{Yoshikawa12}):
There exist $\ell\in{\bf Z}_{>0}$ and an automorphic form $\Phi_{M}$ on $\Omega_{\Lambda}^{+}$ with
\begin{equation}
\label{eqn:automorphic:property:sect:0}
\tau_{M}=\|\Phi_{M}\|^{-1/2\ell},
\qquad
{\rm wt}\,\Phi_{M}=((r-6)\ell,4\ell),
\qquad
{\rm div}\,\Phi_{M}=\ell\,{\mathcal D}_{\Lambda}.
\end{equation} 
(We will eventually see that $\ell$ can be taken to be $2^{g-1}(2^{g}+1)$.) 
By construction, the automorphic form $\Psi_{\Lambda}(\cdot,2^{g-1}\,F_{\Lambda})^{\ell}\otimes J_{M}^{*}\chi_{g}^{8\ell}$ has 
the same weight as $\Phi_{M}^{2^{g-1}(2^{g}+1)}$. 
Hence by the Koecher principle, it is sufficient to show that 
\begin{equation}
\label{eqn:effectivitytoprove:sect:0}
{\rm div}(\Psi_{\Lambda}(\cdot,2^{g-1}\,F_{\Lambda})^{\ell}\otimes J_{M}^{*}\chi_{g}^{8\ell}) 
\geq {\rm div}(\Phi_{M}^{2^{g-1}(2^{g}+1)}). 
\end{equation}
The divisor of $\Psi_{\Lambda}(\cdot,2^{g-1}\,F_{\Lambda})$ is calculated by the theory of Borcherds product \cite{Borcherds98}, 
so the problem is reduced to calculating the divisor of $J_{M}^{*}\chi_{g}^{8}$.  
\par
Since $\delta=1$, the isometry class of $M$ is determined by $g$ and the number $k+1$ of the components of the fixed curve.
(See Section~\ref{sect:3.2} for the formula for the invariants $g=g(M)$ and $k=k(M)$.)
Write $M_{g,k}$ for the lattice with these invariants and set $\Lambda_{g,k}=M_{g,k}^{\perp}$.
Let ${\mathcal D}_{\Lambda_{g,k}}={\mathcal D}_{\Lambda_{g,k}}^{+}+{\mathcal D}_{\Lambda_{g,k}}^{-}$ be the decomposition according to 
the type of $(-2)$-vectors of $\Lambda_{g,k}$, 
and let ${\mathcal H}_{\Lambda_{g,k}}$ be the characteristic Heegner divisor (cf. Section~\ref{sect:2}). 
We will proceed inductively on $k$. 
\par{\bf (b) }
We first study the series $k=0$ by a geometric approach (Section~\ref{sect:6}). 
Curves with vanishing theta constants are characterized by the existence of certain half-canonical bundle. 
By analyzing bi-anticanonical sections of Del Pezzo surfaces with such property, 
we prove that the support of ${\rm div}(J_{M_{g,0}}^{*}\chi_{g}^{8})$ on ${\mathcal M}_{\Lambda_{g,0}}^{0}$ coincides with
${\mathcal H}_{\Lambda_{g,0}}$. 
Hence there exist integers $a_{g}, b_{g}, c_{g}\geq0$ such that 
\begin{equation}
\label{eqn:sect0:2}
{\rm div}(J_{M_{g,0}}^{*}\chi_{g}^{8})
=
a_{g}\,{\mathcal D}_{\Lambda_{g,0}}^{-}
+
b_{g}\,{\mathcal H}_{\Lambda_{g,0}}
+
c_{g}\,{\mathcal D}_{\Lambda_{g,0}}^{+}.
\end{equation}
\par{\bf (c) }
We have a natural inclusion of lattices $\Lambda_{g,k+1}\subset\Lambda_{g,k}$, which induces an inclusion of domains
$i\colon\Omega_{\Lambda_{g,k+1}}\hookrightarrow\Omega_{\Lambda_{g,k}}$.
Then we show that $J_{M_{g,k}}\circ i=J_{M_{g,k+1}}$ outside a locus of codimension $2$
and that $i^{*}{\mathcal H}_{\Lambda_{g,k}}=2{\mathcal H}_{\Lambda_{g,k+1}}$ and similar relation between
${\mathcal D}_{\Lambda_{g,k}}^{\pm}$ and ${\mathcal D}_{\Lambda_{g,k+1}}^{\pm}$ (cf. Sections \ref{sect:2} and \ref{sect:3}).
This enables us to inductively extend \eqref{eqn:sect0:2} to the case $k\geq1$ (Section~\ref{sect:8}): 
\begin{equation*}
\label{eqn:sect0:2:k>0}
{\rm div}(J_{M_{g,k}}^{*}\chi_{g}^{8})
=
a_{g,k}\,{\mathcal D}_{\Lambda_{g,k}}^{-}
+
b_{g,k}\,{\mathcal H}_{\Lambda_{g,k}}
+
c_{g,k}\,{\mathcal D}_{\Lambda_{g,k}}^{+},
\end{equation*}
where $a_{g,k}, b_{g,k}, c_{g,k}$ are integers satisfying  
\begin{equation*}
\label{eqn:sect0:3}
a_{g,k}=a_{g},
\qquad
b_{g,k}=2^{k}b_{g},
\qquad
c_{g,k}=0\,\,(g<10).
\end{equation*}
By using a formula in \cite{Yoshikawa13} and the formula \cite{TeixidoriBigas88} for the thetanull divisor on the moduli space of curves, 
we also prove the estimates $a_{g}\geq2^{2g-1}$ and $b_{g}\geq2^{4}$.
\par{\bf (d) }
Substituting these relations and estimates in the left-hand side of \eqref{eqn:effectivitytoprove:sect:0}, 
we obtain the desired inequality when $g<10$. 
In case $g=10$, an extra argument is required. See Section~\ref{sect:8}.
Note that ${\mathcal H}_{\Lambda}$ vanishes if $r\geq10$, which explains why the proof of Theorem~\ref{thm:main:theorem:1}
is much simpler in the range $r\geq10$, $(r,\delta)\not=(10,0)$ (cf. \cite{Yoshikawa13}).
\par
This paper is organized as follows.
Sections~\ref{sect:1}-\ref{sect:5} are mainly preliminaries.
In Section~\ref{sect:6} (resp.~\ref{sect:7}), we study the even theta characteristics of the fixed curve for $2$-elementary $K3$ surfaces with $\delta=1$
(resp. $\delta=0$).
In Section~\ref{sect:8} (resp.~\ref{sect:9}), we prove Theorem~\ref{thm:main:theorem:1} when $\delta=1$ (resp. $\delta=0$).
In Section~\ref{sect:10}, we introduce the twisted holomorphic torsion invariant $\tau_{M}^{\rm spin}$ and prove Theorem~\ref{thm:main:theorem:2}.
In Section~\ref{sect:11}, we study an equivariant analogue of Borcherds' conjecture.
\par
{\bf Acknowledgements }
The first-named author was partially supported by JSPS KAKENHI Grant Numbers 12809324 and 22224001.
The second-named author is partially supported by JSPS KAKENHI Grant Numbers 23340017, 22244003, 22224001, 2522070.
He is grateful to Professor Jean-Michel Bismut for helpful discussions and to Professor Riccardo Salvati Manni for answering his questions.

\section{Lattices}
\label{sect:1}
\par
A free ${\bf Z}$-module of finite rank endowed with a non-degenerate, integral, symmetric bilinear form is called a lattice.
The rank and signature of a lattice $L$ are denoted by $r(L)$ and ${\rm sign}(L)=(b^{+}(L),b^{-}(L))$, respectively.
For a lattice $L=({\bf Z}^{r},\langle\cdot,\cdot\rangle)$ and an integer $k\in{\bf Z}\setminus\{0\}$, we define $L(k):=({\bf Z}^{r},k\langle\cdot,\cdot\rangle)$.
The group of isometries of $L$ is denoted by $O(L)$.
The set of roots of $L$ is defined as $\Delta_{L}:=\{d\in L;\,\langle d,d\rangle=-2\}$.  
For the root systems $A_{k}$, $D_{k}$, $E_{k}$, the corresponding {\em negative-definite} root lattices are denoted by 
${\Bbb A}_{k}$, ${\Bbb D}_{k}$, ${\Bbb E}_{k}$, respectively. We set ${\Bbb A}_{1}^{+}:={\Bbb A}_{1}(-1)$ etc.
The hyperbolic plane ${\Bbb U}$ is the even unimodular lattice of signature $(1,1)$.
\par
For an even lattice $L$, its dual lattice is denoted by $L^{\lor}$.
The finite abelian group $A_{L}:=L^{\lor}/L$ is called the discriminant group of $L$, which is equipped with
the ${\bf Q}/2{\bf Z}$-valued quadratic form $q_{L}$ called the discriminant form and the ${\bf Q}/{\bf Z}$-valued bilinear form $b_{L}$
called the discriminant bilinear form. The automorphism group of $(A_{L},q_{L})$ is denoted by $O(q_{L})$. 
\par
A lattice $L$ is $2$-elementary if there exists $l\in{\bf Z}_{\geq0}$ with $A_{L}\cong({\bf Z}/2{\bf Z})^{\oplus l}$. 
For a $2$-elementary lattice $L$, we set $l(L):=\dim_{{\bf F}_{2}}A_{L}$.
When $L$ is an even $2$-elementary lattice, the parity of $q_{L}$ is denoted by $\delta(L)\in\{0,1\}$.
By Nikulin \cite[Th.\,3.6.2]{Nikulin80a}, the isometry class of an indefinite even $2$-elementary lattice $L$
is determined by the triplet $({\rm sign}(L),l(L),\delta(L))$. 
For an even $2$-elementary lattice $L$, 
there is a unique element ${\bf 1}_{L}\in A_{L}$, called the {\em characteristic element}, such that 
$b_{L}(\gamma,{\bf 1}_{L})=q_{L}(\gamma)\mod{\bf Z}$ for all $\gamma\in A_{L}$. 
Then $g({\bf 1}_{L})={\bf 1}_{L}$ for all $g\in O(q_{L})$. By definition, ${\bf 1}_{L}=0$ if and only if $\delta(L)=0$.
\par
The $K3$-lattice is defined as the even unimodular lattice of signature $(3,19)$
\begin{center}
${\Bbb L}_{K3}:={\Bbb U}\oplus{\Bbb U}\oplus{\Bbb U}\oplus{\Bbb E}_{8}\oplus{\Bbb E}_{8}$.
\end{center}
It is classical that ${\Bbb L}_{K3}$ is isometric to the second integral cohomology lattice of a $K3$ surface.
For a sublattice $\Lambda\subset{\Bbb L}_{K3}$, we define $\Lambda^{\perp}:=\{l\in{\Bbb L}_{K3};\,\langle l,\Lambda\rangle=0\}$.
\par
A primitive $2$-elementary Lorentzian sublattice of ${\Bbb L}_{K3}$ isometric to ${\Bbb U}(2)\oplus{\Bbb E}_{8}(2)$ is said to be {\em exceptional}. 
(For the reason why this is exceptional, see \S~\ref{sect:3.2} below.)
Its orthogonal complements in ${\Bbb L}_{K3}$, i.e., ${\Bbb U}\oplus{\Bbb U}(2)\oplus{\Bbb E}_{8}(2)$ is also said to be exceptional.

\begin{proposition}
\label{prop:classification:sublattice:sign(2,r-2):K3} 
The isometry classes of primitive $2$-elementary sublattices $\Lambda$ of ${\Bbb L}_{K3}$ with signature $(2,r(\Lambda)-2)$ 
consist of the following $75$ classes in Table~\ref{table:list:sublattice:sign(2,r-2):K3}, where
\begin{center}
$g(\Lambda):=\{r(\Lambda)-l(\Lambda)\}/2$. 
\end{center}
\begin{table}[h]
\caption{Primitive $2$-elementary sublattices of ${\Bbb L}_{K3}$ with $b^{+}=2$}
\label{table:list:sublattice:sign(2,r-2):K3}
\begin{center}
\begin{tabular}{|c|cc|c|} \hline
$g$&$\delta=1$&\,&$\delta=0$
\\ \hline
$0$
&
$({\Bbb A}_{1}^{+})^{\oplus2}\oplus{\Bbb A}_{1}^{\oplus t}$
&
$(0\leq t\leq 9)$
&${\Bbb U}(2)^{\oplus2}$
\\ \hline
$1$
&
${\Bbb U}\oplus{\Bbb A}_{1}^{+}\oplus{\Bbb A}_{1}^{\oplus t}$
&
$(0\leq t\leq 9)$
&
${\Bbb U}\oplus{\Bbb U}(2)$,\,\,
${\Bbb U}(2)^{\oplus2}\oplus{\Bbb D}_{4}$,
\\
\,
&
\,
&
\,
&
${\Bbb U}\oplus{\Bbb U}(2)\oplus{\Bbb E}_{8}(2)$
\\ \hline
$2$
&
${\Bbb U}^{\oplus2}\oplus{\Bbb A}_{1}^{\oplus t}$
&
$(1\leq t\leq 9)$
&
${\Bbb U}^{\oplus2}$,\,\,
${\Bbb U}\oplus{\Bbb U}(2)\oplus{\Bbb D}_{4}$,\,\,
${\Bbb U}^{\oplus2}\oplus{\Bbb E}_{8}(2)$
\\ \hline
$3$
&
${\Bbb U}^{\oplus2}\oplus{\Bbb D}_{4}\oplus{\Bbb A}_{1}^{\oplus t}$
&
$(1\leq t\leq 6)$
&
${\Bbb U}^{\oplus2}\oplus{\Bbb D}_{4}$,\,\,
${\Bbb U}\oplus{\Bbb U}(2)\oplus{\Bbb D}_{4}^{\oplus2}$
\\ \hline
$4$
&
${\Bbb U}^{\oplus2}\oplus{\Bbb D}_{6}\oplus{\Bbb A}_{1}^{\oplus t}$
&
$(0\leq t\leq 5)$
&
${\Bbb U}^{\oplus2}\oplus{\Bbb D}_{4}^{\oplus2}$
\\ \hline
$5$
&
${\Bbb U}^{\oplus2}\oplus{\Bbb E}_{7}\oplus{\Bbb A}_{1}^{\oplus t}$
&
$(0\leq t\leq 5)$
&
${\Bbb U}^{\oplus2}\oplus{\Bbb D}_{8}$
\\ \hline
$6$
&
${\Bbb U}^{\oplus2}\oplus{\Bbb E}_{8}\oplus{\Bbb A}_{1}^{\oplus t}$
&
$(1\leq t\leq 5)$
&
${\Bbb U}^{\oplus2}\oplus{\Bbb E}_{8}$,\,\,
${\Bbb U}^{\oplus2}\oplus{\Bbb D}_{4}\oplus{\Bbb D}_{8}$
\\ \hline
$7$
&
${\Bbb U}^{\oplus2}\oplus{\Bbb D}_{4}\oplus{\Bbb E}_{8}\oplus{\Bbb A}_{1}^{\oplus t}$
&
$(1\leq t\leq 2)$
&
${\Bbb U}^{\oplus2}\oplus{\Bbb D}_{4}\oplus{\Bbb E}_{8}$
\\ \hline
$8$
&
${\Bbb U}^{\oplus2}\oplus{\Bbb D}_{6}\oplus{\Bbb E}_{8}\oplus{\Bbb A}_{1}^{\oplus t}$
&
$(0\leq t\leq 1)$
&\,
\\ \hline
$9$
&
${\Bbb U}^{\oplus2}\oplus{\Bbb E}_{7}\oplus{\Bbb E}_{8}\oplus{\Bbb A}_{1}^{\oplus t}$
&
$(0\leq t\leq 1)$
&
${\Bbb U}^{\oplus2}\oplus{\Bbb D}_{8}\oplus{\Bbb E}_{8}$
\\ \hline
$10$
&
${\Bbb U}^{\oplus2}\oplus{\Bbb E}_{8}^{\oplus2}\oplus{\Bbb A}_{1}$
&\,
&
${\Bbb U}^{\oplus2}\oplus{\Bbb E}_{8}^{\oplus2}$
\\ \hline
\end{tabular}
\end{center}
\end{table}
\end{proposition}

\begin{pf}
See e.g. \cite[p.705 Table 2 and p.706 Table 3]{FinashinKharlamov08}. Notice that the representative of each isometry class
is not necessarily identical to the one in \cite[p.705 Table 2 and p.706 Table 3]{FinashinKharlamov08}.
\end{pf}

\par
For a primitive $2$-elementary Lorentzian sublattice $M\subset{\Bbb L}_{K3}$ and a root $d\in\Delta_{M^{\perp}}$, 
the smallest sublattice of ${\Bbb L}_{K3}$ containing $M$ and $d$ is given by
$$
[M\perp d]:=(M^{\perp}\cap d^{\perp})^{\perp}.
$$ 
Then $[M\perp d]$ is again a primitive $2$-elementary Lorentzian sublattice of ${\Bbb L}_{K3}$ with $[M\perp d]^{\perp}=M^{\perp}\cap d^{\perp}$.

\section{Domains of type IV and modular varieties of orthogonal type}
\label{sect:2}
\par
In Section~\ref{sect:2}, $\Lambda$ is assumed to be an even $2$-elementary lattice with ${\rm sign}(\Lambda)=(2,r(\Lambda)-2)$.
We define the complex manifold $\Omega_{\Lambda}$ with projective $O(\Lambda)$-action by
$$
\Omega_{\Lambda}
:=
\{[x]\in{\bf P}(\Lambda\otimes{\bf C});\,
\langle x,x\rangle=0,\,
\langle x,\bar{x}\rangle>0\}.
$$
Then $\Omega_{\Lambda}$ has two connected components $\Omega_{\Lambda}^{\pm}$, 
each of which is isomorphic to a bounded symmetric domain of type IV of dimension $r(\Lambda)-2$. 
The orthogonal modular variety ${\mathcal M}_{\Lambda}$ associated to $\Lambda$ is defined as the analytic space
$$
{\mathcal M}_{\Lambda}
:=
\Omega_{\Lambda}/O(\Lambda)
=
\Omega_{\Lambda}^{+}/O^{+}(\Lambda),
$$
where 
$$
O^{+}(\Lambda):=\{g\in O(\Lambda);\,g(\Omega_{\Lambda}^{\pm})=\Omega_{\Lambda}^{\pm}\}.
$$
We denote by ${\mathcal M}_{\Lambda}^{*}$ the Baily--Borel--Satake compactification of ${\mathcal M}_{\Lambda}$,
which is an irreducible normal projective variety of dimension $r(\Lambda)-2$ with 
$\dim({\mathcal M}_{\Lambda}^{*}\setminus{\mathcal M}_{\Lambda})\leq1$.

\subsection{Discriminant locus}
\label{sect:2.1}
\par
For $\lambda\in\Lambda$ with $\langle\lambda,\lambda\rangle<0$, we define
$$
H_{\lambda}:=\{[x]\in\Omega_{\Lambda};\,\langle x,\lambda\rangle=0\}.
$$
Then $H_{\lambda}$ is a nonzero divisor on $\Omega_{\Lambda}$.
For any root $d\in\Delta_{\Lambda}$, we have the relation
\begin{equation}
\label{eqn:root:hyperplane}
H_{d}=\Omega_{\Lambda\cap d^{\perp}}.
\end{equation}
The {\em discriminant locus} of $\Omega_{\Lambda}$ is the reduced divisor of $\Omega_{\Lambda}$ defined by 
$$
{\mathcal D}_{\Lambda}:=\sum_{d\in\Delta_{\Lambda}/\pm1}H_{d}.
$$
We define the $O(\Lambda)$-invariant Zariski open subset $\Omega_{\Lambda}^{0}$ of $\Omega_{\Lambda}$ by
$$
\Omega_{\Lambda}^{0}:=\Omega_{\Lambda}\setminus{\mathcal D}_{\Lambda}.
$$
We set
$$
\overline{\mathcal D}_{\Lambda}:={\mathcal D}_{\Lambda}/O(\Lambda),
\qquad
{\mathcal M}_{\Lambda}^{0}:=\Omega_{\Lambda}^{0}/O(\Lambda)={\mathcal M}_{\Lambda}\setminus\overline{\mathcal D}_{\Lambda}.
$$

\subsection{Some subloci of ${\mathcal D}_{\Lambda}$}
\label{sect:2.2}
\par
We define the decomposition $\Delta_{\Lambda}=\Delta^{+}_{\Lambda}\amalg\Delta^{-}_{\Lambda}$ by
$$
\Delta^{+}_{\Lambda}:=\{d\in\Delta_{\Lambda},\,d/2\in\Lambda^{\lor}\},
\qquad
\Delta^{-}_{\Lambda}:=\{d\in\Delta_{\Lambda},\,d/2\not\in\Lambda^{\lor}\}.
$$
Then $\Delta^{\pm}_{\Lambda}$ are $O(\Lambda)$-invariant. 
We define the $O(\Lambda)$-invariant reduced divisors ${\mathcal D}^{\pm}_{\Lambda}$ on $\Omega_{\Lambda}$ 
and the corresponding divisors $\overline{\mathcal D}^{\pm}_{\Lambda}$ on ${\mathcal M}_{\Lambda}$ by
$$
{\mathcal D}^{\pm}_{\Lambda}:=\sum_{d\in\Delta^{\pm}_{\Lambda}/\pm1}H_{d},
\qquad
\overline{\mathcal D}^{\pm}_{\Lambda}:={\mathcal D}^{\pm}_{\Lambda}/O(\Lambda).
$$

\begin{proposition}
\label{prop:quasi:pull:back:discriminant:divisor}
Let $d\in\Delta_{\Lambda}^{+}$. 
Let $i\colon\Omega_{\Lambda\cap d^{\perp}}=H_{d}\hookrightarrow\Omega_{\Lambda}$ be the inclusion. 
Then the following equalities of divisors on $\Omega_{\Lambda\cap d^{\perp}}$ hold:
$$
i^{*}({\mathcal D}_{\Lambda}^{+} - H_{d}) 
=
{\mathcal D}_{\Lambda\cap d^{\perp}}^{+},
\qquad\quad
i^{*}{\mathcal D}_{\Lambda}^{-}
=
{\mathcal D}_{\Lambda\cap d^{\perp}}^{-}.
$$
\end{proposition}

\begin{pf}
Let $\delta\in\Delta_{\Lambda}\setminus\{\pm d\}$. Then $H_{d}\cap H_{\delta}\not=\emptyset$ if and only if $L:={\bf Z}d+{\bf Z}\delta$ is negative-definite.
Since the Gram matrix of $L$ with respect to the basis $\{d,\delta\}$ is given by $-2\binom{1\,a}{a\,1}$ where $a=-\langle\delta,d/2\rangle\in{\bf Z}$,
we conclude that $H_{d}\cap H_{\delta}\not=\emptyset$ if and only if $a=0$, i.e., $\delta\in\Delta_{\Lambda}\cap d^{\perp}=\Delta_{\Lambda\cap d^{\perp}}$.
Since $\Delta_{\Lambda}^{\pm}\cap d^{\perp}=\Delta_{\Lambda\cap d^{\perp}}^{\pm}$, we get
$$
i^{*}({\mathcal D}_{\Lambda}^{+}-H_{d})
=
\sum_{\delta\in(\Delta_{\Lambda}^{+}\setminus\{\pm d\})/\pm1}i^{*}H_{\delta}
=
\sum_{\delta\in\Delta_{\Lambda}^{+}\cap d^{\perp}/\pm1}i^{*}H_{\delta}
=
\sum_{\delta\in\Delta_{\Lambda\cap d^{\perp}}^{+}/\pm1}H_{\delta}
=
{\mathcal D}_{\Lambda\cap d^{\perp}}^{+},
$$
$$
i^{*}{\mathcal D}_{\Lambda}^{-}
=
\sum_{\delta\in\Delta_{\Lambda}^{-}/\pm1}i^{*}H_{\delta}
=
\sum_{\delta\in\Delta_{\Lambda}^{-}\cap d^{\perp}/\pm1}i^{*}H_{\delta}
=
\sum_{\delta\in\Delta_{\Lambda\cap d^{\perp}}^{-}/\pm1}H_{\delta}
=
{\mathcal D}_{\Lambda\cap d^{\perp}}^{-}.
$$
This proves the result.
\end{pf}

For $d\in\Delta_{\Lambda}$, we define a non-empty Zariski open subset $H^{0}_{d}\subset H_{d}$ by
\begin{center}
$H_{d}^{0}:=H_{d}\setminus\bigcup_{\delta\in\Delta_{\Lambda}\setminus\{\pm d\}}H_{\delta}$.
\end{center}
We set
$$
{\mathcal D}^{0,\pm}_{\Lambda}:=\sum_{d\in\Delta^{\pm}_{\Lambda}/\pm1}H^{0}_{d},
\qquad
{\mathcal D}_{\Lambda}^{0}
:=
{\mathcal D}^{0,+}_{\Lambda}+{\mathcal D}^{0,-}_{\Lambda}
=
\sum_{d\in\Delta_{\Lambda}/\pm1}H_{d}^{0}.
$$
Then $\Omega_{\Lambda}^{0}\amalg{\mathcal D}_{\Lambda}^{0}$ is a Zariski open subset of $\Omega_{\Lambda}$, 
whose complement has codimension $\geq2$ when $r(\Lambda)\geq4$ and is empty when $r(\Lambda)\leq3$.

\subsection{Characteristic Heegner divisor}
\label{sect:2.3}
\par
Set 
$$
\varepsilon_{\Lambda}:=\{12-r(\Lambda)\}/2.
$$ 
We define the {\em characteristic Heegner divisor} of $\Omega_{\Lambda}$ as the reduced divisor
$$
{\mathcal H}_{\Lambda}
=
{\mathcal H}_{\Lambda}(\varepsilon_{\Lambda},{\bf 1}_{\Lambda})
:=
\sum_{\lambda\in\Lambda^{\lor}/\pm1,\,\lambda^{2}=\varepsilon_{\Lambda},\,[\lambda]={\bf 1}_{\Lambda}}H_{\lambda},
$$
where $[\lambda]:=\lambda+\Lambda\in A_{\Lambda}$.
Since ${\bf 1}_{\Lambda}$ is $O(q_{\Lambda})$-invariant, ${\mathcal H}_{\Lambda}$ is $O(\Lambda)$-invariant.
Since $\varepsilon_{\Lambda}\geq0$ when $r(\Lambda)\leq12$, we get
${\mathcal H}_{\Lambda}=0$ if $r(\Lambda)\leq12$.

\begin{proposition}
\label{prop:quasi:pull:back:characteristic:divisor}
Let $d\in\Delta_{\Lambda}^{+}$. Let $i\colon\Omega_{\Lambda\cap d^{\perp}}=H_{d}\hookrightarrow\Omega_{\Lambda}$ be the inclusion. 
Then the following equality of divisors on $\Omega_{\Lambda\cap d^{\perp}}$ holds
$$
i^{*}{\mathcal H}_{\Lambda}=2\,{\mathcal H}_{\Lambda\cap d^{\perp}}.
$$
\end{proposition}

\begin{pf}
Since $r(\Lambda)\leq21$, we get $\varepsilon_{\Lambda}\geq-\frac{9}{2}$.  Set $\Lambda':=\Lambda\cap d^{\perp}$.
Since $d\in\Delta_{\Lambda}^{+}$, we deduce from  \cite[Prop.\,3.1]{FinashinKharlamov08}
that $\Lambda$ and $\Lambda'\oplus{\bf Z}d$ have the same invariants $(r,l,\delta)$.
Hence we get the orthogonal decomposition $\Lambda=\Lambda'\oplus{\bf Z}d$.
Let $\lambda\in\Lambda^{\lor}$ be such that $\lambda^{2}=\varepsilon_{\Lambda}$ and $[\lambda]={\bf 1}_{\Lambda}$.
Then we can write $\lambda=\lambda'+a(d/2)$, where $\lambda'\in(\Lambda')^{\lor}$ and $a=-\langle\lambda,d\rangle\in{\bf Z}$.
Since $[\lambda]={\bf 1}_{\Lambda}$, we get $a\equiv1\mod2$. Hence $a=2k+1$ for some $k\in{\bf Z}$ and $\lambda'\in{\bf 1}_{\Lambda'}$.
Since $H_{\lambda'}\not=\emptyset$ if and only if $(\lambda')^{2}<0$, we get $i^{*}H_{\lambda}=H_{\lambda'}\not=\emptyset$ if and only if
$0>(\lambda')^{2}=\lambda^{2}+a^{2}/2=(\varepsilon_{\Lambda}+\frac{1}{2})+2k(k+1)$.
Since $\varepsilon_{\Lambda}\geq-\frac{9}{2}$ and hence $-2\leq\frac{-1-\sqrt{-2\varepsilon_{\Lambda}}}{2}<\frac{-1+\sqrt{-2\varepsilon_{\Lambda}}}{2}\leq1$, 
we see that $(\varepsilon_{\Lambda}+\frac{1}{2})+2k(k+1)<0$ if and only if $k=0,-1$.
This proves that 
\begin{equation}
\label{eqn:fiber:projection}
i^{*}H_{\lambda}=H_{\lambda'}\not=\emptyset
\qquad\Longleftrightarrow\qquad
\lambda=\lambda'\pm(d/2).
\end{equation}
When $\lambda=\lambda'\pm(d/2)$, we get $(\lambda')^{2}=\varepsilon_{\Lambda}+\frac{1}{2}=\varepsilon_{\Lambda'}$. 
This, together with \eqref{eqn:fiber:projection}, yields that
$$
\begin{aligned}
i^{*}{\mathcal H}_{\Lambda}
&=
\sum_{\lambda\in\Lambda^{\lor}/\pm1,\,\lambda^{2}=\varepsilon_{\Lambda},\,[\lambda]={\bf 1}_{\Lambda}}i^{*}H_{\lambda}
=
\sum_{\lambda'\in(\Lambda')^{\lor}/\pm1,\,(\lambda')^{2}=\varepsilon_{\Lambda'},\,[\lambda']={\bf 1}_{\Lambda'}}i^{*}H_{\lambda'\pm(d/2)}
\\
&=
2\sum_{\lambda'\in(\Lambda')^{\lor}/\pm1,\,(\lambda')^{2}=\varepsilon_{\Lambda'},\,[\lambda']={\bf 1}_{\Lambda'}}H_{\lambda'}
=
2\,{\mathcal H}_{\Lambda'}.
\end{aligned}
$$
This proves the proposition.
\end{pf}

\section{$2$-elementary $K3$ surfaces and the Torelli map}
\label{sect:3}
\par

\subsection
{$2$-elementary $K3$ surfaces}
\label{sect:3.1}
\par
A $K3$ surface $X$ equipped with a holomorphic involution $\iota\colon X\to X$ is called a {\it $2$-elementary $K3$ surface} if $\iota$ is anti-symplectic:
\begin{equation}
\label{eqn:condition:2-elementary:K3}
\iota^{*}|_{H^{0}(K_{X})}=-1.
\end{equation}
The possible deformation types of $2$-elementary $K3$ surfaces were determined by Nikulin. 
(See \cite[Sect.\,2.3]{AlexeevNikulin06} and the references therein.) 
Let $\alpha\colon H^{2}(X,{\bf Z})\cong{\Bbb L}_{K3}$ be an isometry of lattices.
Set $H^{2}(X,{\bf Z})_{\pm}:=\{l\in H^{2}(X,{\bf Z});\,\iota^{*}l=\pm l\}$ and
\begin{equation}
\label{eqn:type:2-elementary:K3}
M:=\alpha(H^{2}(X,{\bf Z})_{+}),
\qquad
\Lambda:=M^{\perp}=\alpha(H^{2}(X,{\bf Z})_{-}).
\end{equation}
Then $M\subset{\Bbb L}_{K3}$ must be a primitive $2$-elementary Lorentzian sublattice. 
Conversely, for any primitive $2$-elementary Lorentzian sublattice $M\subset{\Bbb L}_{K3}$,
there exists a $2$-elementary $K3$ surface with \eqref{eqn:type:2-elementary:K3}.
For a $2$-elementary $K3$ surface $(X,\iota)$, the isometry class of $H^{2}(X,{\bf Z})_{+}$ is called the {\em type} of  $(X,\iota)$.
By an abuse of notation, the sublattice itself $\alpha(H^{2}(X,{\bf Z})_{+})\subset{\Bbb L}_{K3}$ is also called the type of $(X,\iota)$.
Then there is a one-to-one correspondence between the deformation types of $2$-elementary $K3$ surfaces and the triplets $(r,l,\delta)$.
Since the latter consists of $75$ points, there exist mutually distinct $75$ deformation types of $2$-elementary $K3$ surfaces.
For a given primitive $2$-elementary Lorentzian sublattice $M\subset{\Bbb L}_{K3}$, the moduli space of $2$-elementary $K3$ surfaces of type $M$
is given as follows.
\par
Let $(X,\iota)$ be a $2$-elementary $K3$ surface of type $M$ and let $\alpha\colon H^{2}(X,{\bf Z})\cong{\Bbb L}_{K3}$ be an isometry
satisfying \eqref{eqn:type:2-elementary:K3}. 
Since $H^{2,0}(X,{\bf C})\subset H^{2}(X,{\bf Z})_{-}\otimes{\bf C}$ by \eqref{eqn:condition:2-elementary:K3}, we get
$$
\pi_{M}(X,\iota,\alpha):=[\alpha(H^{2,0}(X,{\bf C}))]\in\Omega_{\Lambda}^{0}.
$$
Its $O(\Lambda)$-orbit is called the Griffiths period of $(X,\iota)$ and is denoted by
$$
\overline{\pi}_{M}(X,\iota):=O(\Lambda)\cdot\pi_{M}(X,\iota,\alpha)\in{\mathcal M}_{\Lambda}^{0}.
$$
By \cite[Th.\,1.8]{Yoshikawa04}, \cite[Prop.\,11.2]{Yoshikawa13}, 
the coarse moduli space of $2$-elementary $K3$ surfaces of type $M$ is isomorphic to ${\mathcal M}_{\Lambda}^{0}$ via the period map $\overline{\pi}_{M}$.
In the rest of this paper, we identify the point $\overline{\pi}_{M}(X,\iota)\in {\mathcal M}_{\Lambda}^{0}$ 
with the isomorphism class of $(X,\iota)$.

\subsection
{The Torelli map for $2$-elementary $K3$ surfaces}
\label{sect:3.2}
\par

\subsubsection{The set of fixed points}
\label{sect:3.2.1}
For a $2$-elementary $K3$ surface $(X,\iota)$ of type $M$, set 
$$
X^{\iota}:=\{x\in X;\,\iota(x)=x\}.
$$
Then $X^{\iota}=\emptyset$ if and only if $M$ is exceptional, i.e., $M\cong{\Bbb U}(2)\oplus{\Bbb E}_{8}(2)$.
In this case, the quotient $X/\iota$ is an Enriques surface and $(X,\iota)$ is the universal covering of an Enriques surface
endowed with the non-trivial covering transformation.
When $M$ is non-exceptional, by Nikulin \cite[Th.\,4.2.2]{Nikulin83}, we have
\begin{equation}
\label{eqn:fixed:points}
X^{\iota}=C^{(g(M))}\amalg E_{1}\amalg\cdots\amalg E_{k(M)}
\end{equation}
for $M\not\cong{\Bbb U}(N)\oplus{\Bbb E}_{8}(2)$ $(N=1,2)$ and we have
$X^{\iota}=C_{1}^{(1)}\amalg C_{2}^{(1)}$ for $M\cong{\Bbb U}\oplus{\Bbb E}_{8}(2)$,
where $C^{(g)}$ is a projective curve of genus $g$  and $E_{i}\cong{\bf P}^{1}$ and
$$
g(M):=\{22-r(M)-l(M)\}/2,
\qquad
k(M):=\{r(M)-l(M)\}/2.
$$
Since $r(\Lambda)=22-r(M)$ and $l(\Lambda)=l(M)$, we have the following relations
$$
g(M)=\{r(\Lambda)-l(\Lambda)\}/2,
\qquad
k(M)=\{22-r(\Lambda)-l(\Lambda)\}/2.
$$
As we defined $g(\Lambda)=(r(\Lambda)-l(\Lambda))/2$ in Proposition~\ref{prop:classification:sublattice:sign(2,r-2):K3},
we have $g(M)=g(\Lambda)$. Notice that, when $M\cong{\Bbb U}(2)\oplus{\Bbb E}_{8}(2)$ and hence $X^{\iota}$ is empty, 
$g(M)=g(\Lambda)=1$ has {\em no} geometric meaning.

\subsubsection{The Torelli map}
\label{sect:3.2.2}
For $g\geq0$,  let ${\frak S}_{g}$ be the Siegel upper half-space of degree $g$ and 
let ${\rm Sp}_{2g}({\bf Z})$ be the symplectic group of degree $2g$ over ${\bf Z}$. We define
$$
{\mathcal A}_{g}:={\frak S}_{g}/{\rm Sp}_{2g}({\bf Z}).
$$
The Satake compactification of ${\mathcal A}_{g}$ is denoted by ${\mathcal A}_{g}^{*}$. 
\par
For a $2$-elementary $K3$ surface $(X,\iota)$ of type $M$, the period of $X^{\iota}$ is denoted by $\varOmega(X^{\iota})\in{\mathcal A}_{g(M)}$.
We define a map $\overline{J}_{M}^{0}\colon{\mathcal M}_{\Lambda}^{0}\to{\mathcal A}_{g(M)}$ by
$$
\overline{J}_{M}^{0}(X,\iota)
=
\overline{J}_{M}^{0}(\overline{\pi}_{M}(X,\iota))
:=
\varOmega(X^{\iota}).
$$
Let $\varPi_{\Lambda}\colon\Omega_{\Lambda}\to{\mathcal M}_{\Lambda}$ be the projection. 
The {\em Torelli map} is the $O(\Lambda)$-equivariant holomorphic map $J_{M}^{0}\colon\Omega_{\Lambda}^{0}\to{\mathcal A}_{g(M)}$ defined by
$$
J_{M}^{0}:=\overline{J}_{M}^{0}\circ\varPi_{\Lambda}|_{\Omega_{\Lambda}^{0}}.
$$
In Theorem~\ref{thm:extension:Torelli:map} below, we will extend $J_{M}^{0}$ to a certain Zariski open subset of $\Omega_{\Lambda}$ containing
$\Omega_{\Lambda}^{0}\cup{\mathcal D}_{\Lambda}^{0}$ and prove its compatibility 
with respect to the inclusion $\Omega_{\Lambda\cap d^{\perp}}=H_{d}\hookrightarrow\Omega_{\Lambda}$ for $d\in\Delta_{\Lambda}$.
For this, we introduce a stratification on $\Omega_{\Lambda}$.

\subsubsection{A stratification of $\Omega_{\Lambda}$}
\label{sect:3.2.3}
\par
For a primitive sublattice $L\subset\Lambda$ generated by $\Delta_{L}$, we define 
\begin{center}
$H_{L}:=\bigcap_{d\in\Delta_{L}}H_{d}$,
\qquad
$H_{L}^{0}:=H_{L}\setminus\bigcup_{\delta\in\Delta_{\Lambda}\setminus\Delta_{L}}H_{\delta}$.
\end{center}
Then $H_{L}\not=\emptyset$ if and only if $L$ is negative-definite, i.e., {\em $L$ is a root lattice}. 
If $H_{L}\not=\emptyset$, then $H_{L}^{0}$ is a non-empty dense Zariski open subset of $H_{L}$. 
By definition, it is obvious that if $r(L)=r(L')$ and $L\not=L'$, then $H_{L}^{0}\cap H_{L'}^{0}=\emptyset$.
Set
\begin{center}
$\Omega_{\Lambda}^{k}:=\amalg_{L\subset\Lambda,\,r(L)=k}H_{L}^{0}$,
\qquad
$\Omega_{\Lambda}^{\geq k}:=\amalg_{l\geq k}\Omega_{\Lambda}^{l}=\bigcup_{L\subset\Lambda,\,r(L)=k}H_{L}$,
\end{center}
where $L$ runs over the set of all primitive root sublattices of $\Lambda$ of rank $k$.
Then $\Omega_{\Lambda}^{\geq k}$ is a Zariski closed subset of $\Omega_{\Lambda}$ of pure dimension $k$,
and $\Omega_{\Lambda}\setminus\Omega_{\Lambda}^{\geq k+1}$ is a dense Zariski open subset of $\Omega_{\Lambda}$, 
whose complement has codimension $k+1$. We have
\begin{center}
${\mathcal D}_{\Lambda}=\Omega_{\Lambda}^{\geq 1}=\sum_{d\in\Delta_{\Lambda}/\pm1}H_{d}$,
\qquad
$\Omega_{\Lambda}^{0}=\Omega_{\Lambda}\setminus{\mathcal D}_{\Lambda}$.
\end{center}
\par
For a root lattice ${\Bbb K}$, let ${\Bbb K}(\Lambda)$ be the set of primitive sublattices of $\Lambda$ isometric to ${\Bbb K}$. 
Since a root lattice of rank $2$ is either ${\Bbb A}_{1}^{\oplus2}$ or ${\Bbb A}_{2}$, we have
\begin{center}
$\Omega_{\Lambda}^{\geq2}=(\bigcup_{L\in{\Bbb A}_{1}^{\oplus2}(\Lambda)}H_{L})\cup(\bigcup_{L\in{\Bbb A}_{2}(\Lambda)}H_{L})$.
\end{center}
We set
\begin{center}
${\mathcal D}_{\Lambda}^{1,+}:=\bigcup_{L\in{\Bbb A}_{1}^{\oplus2}(\Lambda)}H_{L}^{0}$,
\qquad
${\mathcal D}_{\Lambda}^{1,-}:=\bigcup_{L\in{\Bbb A}_{2}(\Lambda)}H_{L}^{0}$.
\end{center}
Then ${\mathcal D}_{\Lambda}^{1,+}\amalg{\mathcal D}_{\Lambda}^{1,-}=\Omega_{\Lambda}^{2}$ and
we have the stratification
\begin{equation}
\label{eqn:stratification}
\Omega_{\Lambda}\setminus\Omega_{\Lambda}^{\geq3}
=
\Omega_{\Lambda}^{0}\amalg{\mathcal D}_{\Lambda}^{0}\amalg{\mathcal D}_{\Lambda}^{1,+}\amalg{\mathcal D}_{\Lambda}^{1,-}
\end{equation}
such that 
${\mathcal D}_{\Lambda}\setminus\Omega_{\Lambda}^{\geq3}={\mathcal D}_{\Lambda}^{0}\amalg{\mathcal D}_{\Lambda}^{1,+}\amalg{\mathcal D}_{\Lambda}^{1,-}$
is a divisor of $\Omega_{\Lambda}\setminus\Omega_{\Lambda}^{\geq3}$.

\begin{lemma}
\label{lemma:stratification:discriminant:locus}
If $d\in\Delta_{\Lambda}^{+}$, then 
$H_{d}\setminus\Omega_{\Lambda}^{\geq3}=\Omega_{\Lambda\cap d^{\perp}}^{0}\amalg{\mathcal D}_{\Lambda\cap d^{\perp}}^{0}$.
\end{lemma}

\begin{pf}
{\em (Step 1) }
By the definition of ${\mathcal D}_{\Lambda}^{0}$, we have $H_{d}^{0}=H_{d}\cap{\mathcal D}_{\Lambda}^{0}$.
Since $d\in\Delta_{\Lambda}^{+}$, we have $H_{d}\cap H_{\delta}\not=\emptyset$ for $\delta\in\Delta_{\Lambda}\setminus\{\pm d\}$ 
if and only if $\delta\in\Delta_{\Lambda\cap d^{\perp}}$.
Hence $H_{d}\cap{\mathcal D}_{\Lambda}^{0}=H_{d}^{0}=H_{d}\setminus\bigcup_{\delta\in\Delta_{\Lambda}\setminus\{\pm d\}}H_{\delta}
=H_{d}\setminus\bigcup_{\delta\in\Delta_{\Lambda\cap d^{\perp}}}H_{\delta}=\Omega_{\Lambda\cap d^{\perp}}^{0}$.
\par{\em (Step 2) }
Assume $H_{d}\cap{\mathcal D}_{\Lambda}^{1,-}\not=\emptyset$. Then there exist $L\in{\Bbb A}_{2}(\Lambda)$ and $[\eta]\in H_{d}\cap H_{L}^{0}$.
By the definition of $H_{L}^{0}$, we have $d\in\Delta_{L}$. 
Since $L\cong{\Bbb A}_{2}$, there exists $\delta\in\Delta_{L}$ with $\langle d,\delta\rangle=1$.
Since $d\in\Delta_{\Lambda}^{+}$, this yields the contradiction $\langle d,\delta\rangle=2\langle(d/2),\delta\rangle\in2{\bf Z}$.
This proves $H_{d}\cap{\mathcal D}_{\Lambda}^{1,-}=\emptyset$.
\par{\em (Step 3) }
Assume $H_{d}\cap{\mathcal D}_{\Lambda}^{1,+}\not=\emptyset$. Then there exists $L\in{\Bbb A}_{1}^{\oplus2}(\Lambda)$ with $[\eta]\in H_{d}\cap H_{L}^{0}$.
By the definition of $H_{L}^{0}$, we get $d\in\Delta_{L}$. 
Hence $\Delta_{L}=\{\pm d,\pm\delta\}$ for some $\delta\in\Delta_{\Lambda\cap d^{\perp}}$. Thus
$$
\begin{aligned}
H_{d}\cap{\mathcal D}_{\Lambda}^{1,+}
&=
\bigcup_{L\in{\Bbb A}_{1}^{\oplus2}(\Lambda),\,d\in L}H_{d}\cap H_{L}^{0}
=
\bigcup_{\delta\in\Delta_{\Lambda\cap d^{\perp}}}
\{
(H_{d}\cap H_{\delta})\setminus\bigcup_{\epsilon\in\Delta_{\Lambda}\setminus\{\pm d,\pm\delta\}}H_{\epsilon}
\}
\\
&=
\bigcup_{\delta\in\Delta_{\Lambda\cap d^{\perp}}}
\{
(H_{d}\cap H_{\delta})\setminus\bigcup_{\epsilon\in\Delta_{\Lambda\cap d^{\perp}}\setminus\{\pm\delta\}}H_{\epsilon}
\}
=
{\mathcal D}_{\Lambda\cap d^{\perp}}^{0}.
\end{aligned}
$$
Since $d\in\Delta_{\Lambda}^{+}$, the third equality follows from the fact that $H_{d}\cap H_{\epsilon}\not=\emptyset$ 
for $\epsilon\in\Delta_{\Lambda}\setminus\{\pm d\}$ if and only if $\epsilon\in\Delta_{\Lambda\cap d^{\perp}}$.
This proves $H_{d}\cap{\mathcal D}_{\Lambda}^{1,+}={\mathcal D}_{\Lambda\cap d^{\perp}}^{0}$.
\par{\em (Step 4) }
Since $H_{d}\cap\Omega_{\Lambda}^{0}=\emptyset$, the result follows from (Steps 1-3) and \eqref{eqn:stratification}.
\end{pf}

\subsubsection{The local structure of $\Omega_{\Lambda}^{0}$ near ${\mathcal D}_{\Lambda}^{1}$}
\label{sect:3.2.4}
\par
Let $\varDelta\subset{\bf C}$ be the unit disc and set $\varDelta^{*}:=\varDelta\setminus\{0\}$.
Let $L_{0}\subset\varDelta^{2}$ be the diagonal locus and set $L_{1}:=\{0\}\times\varDelta$ and $L_{2}:=\varDelta\times\{0\}$.
Then $L_{0}$, $L_{1}$, $L_{2}$ are lines with $L_{i}\cap L_{j}=\{0\}$ for any $i\not=j$. We have $(\varDelta^{*})^{2}=\varDelta^{2}\setminus(L_{1}\cup L_{2})$.
Set $L_{i}^{*}:=L_{i}\setminus\{(0,0)\}$.

\begin{lemma}
\label{lemma:local:structure:discr:locus}
Let $[\eta]\in{\mathcal D}_{\Lambda}^{1,+}\amalg{\mathcal D}_{\Lambda}^{1,-}$ and set $n:=\dim\Omega_{\Lambda}$. Then the following hold.
\begin{itemize}
\item[(1)]
If $[\eta]\in{\mathcal D}_{\Lambda}^{1,+}$, then there is a neighborhood $U$ of $[\eta]$ in $\Omega_{\Lambda}\setminus\Omega_{\Lambda}^{\geq3}$
such that 
$$
U\cap\Omega_{\Lambda}^{0}\cong(\varDelta^{*})^{2}\times\varDelta^{n-2},
\quad
U\cap{\mathcal D}_{\Lambda}^{0}\cong(L_{1}^{*}\amalg L_{2}^{*})\times\varDelta^{n-2},
\quad
U\cap{\mathcal D}_{\Lambda}^{1,+}\cong\{(0,0)\}\times\varDelta^{n-2}.
$$
\item[(2)]
If $[\eta]\in{\mathcal D}_{\Lambda}^{1,-}$, then there is a neighborhood $U$ of $[\eta]$ in $\Omega_{\Lambda}\setminus\Omega_{\Lambda}^{\geq3}$
such that $U\setminus{\mathcal D}_{\Lambda}\cong(\varDelta^{2}\setminus L_{0}\cup L_{1}\cup L_{2})\times\varDelta^{n-2}$.
\end{itemize}
\end{lemma}

\begin{pf}
Since $[\eta]\in{\mathcal D}_{\Lambda}^{1,+}\amalg{\mathcal D}_{\Lambda}^{1,-}$,
there exist $L\in{\Bbb A}_{1}^{\oplus2}(\Lambda)\amalg{\Bbb A}_{2}(\Lambda)$ and a neighborhood $U$ of $[\eta]$
in $\Omega_{\Lambda}\setminus\Omega_{\Lambda}^{\geq3}$ such that $U\cap{\mathcal D}_{\Lambda}=U\cap\bigcup_{d\in\Delta_{L}}H_{d}$.
\par{\bf (1) }
Assume $[\eta]\in{\mathcal D}_{\Lambda}^{1,+}$. Then $L\in{\Bbb A}_{1}^{\oplus2}(\Lambda)$.
There exists $d_{1},d_{2}\in\Delta_{L}$ with $\langle d_{1},d_{2}\rangle=0$ such that
$U\cap{\mathcal D}_{\Lambda}=U\cap(H_{d_{1}}\cup H_{d_{2}})$.
Replacing $U$ by a smaller neighborhood if necessary,
there exist a system of coordinates $(z_{1},z_{2},w)$, $w=(w_{1},\ldots,w_{n-2})$, on $U$ such that 
$H_{d_{1}}={\rm div}(z_{1})$, $H_{d_{2}}={\rm div}(z_{2})$.
The isomorphism $\psi\colon(U,[\eta])\cong(\varDelta^{n},0)$ induced by $(z_{1},z_{2},w)$ has the desired property.
\par{\bf (2) }
Assume $[\eta]\in{\mathcal D}_{\Lambda}^{1,-}$. Since $L\in{\Bbb A}_{2}(\Lambda)$,
there exist $d_{0},d_{1},d_{2}\in\Delta_{L}$ with $d_{0}=d_{1}+d_{2}$ and $\langle d_{1},d_{2}\rangle=1$ such that
$U\cap{\mathcal D}_{\Lambda}=U\cap(H_{d_{0}}\cup H_{d_{1}}\cup H_{d_{2}})$.
Replacing $U$ by a smaller neighborhood if necessary,
there exist a system of coordinates $(z_{1},z_{2},w)$ on $U$ such that 
$H_{d_{1}}={\rm div}(z_{1})$, $H_{d_{2}}={\rm div}(z_{2})$, $H_{d_{0}}={\rm div}(z_{1}+z_{2})$.
The isomorphism $\psi\colon(U,[\eta])\cong(\varDelta^{n},0)$ induced by $(z_{1},z_{2},w)$ has the desired property.
\end{pf}

\subsubsection{Inclusion of lattices and the Torelli map}
\label{sect:3.2.5}
\par
Recall that for $d\in\Delta_{\Lambda}^{+}$, it follows from Lemma~\ref{lemma:stratification:discriminant:locus} the equality of sets
$H_{d}\setminus\Omega_{\Lambda}^{\geq3}=\Omega_{\Lambda\cap d^{\perp}}^{0}\cup{\mathcal D}_{\Lambda\cap d^{\perp}}^{0}$.
By Lemma~\ref{lemma:local:structure:discr:locus} (1), 
$\Omega_{\Lambda}^{0}\cup{\mathcal D}_{\Lambda}^{0}\cup{\mathcal D}_{\Lambda}^{1,+}$ is a Zariski open subset of
$\Omega_{\Lambda}\setminus\Omega_{\Lambda}^{\geq3}$.

\begin{theorem}
\label{thm:extension:Torelli:map}
$J_{M}^{0}$ extends to a holomorphic map from 
$\Omega_{\Lambda}^{0}\cup{\mathcal D}_{\Lambda}^{0}\cup{\mathcal D}_{\Lambda}^{1,+}$ to ${\mathcal A}_{g}^{*}$.
\end{theorem}

\begin{pf}
Set $n:=\dim\Omega_{\Lambda}$.
Let $[\eta]\in{\mathcal D}_{\Lambda}^{0}\cup{\mathcal D}_{\Lambda}^{1,+}$. By Lemma~\ref{lemma:local:structure:discr:locus} (1),
there is a neighborhood $U$ of $[\eta]$ in $\Omega_{\Lambda}$ such that either
$U\setminus({\mathcal D}_{\Lambda}^{0}\cup{\mathcal D}_{\Lambda}^{1,+})\cong\varDelta^{*}\times\varDelta^{n-1}$
or
$U\setminus({\mathcal D}_{\Lambda}^{0}\cup{\mathcal D}_{\Lambda}^{1,+})\cong(\varDelta^{*})^{2}\times\varDelta^{n-2}$.
By Borel \cite{Borel72}, $J_{M}^{0}$ extends to a holomorphic map from $U$ to ${\mathcal A}_{g}^{*}$.
Since $[\eta]\in{\mathcal D}_{\Lambda}^{0}\cup{\mathcal D}_{\Lambda}^{1,+}$ is an arbitrary point, we get the result.
\end{pf}

\begin{remark}
By Lemma~\ref{lemma:local:structure:discr:locus} (2), Borel's extension theorem does not apply to $J_{M}^{0}$ near ${\mathcal D}_{\Lambda}^{1,-}$.
This explains why $J_{M}$ does not extend to $\Omega_{\Lambda}\setminus\Omega_{\Lambda}^{\geq3}$ in general.
\end{remark}

Denote the extension of $J_{M}^{0}$ by
\begin{center}
$J_{M}\colon\Omega_{\Lambda}^{0}\cup{\mathcal D}_{\Lambda}^{0}\cup{\mathcal D}_{\Lambda}^{1,+}\to{\mathcal A}_{g}^{*}$
\end{center}
and call it again the Torelli map. 
By \cite[Th.\,2.5]{Yoshikawa13}, the following equality holds
$$
J_{M}|_{H_{d}^{0}}=J_{[M\perp d]}|_{\Omega_{\Lambda\cap d^{\perp}}^{0}}
$$
for all $d\in\Delta_{\Lambda}$. The following refinement is crucial for the proof Theorem~\ref{thm:main:theorem:1}.

\begin{theorem}
\label{thm:functoriality:Torelli:map}
If $d\in\Delta_{\Lambda}^{+}$, then
\begin{equation}
\label{eqn:functoriality:Torelli:map}
J_{M}|_{H_{d}\setminus\Omega_{\Lambda}^{\geq3}}
=
J_{[M\perp d]}|_{\Omega_{\Lambda\cap d^{\perp}}^{0}\cup{\mathcal D}_{\Lambda\cap d^{\perp}}^{0}}.
\end{equation}
\end{theorem}

\begin{pf}
Since both $J_{M}|_{H_{d}\setminus\Omega_{\Lambda}^{\geq3}}$ and 
$J_{[M\perp d]}|_{\Omega_{\Lambda\cap d^{\perp}}^{0}\cup{\mathcal D}_{\Lambda\cap d^{\perp}}^{0}}$
are holomorphic maps from $\Omega_{\Lambda\cap d^{\perp}}^{0}\cup{\mathcal D}_{\Lambda\cap d^{\perp}}^{0}$ to ${\mathcal A}_{g}^{*}$
by Lemma~\ref{lemma:stratification:discriminant:locus} and Theorem~\ref{thm:extension:Torelli:map},
it suffices to prove the equality on $\Omega_{\Lambda\cap d^{\perp}}^{0}$. 
Since this was proved in \cite[Th.\,2.5]{Yoshikawa13}, we get the result.
\end{pf}

\subsection{Hyperelliptic linear system}\label{sect:3.3}

This subsection is the technical basis for Sections \ref{sect:6} and \ref{sect:7}. 
We advise the reader to skip for the moment and return when necessary. 
We prepare some tools to realize a given 2-elementary $K3$ surface 
as a double cover of $\mathbf{P}^2$ or a Hirzebruch surface. 

We will use the following notation: 
For $n\geq0$ let 
$$
\mathbf{F}_{n}=\mathbf{P}(\mathcal{O}_{\mathbf{P}^1}\oplus\mathcal{O}_{\mathbf{P}^1}(n))
$$
be the $n$-th Hirzebruch surface, equipped with the natural projection $\pi:\mathbf{F}_{n}\to\mathbf{P}^1$.  
When $n>0$, denote by $\Sigma\subset\mathbf{F}_{n}$ its unique $(-n)$-section. 
We write $L_{a,b}$ for the line bundle on $\mathbf{F}_{n}$ of $\pi$-degree $a$ with $(L_{a,b}, \Sigma)=b$. 
In particular, we have $\pi^{\ast}\mathcal{O}_{\mathbf{P}^1}(1)\simeq L_{0,1}$, 
$\mathcal{O}_{\mathbf{F}_{n}}(\Sigma)\simeq L_{1,-n}$ and $K_{\mathbf{F}_{n}}\simeq L_{-2,-2+n}$. 

Let $(X, \iota)$ be a 2-elementary $K3$ surface. 
A line bundle $L$ on $X$ with $(L, L)=2d>0$ is called \textit{hyperelliptic} (\cite{SD}) 
if the linear system $|L|$ contains a smooth hyperelliptic member. 
In that case, $L$ is base point free and  
every smooth member of $|L|$ is hyperelliptic of genus $d+1$. 
The associated morphism 
\begin{equation}\label{eqn:HE morphism}
\phi_L : X \to |L|^{\vee} \simeq \mathbf{P}^{d+1}
\end{equation} 
is generically two-to-one onto its image, 
mapping a smooth member of $|L|$ to a rational normal curve in a hyperplane of $|L|^{\vee}$. 
According to Saint-Donat (\cite{SD} \S 5), we have the following possibilities 
for the image surface $\phi_L(X)$. 
\begin{enumerate}
\renewcommand{\theenumi}{\Roman{enumi}}
\item $\phi_L(X)$ coincides with $|L|^{\vee}\simeq\mathbf{P}^2$: this is the case $d=1$.  
\item $\phi_L(X)$ is a Veronese surface in $|L|^{\vee}\simeq\mathbf{P}^5$:  
          this happens when $L=2L'$ for $L'\in {\rm Pic}(X)$ with $(L', L')=2$. 
\item $\phi_L(X)$ is a rational normal scroll, that is, the embedding image of $\mathbf{F}_n$ 
           by a line bundle $L_{1,m}$ with $m>0$ and $n+2m=d$. 
\item $\phi_L(X)$ is a cone over a rational normal curve, that is, the image of $\mathbf{F}_d$ by the bundle $L_{1,0}$. 
           In this case $\phi_L$ lifts to a morphism $X\to\mathbf{F}_d$, 
           and we must have $2\leq d\leq4$.     
\end{enumerate}

We now assume that the hyperelliptic bundle $L$ is $\iota$-invariant, in the sense that there exists an isomorphism 
\begin{equation*}
\iota^{\ast}L\simeq L. 
\end{equation*}
Although $\iota$ may not necessarily act on $L$ equivariantly, it does so on the morphism \eqref{eqn:HE morphism}. 
We then obtain an $\iota$-equivariant morphism 
\begin{equation}\label{eqn:HE 2:1 cover}
\phi : X\to Y 
\end{equation}
with $Y=\mathbf{P}^2$ in cases (I), (II), and $Y=\mathbf{F}_n$, $\mathbf{F}_d$ in cases (III), (IV). 
\par
It will be useful to have a purely lattice-theoretic method for finding an $\iota$-invariant hyperelliptic bundle. 
This is based on the following lemmas. 
Denote by $H_+=H^2(X, \mathbf{Z})_+$ the invariant lattice of $(X, \iota)$.

\begin{lemma}\label{HE LB}
Let $L\in H_+$ be nef with $(L, L)=2d>0$. 
Assume that 
\begin{itemize}
\item[(a)] 
there exists $E\in {\rm Pic}(X)$ with $(E, E)\geq0$ and $(E, L)=2$, and 
\item[(b)] 
there is no $F\in H_+$ with $(F, F)=0$ and $(F, L)=1$. 
\end{itemize}
\noindent
Then $L$ is hyperelliptic. 
\end{lemma}

\begin{proof}
We first show that $L$ is base point free. 
Otherwise, by \cite{May} Proposition 5 
the linear system $|L|$ would be of the form $|(d+1)F|+\Gamma$ where $F$ is a smooth elliptic curve 
and $\Gamma$ is a $(-2)$-curve with $(F, \Gamma)=1$. 
Since $\iota$ acts on $|L|$, the class of $F$ is $\iota$-invariant and then 
would violate the assumption $(\textrm{b})$. 
Hence $L$ is free, and a general member $C\in|L|$ is smooth and irreducible of genus $d+1$. 
We show that $C$ is hyperelliptic. 
Since $g(C)=2$ when $d=1$, we may assume $d>1$. 
Let $E$ be a divisor as in the assumption $(\textrm{a})$. 
Since $h^0(-E)=0$, we have $h^0(E)\geq2$ by the Riemann-Roch inequality. 
Consider the exact sequence 
\begin{equation}
0 \to H^0(E-L) \to H^0(E) \to H^0(E|_C) \to \cdots 
\end{equation}
Since $(E-L, L)<0$, we have $h^0(E-L)=0$ by the nefness of $L$. 
Hence $h^0(E|_C)\geq h^{0}(E)\geq2$ and $E|_C$ gives a $g^1_2$ on $C$. 
\end{proof}

The conditions (a) and (b) are purely arithmetic. 
To meet the nefness condition is always possible by the following.

\begin{lemma}\label{reflection}
Let $W(X)$ be the Weyl group of ${\rm Pic}(X)$ generated by the reflections with respect to $(-2)$-vectors in ${\rm Pic}(X)$. 
Let $L\in H_+$ be a line bundle with $(L, L)\geq0$ and $(L, L_0)>0$ for some ample class $L_0$. 
Then there exists $w\in W(X)$ such that $w\circ\iota=\iota\circ w$ and $w(L)$ is nef. 
\end{lemma}

\begin{proof}
The same argument as in \cite{B-H-P-V} Proposition 21.1 applies with few minor modification. 
We leave it to the reader. 
\end{proof}

Thus we can obtain an $\iota$-equivariant morphism \eqref{eqn:HE 2:1 cover} 
by just finding a vector in $H_+$ with the arithmetic conditions (a) and (b). 
The $\iota$-equivariant Weyl group action as in Lemma \ref{reflection} would then carry this vector to a class of hyperelliptic bundle. 
\par
We are interested in when $\iota$ acts by the covering transformation of $\phi:X\to Y$, 
in which case $(X, \iota)$ may be recovered from $Y$ and the branch curve of $\phi$.  
Let $g$ denote the genus of the main component of the fixed curve $X^{\iota}$. 

\begin{lemma}\label{cover trans}
The involution $\iota$ acts trivially on $Y$ when 

$(1)$ $g\geq3$ in case $Y=\mathbf{P}^2$, 

$(2)$ $g\geq4$ in case $Y=\mathbf{P}^1\times\mathbf{P}^1$, and 

$(3)$ $g\geq n+2$ in case $Y=\mathbf{F}_n$ with $n>0$. 
\end{lemma}

\begin{proof}
We may assume that $g>0$. 
Let $B\subset Y$ be the branch curve of $\phi$, which belongs to $|-2K_Y|$. 
Suppose that $\iota$ acts nontrivially on $Y$; 
then the genus $g$ component of $X^{\iota}$ is 
the normalization of the double cover of a curve component $D$ of $Y^{\iota}$ branched over $B|_D$. 
When $Y=\mathbf{P}^2$, $D$ must be a line and so intersects with $B$ at six points. 
When $Y=\mathbf{P}^1\times\mathbf{P}^1$, $D$ is either a ruling fiber or a smooth bidegree $(1, 1)$ curve, 
which satisfies $(D, B)\leq8$. 
Finally, let $Y=\mathbf{F}_n$ with $n>0$. 
If $\iota$ acts nontrivially on the $(-n)$-section $\Sigma$, 
then $D$ is a ruling fiber so that $(D, -2K_Y)=4$. 
If $\iota$ acts trivially on $\Sigma$, we have $Y^{\iota}=H+\Sigma$ for a smooth $H\in|L_{1,0}|$
because $\iota$ must preserve every fiber of $\mathbf{F}_n$ and hence induces a non-trivial involution on every fiber. 
Then $(\Sigma,B)=(\Sigma, -2K_Y)\leq2$ and $(H,B)=(H, -2K_Y)=2n+4$. 
This gives us the estimate $2g+2\leq(D, B)\leq 2n+4$. 
\end{proof}

This criterion is coarse, but will suffice for our purpose.

\section{Automorphic forms on the period domain}
\label{sect:4}
\par

\subsection{Siegel modular forms}
\label{sect:4.1}
\par
Recall that the line bundle on ${\mathcal A}_{g}$ associated with the automorphic factor 
${\rm Sp}_{2g}({\bf Z})\ni\binom{A\,B}{C\,D}\mapsto\det(C\varOmega+D)\in{\mathcal O}({\frak S}_{g})$, $\varOmega\in{\frak S}_{g}$ 
is called the {\em Hodge line bundle} on ${\mathcal A}_{g}$ and is denoted by ${\mathcal F}_{g}$ in this paper. 
A holomorphic section of ${\mathcal F}_{g}^{\otimes q}$ is identified with a Siegel modular form of weight $q$,
and ${\mathcal F}_{g}^{\otimes q}$ is equipped with the Hermitian metric $\|\cdot\|_{{\mathcal F}_{g}^{\otimes q}}$ called the Petersson norm:
For any Siegel modular form $S$ of weight $q$, we define $\|S(\varOmega)\|_{{\mathcal F}_{g}^{\otimes q}}^{2}:=(\det\Im\,\varOmega)^{q}|S(\varOmega)|^{2}$.
\par
In this paper, the following Siegel modular forms on ${\frak S}_{g}$ play crucial roles:
$$
\chi_{g}(\varOmega)^{8}:=\prod_{(a,b)\,{\rm even}}\theta_{a,b}(\varOmega)^{8},
$$
$$
\Upsilon_{g}(\varOmega):=\chi_{g}(\varOmega)^{8}\sum_{(a,b)\,{\rm even}}\theta_{a,b}(\varOmega)^{-8}.
$$
Here 
$$
\theta_{a,b}(\varOmega):=\sum_{n\in{\bf Z}^{g}}\exp\{\pi\sqrt{-1}{}^{t}(n+a)\varOmega(n+a)+2\pi\sqrt{-1}{}^{t}(n+a)b\}
$$
is the theta constant with even characteristic $(a,b)$, where $a,b\in\{0,1/2\}^{g}$ and $4{}^{t}ab\equiv0\mod2$.
For $g=0$, we set $\chi_{0}=\Upsilon_{0}=1$. 
Note that $\Upsilon_{g}(\varOmega)$ is the elementary symmetric polynomial of degree $2^{g-1}(2^{g}+1)-1=(2^{g-1}+1)(2^{g}-1)$ 
in the even theta constants $\theta_{a,b}(\varOmega)^{8}$.
By \cite[p.176 Cor. and p.182 Th.\,3]{Igusa72}, $\chi_{g}^{8}$ (resp. $\Upsilon_{g}$) is a Siegel modular form of weight $2^{g+1}(2^{g}+1)$ 
(resp. $2(2^{g}-1)(2^{g}+2)$).
The locus of vanishing thetanull $\theta_{{\rm null},g}$ is the reduced divisor on ${\mathcal A}_{g}$ defined by $\chi_{g}$
$$
\theta_{{\rm null},g}:=\{[\varOmega]\in{\mathcal A}_{g};\,\chi_{g}(\varOmega)=0\}.
$$

\begin{lemma}
\label{lemma:locus:vanishing:two:theta:char}
There exist at least two distinct vanishing even theta constants at $\varOmega\in{\frak S}_{g}$ 
if and only if $\chi_{g}(\varOmega)=\Upsilon_{g}(\varOmega)=0$.
In particular, a smooth projective curve $C$ of genus $g$ has at least two effective even theta characteristics if and only if
$\chi_{g}(\varOmega(C))=\Upsilon_{g}(\varOmega(C))=0$.
\end{lemma}

\begin{pf}
Assume $\chi_{g}(\varOmega)=\Upsilon_{g}(\varOmega)=0$. Since $\chi_{g}(\varOmega)=0$, there is an even pair $(a,b)$ with $\theta_{a,b}(\varOmega)=0$.
Then $\Upsilon_{g}(\varOmega)=\prod_{(c,d)\not=(a,b)}\theta_{c,d}(\varOmega)^{8}$, where $(c,d)\in\{0,1/2\}^{2g}$ runs over all even pairs
distinct from $(a,b)$. Since $\Upsilon_{g}(\varOmega)=0$, we get $\theta_{c,d}(\varOmega)=0$ for some even pair $(c,d)\not=(a,b)$.
Thus $\theta_{a,b}(\varOmega)=\theta_{c,d}(\varOmega)=0$. The converse is trivial. 
The second assertion follows from the Riemann singularity theorem \cite[p.226]{A-C-G-H}. 
\end{pf}

For the proof of Theorem~\ref{thm:main:theorem:1} (2), (3),
we need an estimate for the vanishing order of $\Upsilon_{g}$ for certain ordinary singular families of curves.

\begin{lemma}
\label{lemma:estimate:zero:Upsilon_g:ordinary:singular:family}
Let $p\colon{\mathcal C}\to\varDelta$ be an ordinary singular family of curves of genus $g>0$ with irreducible $C_{0}:=p^{-1}(0)$. 
Namely, $p\colon{\mathcal C}\to\varDelta$ is a proper surjective holomorphic function from a complex surface ${\mathcal C}$ to the unit disc $\varDelta$
without critical points on ${\mathcal C}\setminus C_{0}$ and with a unique, non-degenerate critical point on $C_{0}$.
Assume that $\chi_{g}(\varOmega(C_{t}))=0$ and $\Upsilon_{g}(\varOmega(C_{t}))\not=0$ for all $t\in\varDelta^{*}$ and that 
$\chi_{g-1}(\varOmega(\widehat{C}_{0}))\not=0$, where $\widehat{C}_{0}$ is the normalization of $C_{0}$.
Then there exists $h(t)\in{\mathcal O}(\varDelta)$ such that
$$
\log\|\Upsilon_{g}(\varOmega(C_{t}))\|^{2}
=
(2^{2g-2}-1)\log|t|^{2}+\log|h(t)|^{2}+O\left(\log\log|t|^{-1}\right)
\qquad
(t\to0).
$$
\end{lemma}

\begin{pf}
We follow \cite[Proof of Lemma 4.1]{Yoshikawa13}.
For $\varOmega\in{\frak S}_{g}$, write $\varOmega=\binom{z\,{}^{t}\omega}{\omega\,Z}$,
where $z\in{\frak H}$, $\omega\in{\bf C}^{g-1}$, $Z\in{\frak S}_{g-1}$.
For $t\in\varDelta^{*}$, we can express
$$
\varOmega(C_{t})=\left[\frac{\log t}{2\pi i}A+\psi(t)\right]\in{\mathcal A}_{g},
\qquad
A=
\begin{pmatrix}
1&{}^{t}{\bf 0}_{g-1}\\
{\bf 0}_{g-1}&O_{g-1}
\end{pmatrix},
$$
where $\psi(t)$ is a holomorphic function on $\varDelta$ with values in complex symmetric $g\times g$-matrices such that
$\psi(0)=\binom{\psi_{0}\,{}^{t}\omega_{0}}{\omega_{0}\,\,Z_{0}}$, $Z_{0}\in{\frak S}_{g-1}$, $\varOmega(\widehat{C}_{0})=[Z_{0}]\in{\mathcal A}_{g-1}$.
\par
By the assumption $\chi_{g}(\varOmega(C_{t}))=0$, $\Upsilon_{g}(\varOmega(C_{t}))\not=0$ for all $t\in\varDelta^{*}$
and Lemma~\ref{lemma:locus:vanishing:two:theta:char}, $C_{t}$ has a unique effective even theta characteristic for $t\in\varDelta^{*}$.
By fixing a marking of a reference curve, there is a unique even pair $(a,b)$, $a,b\in\{0,1/2\}^{g}$ such that 
$\theta_{a,b}(\varOmega(C_{t}))=0$ on $\varDelta$ and $\theta_{c,d}(\varOmega(C_{t}))\not=0$ on $\varDelta^{*}$ for all even $(c,d)\not=(a,b)$.
Write $a=(a_{1},a')$ and $b=(b_{1},b')$. 
If $a_{1}=0$, then we get $\theta_{a',b'}(Z_{0})=0$ for the even pair $(a',b')$, $a',b'\in\{0,1/2\}^{g-1}$ by \cite[Eq.\,(4.4)]{Yoshikawa13}, 
which contradicts the assumption $\chi_{g-1}(Z_{0})=\chi_{g-1}(\varOmega(\widehat{C}_{0}))\not=0$. Thus $a_{1}=1/2$.
\par 
By \cite[Eqs.(4.3), (4.4)]{Yoshikawa13}, there is a holomorphic function $F_{a,b}(\zeta,\omega,Z)$ such that
$$
\prod_{(c,d)\not=(a,b)}\theta_{c,d}(\varOmega)
=
(e^{\pi iz/4})^{2^{2(g-1)}-1}\,F_{a,b}(e^{\pi iz},\omega,Z).
$$
Hence there is a holomorphic function $\phi(\zeta,\omega,Z)$ such that
$$
\Upsilon_{g}(\varOmega)=(e^{2\pi iz})^{2^{2(g-1)}-1}\,\phi(e^{\pi iz},\omega,Z).
$$
Since $\Upsilon_{g}(\varOmega)$ is a Siegel modular form and hence $\Upsilon_{g}(\varOmega+A)=\Upsilon_{g}(\varOmega)$,
$\phi(\zeta,\omega,Z)$ is an even function in $\zeta$. There exists a holomorphic function $h(t)\in{\mathcal O}(\varDelta)$ such that
$$
\Upsilon_{g}((\log t/2\pi i)A+\psi(t))=t^{2^{2g-2}-1}h(t).
$$
This, together with \cite[(4.7)]{Yoshikawa13}, implies the result.
\end{pf}

As a consequence of Lemma~\ref{lemma:estimate:zero:Upsilon_g:ordinary:singular:family},
we get the following.

\begin{lemma}
\label{lemma:estimate:alpha:U(2)}
Let $M\subset{\Bbb L}_{K3}$ be a primitive $2$-elementary Lorentzian sublattice and set $\Lambda=M^{\perp}$.
Let $\gamma\colon\varDelta\to{\mathcal M}_{\Lambda}$ be a holomorphic curve with $\gamma(\varDelta^{*})\subset{\mathcal M}_{\Lambda}^{0}$
intersecting $\overline{\mathcal D}_{\Lambda}^{0}$ transversally at $\gamma(0)\in\overline{\mathcal D}_{\Lambda}^{0}$. 
Let $\widetilde{\gamma}\colon\varDelta\to\Omega_{\Lambda}$ be its lift with $\gamma(t^{2})=\varPi_{\Lambda}\circ\widetilde{\gamma}(t)$
and let $d\in\Delta_{\Lambda}^{-}$ be such that $\widetilde{\gamma}(0)\in H_{d}$.
If $\chi_{g}({J}_{M}(\widetilde{\gamma}(t)))=0$ and $\Upsilon_{g}({J}_{M}(\widetilde{\gamma}(t)))\not=0$ for all $t\in\varDelta^{*}$ 
and if $\chi_{g-1}(J_{[M\perp d]}(\widetilde{\gamma}(0)))\not=0$, then 
$$
{\rm ord}_{t=0}\widetilde{\gamma}^{*}(J_{M}^{*}\Upsilon_{g})\geq2(2^{2(g-1)}-1).
$$
\end{lemma}

\begin{pf}
Let $\overline{J}_{M}\colon{\mathcal M}_{\Lambda}^{0}\cup\overline{\mathcal D}_{\Lambda}^{0}\to{\mathcal A}_{g}^{*}$ be the extension of 
$\overline{J}_{M}^{0}\colon{\mathcal M}_{\Lambda}^{0}\to{\mathcal A}_{g}$.
By \cite[Theorem 2.3 (1), (2)]{Yoshikawa13}, there is an ordinary singular family of curves $p\colon{\mathcal C}\to\varDelta$ of genus $g$ 
with irreducible $C_{0}$ and with period map $\overline{J}_{M}\circ\gamma$.
By Lemma~\ref{lemma:estimate:zero:Upsilon_g:ordinary:singular:family}, we get
$$
{\rm ord}_{t=0}\gamma^{*}(\overline{J}_{M}^{*}\Upsilon_{g})\geq2^{2(g-1)}-1,
$$
which, together with $J_{M}(\gamma(t^{2}))=\overline{J}_{M}(\widetilde{\gamma}(t))$, yields the result.
\end{pf}

\subsection{Automorphic forms on $\Omega_{\Lambda}$}
\label{sect:4.2}
\par
Let $M\subset{\Bbb L}_{K3}$ be a primitive $2$-elementary Lorentzian sublattice and set $\Lambda=M^{\perp}$ as before.
Let $q\in{\bf Z}_{>0}$ be such that ${\mathcal F}_{g}^{\otimes q}$ extends to a very ample line bundle on ${\mathcal A}_{g}^{*}$.
Let $i\colon\Omega_{\Lambda}^{0}\cup{\mathcal D}_{\Lambda}^{0}\hookrightarrow\Omega_{\Lambda}$ be the inclusion 
and define $\lambda_{M}^{q}$ as the trivial extension of $J_{M}^{*}{\mathcal F}_{g(M)}^{\otimes q}$ 
from $\Omega_{\Lambda}^{0}\cup{\mathcal D}_{\Lambda}^{0}$ to $\Omega_{\Lambda}$, i.e.,
$$
\lambda_{M}^{q}:=i_{*}{\mathcal O}_{\Omega_{\Lambda}^{0}\cup{\mathcal D}_{\Lambda}^{0}}\left(J_{M}^{*}{\mathcal F}_{g(M)}^{\otimes q}\right).
$$
Since $\Omega_{\Lambda}\setminus(\Omega_{\Lambda}^{0}\cup{\mathcal D}_{\Lambda}^{0})$ has codimension $2$ in $\Omega_{\Lambda}$,
$\lambda_{M}^{q}$ is an $O(\Lambda)$-equivariant invertible sheaf on $\Omega_{\Lambda}$.
On $\Omega_{\Lambda}^{0}$, $\lambda_{M}^{q}$ is equipped with the Hermitian metric
$$
\|\cdot\|_{\lambda_{M}^{q}}:=J_{M}^{*}\|\cdot\|_{{\mathcal F}_{g}^{\otimes q}}.
$$
\par
Fix $l_{\Lambda}\in\Lambda\otimes{\bf R}$ with $\langle l_{\Lambda},l_{\Lambda}\rangle\geq0$.
Define $j_{\Lambda}(\gamma,\cdot)\in{\mathcal O}^{*}_{\Omega_{\Lambda}}$, $\gamma\in O(\Lambda)$ 
and  $K_{\Lambda}(\cdot)\in C^{\infty}(\Omega_{\Lambda})$ by
$$
j_{\Lambda}(\gamma,[\eta])
:=
\frac{\langle\gamma(\eta),l_{\Lambda}\rangle}{\langle\eta,l_{\Lambda}\rangle},
\qquad
K_{\Lambda}([\eta]):=\frac{\langle\eta,\overline{\eta}\rangle}{|\langle\eta,l_{\Lambda}\rangle|^{2}}.
$$
Let $p,q\in{\bf Z}$.
Then $F\in H^{0}(\Omega_{\Lambda},\lambda_{M}^{q})$ is called an 
{\em automorphic form on $\Omega_{\Lambda}$ for $O(\Lambda)$ of weight $(p,q)$}
if it satisfies the following functional equation on $\Omega_{\Lambda}$:
\begin{equation}
\label{eqn:functional:eq:automorphic:form}
F(\gamma\cdot[\eta])=j_{\Lambda}(\gamma,[\eta])^{p}\,\gamma(F([\eta])),
\qquad
\forall\,\gamma\in O(\Lambda).
\end{equation}
The notion of automorphic forms on $\Omega_{\Lambda}^{+}$ for $O^{+}(\Lambda)$ of weight $(p,q)$ is defined in the same way.
In the rest of this paper, the vector space of automorphic forms on $\Omega_{\Lambda}$ for $O(\Lambda)$ of weight $(p,q)$ is identified with
the vector space of automorphic forms on $\Omega_{\Lambda}^{+}$ for $O^{+}(\Lambda)$ of weight $(p,q)$ via the restriction map 
$$
H^{0}(\Omega_{\Lambda},\lambda_{M}^{q})\ni F\to F|_{\Omega_{\Lambda}^{+}}\in H^{0}(\Omega_{\Lambda}^{+},\lambda_{M}^{q}).
$$
\par
We define the {\em Petersson norm} of an automorphic form $F$ on $\Omega_{\Lambda}$ for $O(\Lambda)$ of weight $(p,q)$
as the $O(\Lambda)$-invariant $C^{\infty}$ function on $\Omega_{\Lambda}^{0}$ defined as
$$
\|F([\eta])\|^{2}:=K_{\Lambda}([\eta])^{p}\cdot\|F([\eta])\|_{\lambda_{M}^{q}}^{2}.
$$

\section{The invariant $\tau_{M}$ and its automorphic property}
\label{sect:5}
\par
Let $(X,\iota)$ be a $2$-elementary $K3$ surface of type $M$. Let $\gamma$ be an $\iota$-invariant K\"ahler form on $X$. 
The Laplacian acting on $(0,q)$-forms on $X$ is denoted by $\square_{0.q}$.
Write $\sigma(\square_{0,q})$ for the spectrum of $\square_{0,q}$ and $E_{0,q}(\lambda)$ for eigenspace of $\square_{0,q}$ 
corresponding to $\lambda\in\sigma(\square_{0,q})$. 
Since $\iota$ preserves $\gamma$, $\iota$ acts on $E_{0,q}(\lambda)$. 
The equivariant zeta function of $\square_{0,q}$ is defined as the following convergent series
for $s\in{\bf C}$ with $\Re s\gg0$:
$$
\zeta_{0,q}(\iota)(s)
:=
\sum_{\lambda\in\sigma(\square_{0,q})\setminus\{0\}}
{\rm Tr}\,(\iota|_{E_{0,q}(\lambda)})\,\lambda^{-s}.
$$ 
Then $\zeta_{0,q}(\iota)(s)$ extends meromorphically to ${\bf C}$ and is holomorphic at $s=0$. 
We define the {\it equivariant analytic torsion} \cite{Bismut95} of $(X,\gamma)$ as
$$
\tau_{{\bf Z}_{2}}(X,\gamma)(\iota)
:=
\exp[
-\sum_{q\geq0}(-1)^{q}q\,\zeta'_{0,q}(\iota)(0)].
$$
Let $\eta\in H^{0}(X,K_{X})\setminus\{0\}$ and set $\|\eta\|_{L^{2}}^{2}:=(2\pi)^{-2}\int_{X}\eta\wedge\bar{\eta}$.
For a compact K\"ahler manifold $(V,\omega)$, define ${\rm Vol}(V,\omega):=(2\pi)^{-\dim V}\int_{V}\omega^{\dim V}/(\dim V)!$.
\par
When $X^{\iota}\not=\emptyset$, write $X^{\iota}=\amalg_{i}C_{i}$ for the decomposition into the connected components. 
Let $\tau(C_{i},\gamma|_{C_{i}})$ be the analytic torsion \cite{RaySinger73} of $(C_{i},\gamma|_{C_{i}})$.
We define
$$
\tau(X^{\iota},\gamma|_{X^{\iota}}):=\prod_{i}\tau(C_{i},\gamma|_{C_{i}}),
\qquad
{\rm Vol}(X^{\iota},\gamma|_{X^{\iota}}):=\prod_{i}{\rm Vol}(C_{i},\gamma|_{C_{i}}).
$$
When $X^{\iota}=\emptyset$, we set $\tau(X^{\iota},\gamma|_{X^{\iota}})={\rm Vol}(X^{\iota},\gamma|_{X^{\iota}})=1$.
Let $c_{1}(X^{\iota},\gamma|_{X^{\iota}})$ be the first Chern form of $(X^{\iota},\gamma|_{X^{\iota}})$.
By \cite[Th.\,5.7]{Yoshikawa04}, the real number 
$$
\begin{aligned}
\tau_{M}(X,\iota)
&:=
{\rm Vol}(X,\gamma)^{\frac{14-r(M)}{4}}
\tau_{{\bf Z}_{2}}(X,\gamma)(\iota)\,
{\rm Vol}(X^{\iota},\gamma|_{X^{\iota}})
\tau(X^{\iota},\gamma|_{X^{\iota}})
\\
&\quad
\times
\exp\left[
\frac{1}{8}\int_{X^{\iota}}
\log\left.\left(\frac{\eta\wedge\overline{\eta}}
{\gamma^{2}/2!}\cdot
\frac{{\rm Vol}(X,\gamma)}{\|\eta\|_{L^{2}}^{2}}
\right)\right|_{X^{\iota}}
c_{1}(X^{\iota},\gamma|_{X^{\iota}})
\right]
\end{aligned}
$$
is determined by the isomorphism class of $(X,\iota)$ and hence the period $\overline{\pi}_{M}(X,\iota)$. 
For the arithmetic counterpart of the invariant $\tau_{M}(X,\iota)$, we refer to \cite{MaillotRoessler09}.
\par
We set $\Lambda=M^{\perp}$ and we regard $\tau_{M}$ as the $O(\Lambda)$-invariant function on $\Omega_{\Lambda}^{0}$ or equivalently 
the function on ${\mathcal M}_{\Lambda}^{0}$ defined by
$$
\tau_{M}\left(\overline{\pi}_{M}(X,\iota)\right)
:=
\tau_{M}(X,\iota).
$$
\par
As an application of the theory of (equivariant) Quillen metrics \cite{BGS88}, \cite{Bismut95}, \cite{MaX00}, 
the automorphy of $\tau_{M}$ was established.

\begin{theorem}[\cite{Yoshikawa04}, \cite{Yoshikawa13}, \cite{Yoshikawa12}]
\label{thm:automorphic:property:Phi:M}
There exist $\ell\in{\bf Z}_{>0}$ and an automorphic form $\Phi_{M}$ on $\Omega_{\Lambda}$ for $O(\Lambda)$ of weight $(\ell(r(M)-6),4\ell)$ with
\begin{equation}
\label{eqn:automorphic:property:torsion}
\tau_{M}=\left\|\Phi_{M}\right\|^{-\frac{1}{2\ell}},
\qquad
{\rm div}\,\Phi_{M}=\ell\,{\mathcal D}_{\Lambda}.
\end{equation}
\end{theorem}

In the rest of this paper, we determine $\Phi_{M}$ for all $M$.
Since it was done for exceptional $M$ in \cite{Yoshikawa04}, {\em $M$ is assumed to be non-exceptional} in what follows.

\section{The locus of vanishing theta-null: the case $\delta=1$}
\label{sect:6}
\par
In the rest of this paper, for a primitive $2$-elementary Lorentzian sublattice $M\subset{\Bbb L}_{K3}$, 
we write $r$, $l$, $\delta$, $g$, $k$ for $r(M)$, $l(M)$, $\delta(M)$, $g(M)$, $k(M)$, respectively, when there is no possibility of confusion. 
Recall that these invariants are introduced in Sections \ref{sect:1} and \ref{sect:3.2.1}. 
We write 
$$
\Lambda=M^{\perp}.
$$
Let ${\frak M}_{g}$ be the moduli space of smooth curves of genus $g\geq1$.
In Sections \ref{sect:6} and \ref{sect:7}, we study the Torelli map $J_M: \Omega_{\Lambda}^0\to \mathcal{A}_g$ from the geometric point of view. 
We here view the Torelli map rather as a morphism 
\begin{equation*}
\overline{\mu}_{\Lambda} : \mathcal{M}_{\Lambda}^{0} \to \frak{M}_g 
\end{equation*}
from the moduli space of 2-elementary $K3$ surfaces to that of curves, 
which associates to $(X, \iota)$ the genus $g$ component of $X^{\iota}$. 
Our main purpose here is to describe the inverse image of a certain geometric locus in $\frak{M}_g$ 
as a Heegner divisor of $\mathcal{M}_{\Lambda}^{0}$ in a few cases. 
This will be the first and necessary step toward a more complete description, Theorem~\ref{thm:divisor:pullback:chi:upsilon}, which will be obtained 
in the final part of the paper. 
\par
Recall that 
the characteristic Heegner divisor $\mathcal{H}_{\Lambda}$ of $\Omega_{\Lambda}$ was defined in Section \ref{sect:2.3}, 
and that the thetanull divisor $\theta_{{\rm null},g}$ of ${\mathcal A}_{g}$ was defined in Section \ref{sect:4.1}. 
We denote by 
\begin{equation*}
\overline{\mathcal H}_{\Lambda} \subset \mathcal{M}_{\Lambda} 
\end{equation*}
the {\em reduced} algebraic divisor corresponding to $\mathcal{H}_{\Lambda}$, 
and 
\begin{equation*}
{\frak M}_{g}' = \theta_{{\rm null},g} \cap {\frak M}_{g} 
\end{equation*}
the reduced thetanull divisor in ${\frak M}_{g}$.  
It is well-known that ${\frak M}_{g}'\subset{\frak M}_{g}$ is the locus of curves $C$ having an effective even theta characteristic, 
namely an effective line bundle $L$ with $L^{\otimes2}\simeq K_C$ and $h^0(L)$ even. 

In the present section we treat the moduli spaces in the following two series:  
\begin{itemize}
\item 
$k=0$, $\delta=1$, $3\leq g\leq 10$,   
\item 
$k=1$, $6\leq g\leq 9$.    
\end{itemize}
Notice that in the second series we have $\delta=0$ only when $g=6$. We will prove the following.

\begin{theorem}
\label{thetanull=Heegner}
If $3\leq g\leq10$ and $(k,\delta)=(0,1)$, the Heegner divisor $\overline{\mathcal H}_{\Lambda}$ is irreducible 
and the following set-theoretic equality of (reduced) divisors of $\mathcal{M}_{\Lambda}^0$ holds:
\begin{equation*}\label{eqn:thetanull=Heegner}
\overline{\mu}_{\Lambda}^{-1}({\frak M}_{g}') = \overline{\mathcal H}_{\Lambda}\cap \mathcal{M}_{\Lambda}^0.
\end{equation*} 
\end{theorem}

The same assertion also holds for the second series when $\delta=1$, 
but the following weaker version will suffice for the rest of the paper.

\begin{proposition}
\label{thetanull k=1}
If $6\leq g\leq9$ and $k=1$, 
$\overline{\mu}_{\Lambda}(\mathcal{M}_{\Lambda}^0)$ is not contained in ${\frak M}_{g}'$. 
\end{proposition}

These results will be used in Section \ref{sect:8}. 
Theorem \ref{thetanull=Heegner} will be proved in Sections \ref{sect:6.2}--\ref{ssec:3<g<11,k=0}, 
and Proposition \ref{thetanull k=1} in Section \ref{ssec:k=1}.

\subsection{Proof of Theorem \ref{thetanull=Heegner}: the strategy}
\label{sect:6.2}
\par
Let us first explain the outline of the proof of Theorem \ref{thetanull=Heegner}, 
reducing it to the construction of certain elliptic curves. 

As the first step we see the irreducibility of $\overline{\mathcal H}_{\Lambda}$, 
which holds in a wider range. 

\begin{lemma}
\label{lemma:irreducibility:characteristic:divisor}
When $\delta=1$ and $r\leq9$, $\overline{\mathcal H}_{\Lambda}$ is an irreducible divisor of $\mathcal{M}_{\Lambda}$. 
\end{lemma}

\begin{proof}
This is restated as the property that 
vectors $l\in\Lambda^{\vee}$ with $(l, l)=\varepsilon_{\Lambda}$ and $[l]=\bold{1}_{\Lambda}$ are all equivalent under $O^+(\Lambda)$. 
Consider the vector $l'=2l$ in $\Lambda$. 
Since $\bold{1}_{\Lambda}$ is of order $2$, $l'$ is primitive in $\Lambda$ and satisfies  
\begin{equation*}
{\rm div}(l')=2, \quad 
[l'/{\rm div}(l')] = \bold{1}_{\Lambda}\in A_{\Lambda}, \quad 
(l', l')=4\varepsilon_{\Lambda}. 
\end{equation*}
When $(r, l)\ne(9, 9)$, the lattice $\Lambda$ contains $\mathbb{U}\oplus \mathbb{U}$,  
and then we can resort to the Eichler criterion (cf.~\cite{Scattone87}) 
which says that the $O^+_0(\Lambda)$-equivalence class of a primitive vector $l'\in \Lambda$ 
depends only on the norm $(l', l')$ and the element $[l'/{\rm div}(l')]\in A_{\Lambda}$. 
Hence the above vectors $l'=2l$ are all $O^+_0(\Lambda)$-equivalent. 
When $(r, l)=(9, 9)$, $\overline{\mathcal H}_{\Lambda}$ is defined by $(-2)$-vectors $l'\in\Lambda$ with $\Lambda=\mathbf{Z}l'\oplus (l')^{\perp}$. 
The isometry class of $(l')^{\perp}$ is uniquely determined by \cite{Nikulin80a}, namely 
$(l')^{\perp}\simeq {\Bbb U}^{\oplus2}\oplus{\Bbb E}_{8}(2)$, so that these $(-2)$-vectors $l'$ are all ${\rm O}^{+}(\Lambda)$-equivalent. 
\end{proof}

With the irreducibility of $\overline{\mathcal H}_{\Lambda}$ verified, 
the proof of Theorem \ref{thetanull=Heegner} is reduced to 
showing the non-emptiness of $\overline{\mu}_{\Lambda}^{-1}({\frak M}_{g}')$ and the inclusion 
\begin{equation}
\label{eqn:thetanull subset Heegner}
\overline{\mu}_{\Lambda}^{-1}({\frak M}_{g}') \subset \overline{\mathcal H}_{\Lambda}. 
\end{equation}
We only need to verify \eqref{eqn:thetanull subset Heegner} outside a locus of codimension $\geq2$ of $\mathcal{M}_{\Lambda}^{0}$. 
Our approach will be based on the following geometric observation.

\begin{proposition}
\label{prop:elliptic:curve:theta:char}
Let $(X,\iota)$ be a $2$-elementary $K3$ surface.
If $X$ has a smooth elliptic curve $E$ with $E+\iota(E) \sim X^{\iota}$, 
then the period of $(X, \iota)$ is contained in the Heegner divisor $\overline{\mathcal H}_{\Lambda}$. 
\end{proposition}

\begin{proof}
Let $H_{\pm}$ denote the $\iota$-(anti-)invariant lattices of $(X, \iota)$. 
The cycle $D_{\pm}:=E\pm\iota(E)$ is contained in $H_{\pm}$ respectively. 
We will show that $D_{-}$ satisfies 
\begin{equation}
\label{eqn:condition on cycle for Heegner divisor}
(D_-, D_-)=2r-20, 
\qquad 
D_-/2\in H_-^{\vee}, 
\qquad 
[D_-/2] = \bold{1}_{H_-} \in A_{H_-}. 
\end{equation}
The presence of such an anti-invariant cycle in the Picard lattice 
means that the period of $(X, \iota)$ lies in $\overline{\mathcal H}_{\Lambda}$. 

Since $E$ is an elliptic curve and hence its class in $H^{2}(X,{\bf Z})$ is isotropic, 
the first equality of \eqref{eqn:condition on cycle for Heegner divisor} follows from
\begin{equation}
\label{eqn:1st condition for inv cycle}
2(E, \iota(E)) = (D_+, D_+) = (X^{\iota}, X^{\iota}) = 20-2r. 
\end{equation}
The second property in \eqref{eqn:condition on cycle for Heegner divisor} holds because 
\begin{equation*}
\label{eqn:2nd condition automatic}
(D_-, H_-) = (D_-+D_+, H_-) =(2E, H_-) \subset 2{\bf Z}. 
\end{equation*}
To see the last equality of \eqref{eqn:condition on cycle for Heegner divisor}, 
we note that the anti-isometry $\lambda:A_{H_+}\to A_{H_-}$ induced from the relation $H_+=(H_-)^{\perp}\cap H^2(X,{\bf Z})$ 
maps $[D_+/2]$ to $[D_-/2]$ because $D_+/2+D_-/2=E$ is contained in  $H^2(X,{\bf Z})$. 
Since $[D_+/2]$ is the characteristic element of $A_{H_+}$ by Lemma~\ref{fixed curve characteristic} below, so is $[D_-/2]$ in $A_{H_-}$. 
\end{proof}

\begin{lemma}\label{fixed curve characteristic}
\label{fixed curve charac} 
For any 2-elementary $K3$ surface $(X, \iota)$ the element $[X^{\iota}/2]\in A_{H_+}$ is the characteristic element. 
\end{lemma}

\begin{proof}
Let $f\colon X\to Y$ be the quotient map by $\iota$, and let $B\subset Y$ be the branch curve. 
Every element of the dual lattice $H_+^{\vee}$ can be written as $f^{\ast}L/2$ for some $L\in{\rm Pic}(Y)$. 
Then 
\begin{equation*}
(X^{\iota}/2, f^{\ast}(L/2)) = (B, L)/4 = (-K_Y, L)/2. 
\end{equation*}
We can see that $(L+K_Y, L)\in2{\bf Z}$ from the Riemann-Roch formula. 
Therefore 
\begin{equation*}
(X^{\iota}/2, f^{\ast}(L/2)) \equiv (L, L)/2 = (f^{\ast}(L/2), f^{\ast}(L/2)) \quad {\rm mod} \: {\bf Z}, 
\end{equation*}
which means that $X^{\iota}/2 \in H_+^{\vee}$ and that $[X^{\iota}/2]$ is characteristic. 
\end{proof}

In Sections \ref{ssec:(g,k)=(3,0)} and \ref{ssec:3<g<11,k=0}, we will construct an elliptic curve $E$ as in Proposition~\ref{prop:elliptic:curve:theta:char}. 
The non-emptiness of $\overline{\mu}_{\Lambda}^{-1}({\frak M}_{g}')$ will be seen in the course of proof.

\subsection{Proof of Theorem~\ref{thetanull=Heegner}: the case $g=3$}
\label{ssec:(g,k)=(3,0)}
\par
We begin with the case $g=3$. 
Recall that a smooth curve of genus $3$ has an effective even theta characteristic 
precisely when it is hyperelliptic. 
So let $(X, \iota)$ be a 2-elementary $K3$ surface with $(g, k)=(3, 0)$ such that $X^{\iota}$ is hyperelliptic. 
Consider the degree $4$ line bundle $L=\mathcal{O}_{X}(X^{\iota})$, which is hyperelliptic in the sense of Section \ref{sect:3.3}. 
By Saint-Donat's classification, $L$ defines a degree $2$ morphism 
\begin{equation*}
\phi : X \to Q 
\end{equation*}
onto a quadric $Q\subset\mathbf{P}^3$. 
We may assume that $L$ is ample, 
because the locus where $X^{\iota}$ is hyperelliptic and $\mathcal{O}_{X}(X^{\iota})$ is non-ample 
has codimension $\geq2$ in $\mathcal{M}_{\Lambda}^0$. 
In that case $Q$ is smooth.

\begin{claim}
\label{claim:(g,k)=(3,0)}
$\iota$ acts on $Q$ by switching the two rulings on it. 
\end{claim}

\begin{proof}
First note that $\iota$ acts on $Q$ nontrivially, for the branch curve of $\phi$ is a member of $|-2K_{Q}|$ and hence has genus $9$. 
Then $\iota$ fixes the curve $\phi(X^{\iota})$, which by the definition of $\phi$ has bidegree $(1, 1)$. 
It is easily verified that any non-trivial involution of $Q$ fixing a smooth bidegree $(1, 1)$ curve must switch the two rulings. 
\end{proof}

We choose a ruling line $l$ on $Q$ and put $E = \phi^{\ast}l$. 
Then $E$ is a smooth elliptic curve on $X$ for a general choice of $l$ and satisfies the linear equivalence 
\begin{equation*}
E+\iota(E) \sim \phi^{\ast}\mathcal{O}_Q(1, 1) \sim X^{\iota}.
\end{equation*}
By Proposition~\ref{prop:elliptic:curve:theta:char}, we get the inclusion \eqref{eqn:thetanull subset Heegner}. 

The non-emptiness of $\overline{\mu}_{\Lambda}^{-1}({\frak M}_{3}')$ can be seen by reversing this construction: 
choose a bidegree $(4, 4)$ curve $B\subset \mathbf{P}^1\times\mathbf{P}^1$ 
preserved by the switch involution of $\mathbf{P}^1\times\mathbf{P}^1$ 
and take the double cover $X\to \mathbf{P}^1\times\mathbf{P}^1$ branched over $B$. 
The switch involution can be lifted to $X$ so that its fixed curve is the preimage of the diagonal of 
$\mathbf{P}^1\times\mathbf{P}^1$, which is hyperelliptic of genus $3$. 
Thus Theorem \ref{thetanull=Heegner} is proved in case $g=3$.  
\qed

\subsection{Proof of Theorem~\ref{thetanull=Heegner}: the case $4\leq g\leq10$}
\label{ssec:3<g<11,k=0}
\par
We next treat the case $4\leq g\leq10$ in Theorem~\ref{thetanull=Heegner}. 
Let $(X, \iota)$ be a 2-elementary $K3$ surface with $4\leq g\leq10$, $k=0$ and $\delta=1$. 
Let $Y=X/\iota$ be the quotient surface and $C\subset Y$ be the branch $-2K_Y$-curve. 
The quotient map $X\to Y$ gives the canonical identification 
$$
C\simeq X^{\iota}.
$$ 
The anti-canonical model $\bar{Y}\subset\mathbf{P}^{g-1}$ of $Y$ is a Gorenstein del Pezzo surface of degree $g-1$, 
and $Y$ is its minimal resolution. 
Note that if $X\to\bar{X}$ is the contraction of $(-2)$-curves disjoint from $X^{\iota}$, 
we naturally have $\bar{Y}\simeq\bar{X}/\iota$. 
We may view $C$ also as lying on the smooth locus of $\bar{Y}$. 
By the adjunction formula we have $-K_Y|_C\simeq K_C$, 
and the restriction map $|-K_{Y}|\to |K_{C}|$ is isomorphic because $h^{0}(K_{Y})=h^{1}(K_{Y})=0$. 
Therefore the composition  of inclusions
\begin{equation*}
C \subset \bar{Y} \subset \mathbf{P}^{g-1} 
\end{equation*}
gives the canonical embedding of $C$, 
and we can identify $\mathbf{P}^{g-1}$ with $|K_C|^{\vee}$. 
This also shows that $C$ is non-hyperelliptic.

\begin{lemma} 
If $C$ has an effective even theta characteristic $L$, then $h^0(L)=2$. 
\end{lemma}

\begin{proof}
If $h^0(L)\ne2$, then $h^0(L)\geq4$. 
By Clifford's theorem we have 
\begin{equation*}
6\leq 2 {\dim}|L| < {\rm deg}(L) =g-1. 
\end{equation*}
This holds only when $g\geq8$ and $h^0(L)=4$. 
Actually the case $g=8$ can be excluded because a blow-down of $Y$ presents $C$ as a plane sextic with two double points and hence $C$ has Clifford index $2$. 
The case $g=10$, where $C$ is a smooth plane sextic, is treated in \cite{A-C-G-H} Exercise VI.~B-3.
In case $g=9$, presenting $C$ as plane sextic with a node or cusp, we can argue similarly. 
\end{proof}

\begin{proposition}
\label{lemma:existence:quadric:rank:3}
There exists a locus ${\mathcal Z}\subset\mathcal{M}_{\Lambda}^{0}$ of codimension $\geq2$ with the following property:
When $(X, \iota)$ lies outside ${\mathcal Z}$,
the curve $C$ has a theta characteristic $L$ with $h^0(L)=2$ if and only if 
$C$ can be cut out from $\bar{Y}$ by a quadric $Q\subset\mathbf{P}^{g-1}$ of rank $3$. 
In this case, $L$ is base point free. 
\end{proposition}

\begin{proof}
{\em (Step 1) }
Assume that $C\subset \bar{Y}$ is cut out by a quadric $Q$ of rank $3$.  
The vertex of $Q$, a $(g-4)$-plane, is disjoint from $\bar{Y}$; otherwise $C$ would be singular. 
The pencil of $(g-3)$-planes in $Q$
gives $C$ a theta characteristic $L$ with $h^0(L)\geq2$, which is free because $C$ is disjoint from the vertex. 
We have $h^0(L)=2$, 
for $|L|=|K_C-L|$ is identified with the linear system of hyperplanes of $\mathbf{P}^{g-1}$ containing a $(g-3)$-plane, 
which is a pencil. 
\par{\em (Step 2) }
Conversely, suppose that $C$ has a complete half-canonical pencil $|L|$. 
Choose a basis $\alpha, \beta$ of $H^{0}(L)$ and write $|L|=|L_0|+D_0$ with $|L_0|$ the free part and $D_0$ the fixed part. 
The divisor $D_0$ is defined by $\alpha=\beta=0$. 
If we set 
\begin{equation*}
u_1=\alpha^2, 
\qquad 
u_2=\beta^2, 
\qquad 
u_3=\alpha\beta \; \; \in H^{0}(K_C)=H^{0}(\mathcal{O}_{\mathbf{P}^{g-1}}(1)), 
\end{equation*}
then $C$ is contained in the rank $3$ quadric $Q\subset\mathbf{P}^{g-1}$ defined by $u_1u_2=u_3^2$. 
The vertex of $Q$ is the $(g-4)$-plane $V$ defined (set-theoretically) by $(u_1=u_2=0)\cap Q$. 
The free part $|L_0|$ is given by the pencil of $(g-3)$-planes in $Q$ through $V$, 
and the fixed part $D_0$ is defined by 
\begin{equation}\label{eqn:calculate fixed part 1}
2D_0 = (u_1=u_2=0)|_{C}. 
\end{equation}
In order to show that this quadric $Q$ cuts out $C$ from $\bar{Y}$, 
it suffices to prove $\bar{Y}\not\subset Q$, for then $C$ and $Q|_{\bar{Y}}$ are both $-2K_{\bar{Y}}$-curves on $\bar{Y}$.  
\par{\em (Step 3) }
Suppose the contrary: $\bar{Y}\subset Q$.
We then have three linearly independent sections $\tilde{u}_1, \tilde{u}_2, \tilde{u}_3\in H^0(-K_Y)$ on $Y$ with $\tilde{u}_1\tilde{u}_2=\tilde{u}_3^2$ 
and $u_i=\tilde{u}_i|_{C}$. 
Choose a blow-down $\pi:Y\to\mathbf{P}^2$ 
and let $Z\subset\mathbf{P}^{2}$ be the blown-up points (which possibly contain infinitely near ones). 
We have $\# Z = 10-g$. 
The image $\pi(C)$ of $C$ is a plane sextic having double points at $Z$ and no other singularities. 
Via the mapping 
\begin{equation}\label{eqn:-K-curve and cubic}
|-K_Y| \to |\mathcal{O}_{{\bf P}^{2}}(3)|, \qquad D\mapsto \pi(D), 
\end{equation}
we can identify $|-K_Y|$ with the linear system of plane cubics through $Z$, 
and $D$ is recovered from $\Gamma=\pi(D)$ by $D=\pi^{\ast}\Gamma-\pi^{-1}(Z)$. 
Now by our assumption $\tilde{u}_1\tilde{u}_2=\tilde{u}_3^{2}$, 
the divisors of $\tilde{u}_1$, $\tilde{u}_2$, $\tilde{u}_3$ correspond to 
three linearly independent cubics 
$\Gamma_1$, $\Gamma_2$, $\Gamma_3$ with $\Gamma_1+\Gamma_2=2\Gamma_3$. 
This equality can hold only when 
$$
\Gamma_1=2l_1+l, 
\qquad   
\Gamma_2=2l_2+l, 
\qquad   
\Gamma_3=l_1+l_2+l  
$$
for some distinct lines $l_1, l_2, l$. 
Let $p=l_1\cap l_2$. 
The net spanned by $\Gamma_1, \Gamma_2, \Gamma_3$ consists of splitting cubics $l+l'+l''$ where $l', l''$ are lines through $p$. 
In particular, its base locus is $p\cup l$.  
Recall that the restriction of this net to $C$, after the transformation \eqref{eqn:-K-curve and cubic}, is equal to 
$\mathbf{P}\langle u_1, u_2, u_3\rangle = \mathbf{P}{\rm Sym}^2H^0(L)$. 
\par{\em (Step 4) }
For a plane curve $\Gamma$ we write $\hat{\Gamma}$ for its strict transform in $Y$. 
We also denote $E=\pi^{-1}(p)$. 
We observe the following:  
\begin{enumerate}
\item $Z$ lies on $p\cup l$;
\item $p\in Z$ and $p\notin l$; 
\item the free part $|L_0|$ is given by the projection from $p$;
\item the fixed part $D_0$ is given by 
\begin{equation}\label{eqn:calculate fixed part 2}
2D_0 = (\hat{l}+E)|_{C}. 
\end{equation}
\end{enumerate}  
(1) is obvious. 
We can see (3) by noticing that 
$\phi_{L_0}:C\to\mathbf{P}^1$ coincides with the projection from $p$ after composing them with the conic embedding $\mathbf{P}^1\hookrightarrow \mathbf{P}^2$, 
the composition being the resolution of the rational map defined by ${\rm Sym}^2H^0(L)$. 
Under (2), we have ${\rm div}(\tilde{u}_i)=2\hat{l}_i+E+\hat{l}$ for $i=1, 2$, so that (4) follows from \eqref{eqn:calculate fixed part 1}. 
It remains to see (2).
Firstly, if $p\notin \pi(C)$, then $L_0\sim \pi^{*}\mathcal{O}_{\mathbf{P}^2}(1)|_{C}$ which is absurd. 
If $p\in \pi(C)$ but $p\notin Z={\rm Sing}(\pi(C))$, 
then $L_0\sim \pi^{*}\mathcal{O}_{\mathbf{P}^2}(1)|_{C} - \pi^{-1}(p)$ by (3) and $D_0\ni\pi^{-1}(p)$ by \eqref{eqn:calculate fixed part 1}. 
It follows that $L\geq\pi^{*}{\mathcal O}_{{\bf P}^{2}}(1)|_{C}$, 
and hence ${\dim} |L| \geq 2$, a contradiction. 
Therefore $p\in Z$. 
Next assume $p\in l\cap Z$. 
Then both $\Gamma_1$ and $\Gamma_2$ have multiplicity $3$ at $p$, so that ${\rm div}(\tilde{u}_i)=2\hat{l}_i+2E+\hat{l}$ for $i=1, 2$. 
Then $2D_{0}=(2E+\hat{l})|_{C}$ by \eqref{eqn:calculate fixed part 1}. 
Since $L_{0}\sim \pi^{*}\mathcal{O}_{\mathbf{P}^2}(1)|_{C}-E|_{C}$ by (3), 
we have $L\geq\pi^{*}{\mathcal O}_{{\bf P}^{2}}(1)|_{C}$, the same contradiction as before. This verifies (2).
\par{\em (Step 5) }
Now since the $9-g$ points $Z\backslash p$ lie on $l$ and since no four points of $Z$ can be collinear by the nefness of $-K_{Y}$, 
we must have $|Z\setminus p|\leq3$ and thus $6 \leq g\leq 9$. 
The right hand side of \eqref{eqn:calculate fixed part 2} is divisible by $2$ only if 
$\pi(C)$ is totally tangent to $l$ outside $Z={\rm Sing}(\pi(C))$ and $\pi(C)$ has a cusp at $p$. 
The cusp condition defines a divisor in the moduli space; 
when $7\leq g \leq 9$ (resp.~$g=6$), the tangency condition (resp.~the collinear condition on the three points $Z\setminus p$) 
defines another divisor. 
These two divisors, both Heegner and irreducible, 
defines a codimension $2$ locus ${\mathcal Z}$ in the moduli space.
If $(X,\iota)\not\in{\mathcal Z}$, then \eqref{eqn:calculate fixed part 2} cannot hold and thus $\bar{Y}\not\subset Q$.
By (Step 2), this implies that $C$ is cut out from $\bar{Y}$ by $Q$ when $(X,\iota)\not\in{\mathcal Z}$. 
\end{proof}

Since the intersection of a general Del Pezzo surface $Y\subset{\bf P}^{g-1}$ and a quadric $Q\subset{\bf P}^{g-1}$ of rank $3$ is a smooth curve, 
we see the non-emptiness of $\overline{\mu}_{\Lambda}^{-1}({\frak M}_{g}')$. 

Now let $C\subset \bar{Y}$ be cut out by a quadric $Q$ of rank $3$.  
We take the double cover 
\begin{equation*}
\pi : \tilde{Q} \to \mathbf{P}^{g-1} 
\end{equation*}
branched over $Q$. 
The (contracted) quotient map $\bar{X}\to \bar{Y}$ by $\iota$ 
can be identified with the restriction $\pi^{-1}(\bar{Y})\to \bar{Y}$ of $\pi$. 
We again denote by $\iota$ the covering transformation of $\pi: \tilde{Q} \to \mathbf{P}^{g-1}$. 
We also view $\bar{X}\subset{\bf P}^{g}$ naturally. 
Note that $\tilde{Q}$ is a quadric of rank $4$ in $\mathbf{P}^g$
and hence is the cone over a smooth quadric surface $Q_0\simeq\mathbf{P}^1\times\mathbf{P}^1$ 
with vertex $\mathbf{P}^{g-4}$. 
Let $f\colon \tilde{Q}\dashrightarrow Q_0$ be the projection from the vertex. 
Then the pencils $f^{\ast}|\mathcal{O}_{Q_0}(1, 0)|$ and $f^{\ast}|\mathcal{O}_{Q_0}(0, 1)|$ 
are $\mathbf{P}^1$-families of $(g-2)$-planes that sweep out $\tilde{Q}$. 
As is easily verified, $\iota$ switches these two rulings. 
Restricted to $\bar{X}=\pi^{-1}(\bar{Y})$, the pencil $f^{\ast}|\mathcal{O}_{Q_0}(1, 0)|$ 
induces an elliptic fibration on $X$, say $|E|$. 
The cycle $E+\iota(E)$ is a hyperplane section of $\bar{X}\subset\mathbf{P}^{g}$. 
On the other hand, $X^{\iota}\subset \bar{X}$ is the ramification divisor of $\bar{X}\to \bar{Y}$ 
and hence also cut out by a hyperplane of $\mathbf{P}^g$. 
Therefore $E+\iota(E)$ is linearly equivalent to $X^{\iota}$ on $X$. 
By Proposition~\ref{prop:elliptic:curve:theta:char}, Theorem~\ref{thetanull=Heegner} is proved in case $4\leq g\leq10$. 
\qed

\begin{remark}\label{remark:theta chara via elliptic fibration}
In the above construction, 
the given half-canonical pencil on $C\simeq X^{\iota}$ coincides with the restriction of the elliptic fibration $|E|$: 
this follows because $|E|$ is the restriction of a $\mathbf{P}^{g-2}$-ruling on $\tilde{Q}$, 
which in turn is a component in the pullback of the family of hyperplanes in $\mathbf{P}^{g-1}$ 
that cut out doubly from $Q$ the $\mathbf{P}^{g-3}$-ruling. 
\end{remark}

\begin{remark}\label{remark:uniqueness theta chara k=0}
A general member of $\overline{\mu}_{\Lambda}^{-1}({\frak M}_{g}')$ has a unique effective even theta characteristic, 
because the period of those $(X, \iota)$ having several elliptic curve classes $[E]$ with $E+\iota(E)\sim X^{\iota}$ 
lie in (the image of) the intersection of at least two components of $\mathcal{H}_{\Lambda}$. 
\end{remark}

\subsection{Trigonal curves}\label{sect:trigonal}
\par
This section is preliminaries for the subsequent Sections \ref{ssec:k=1} and \ref{sect:7}.  
A smooth projective curve $C$ of genus $g\geq5$ is said to be \textit{trigonal} 
if it has a degree $3$ morphism to $\mathbf{P}^1$. 
Such a morphism, if exists, is unique up to $\textrm{Aut}(\mathbf{P}^1)$.  
Let us summarize some properties of trigonal curves (cf.~\cite{M-Sc}). 

It is classically known that a trigonal curve $C$ can be canonically embedded in a Hirzebruch surface. 
This is due to the fact that the canonical model of $C$ is contained in a unique rational normal scroll, 
namely the image of a Hirzebruch surface $\mathbf{F}_n$ by a bundle $L_{1,m}$ with $m>0$ (see Section \ref{sect:3.3} for the notation). 
The integer $n$ is called the \textit{scroll invariant}, and $m$ the \textit{Maroni invariant} of $C$. 
As a curve on $\mathbf{F}_n$, $C$ belongs to the linear system $|L_{3,b}|$ with $b=m-n+2$ by the adjunction formula. 
The trigonal map of $C$ is given by the restriction of the projection $\mathbf{F}_n\to\mathbf{P}^1$. 
We have the genus formula $g=3n+2b-2$, which gives the relation of $n$ and $m$. 
These (equivalent) invariants give a stratification of the moduli space of trigonal curves. 
By the canonicity of the embedding $C\subset\mathbf{F}_n$, the isomorphism classes of trigonal curves of genus $g$ and Maroni invariant $m$ 
correspond to the ${\rm Aut}(\mathbf{F}_n)$-orbits in the locus of smooth curves in $|L_{3,b}|$. 

Maroni described the variety $W_d^r(C)\subset{\rm Pic}^d(C)$ of 
line bundles $L$ of degree $d$ and $h^0(L)\geq r+1$ (see \cite[Prop.\,1]{M-Sc}). 
We need his description in the case $d=g-1$.

\begin{proposition}[Maroni]\label{Maroni}
Let $T=L_{0,1}|_C$ be the trigonal bundle and write $W_+=rT+W_{g-1-3r}(C)$, where $W_{g-1-3r}(C):=W_{g-1-3r}^{0}(C)$.
Then 
\begin{equation*}
W_{g-1}^r(C) = W_+ \cup (K_C-W_+). 
\end{equation*} 
\end{proposition}

Using this description, we can give a geometric characterization of trigonal curves having effective theta characteristics.

\begin{lemma}\label{even theta-chara Hirze}
A trigonal curve $C$ of scroll invariant $n$ and Maroni invariant $m$ has 
a theta characteristic $L$ with $h^0(L)\geq r+1$ if and only if 
there exists a curve $H\in|L_{1,m-2r}|$ on $\mathbf{F}_n$ such that 
$H|_C=2D$ for a divisor $D$ of degree $g-1-3r$ on $C$. 
In that case, the divisor $rT+D$ gives such a theta characteristic. 
\end{lemma}

\begin{pf}
A line bundle $L$ is a theta characteristic with $h^0(L)\geq r+1$
if and only if it is a fixed point of the residual involution on $W_{g-1}^r(C)$. 
Such $L$ should be contained in $W_+ \cap (K_C-W_+)$ by Proposition \ref{Maroni}, 
so we can write $L=rT+D$ for some $D\in W_{g-1-3r}(C)$. 
Since $L\sim K_{C}-L$, we have $2D\sim K_C-2rT$. 
The restriction map $|L_{1,m-2r}|\to|K_C-2rT|$ is isomorphic because $h^i(K_{\mathbf{F}_n}-L_{0,2r})=0$ for $i=0, 1$. 
Thus there exists $H\in|L_{1,m-2r}|$ with $H|_C=2D$. 
Conversely, if we have such $H$ and $D$, then $L=rT+D$ is a theta characteristic with $h^{0}(L)\geq r+1$ by Proposition~\ref{Maroni}.
\end{pf}

\subsection{Proof of Proposition \ref{thetanull k=1}}\label{ssec:k=1}
\par
We will show that a general member $C$ of $\overline{\mu}_{\Lambda}(\mathcal{M}_{\Lambda}^{0})$ 
has no effective even theta characteristic 
when $k=1$ and $6\leq g\leq9$, using the description of $C$ given in \cite{Ma}. 

We first consider the case $\delta=1$, where $6\leq g \leq 9$. 
In this case $C$ is a general trigonal curve of genus $g$ with Maroni invariant $2$, 
which can be realized as a general member of the linear system $|L_{3,10-g}|$ on $\mathbf{F}_{g-6}$ 
(\cite{Ma} Corollaries 6.2 and 7.6). 
By Lemma \ref{even theta-chara Hirze}, $C$ has a theta characteristic $L$ with $h^0(L)\geq2$ if and only if 
there exists a curve $H\in|L_{1,0}|$ totally bitangent to $C$. 
However, it is readily seen that a general member of $|L_{3,10-g}|$ has no such tangent curve. 

When $\delta=0$, we have $g=6$. 
In this case $C$ is a plane quintic (\cite{Ma} Corollary 7.3). 
Then our assertion follows from the classical fact that 
any smooth plane quintic has no effective even theta characteristic: 
indeed, according to \cite{A-C-G-H} p.211 we have 
\begin{equation*}
W_5^1(C)=\{ \; \mathcal{O}_C(p-q)\otimes\mathcal{O}_{\mathbf{P}^2}(1) \; | \; p, q\in C \; \}.  
\end{equation*}  
The residual involution acts on this $W_5^1(C)$ by switching $p$ and $q$. 
It has no fixed point other than $\mathcal{O}_{\mathbf{P}^2}(1)|_C$, 
which is an odd theta characteristic. 
Thus Proposition \ref{thetanull k=1} is proved. 
\qed

\section{The locus of vanishing theta-null: the case $\delta=0$}
\label{sect:7}
\par
We continue the geometric study of the Torelli map, 
still viewed as a morphism $\overline{\mu}_{\Lambda}\colon \mathcal{M}_{\Lambda}^{0}\to \frak{M}_{g}$ 
between the moduli spaces. 
In this section we treat the following two series: 
\begin{itemize}
\item $r=2, \: \delta=0 \: \; (g=9, 10)$, 
\item $r=10, \: \delta=0, \: 4\leq g \leq6$. 
\end{itemize}
In these cases, $\overline{\mu}_{\Lambda}(\mathcal{M}_{\Lambda}^{0})$ is contained in $\frak{M}_{g}'$ 
because a general member of $\overline{\mu}_{\Lambda}(\mathcal{M}_{\Lambda}^{0})$ possesses 
a rather apparent effective even theta characteristic. 

For the first series, we will show that 
an analogue of Theorem \ref{thetanull=Heegner} holds by replacing $\frak{M}_{g}'$ with the (reduced) divisor 
\begin{equation*}
\frak{M}_{g}'' :={\rm div}(\Upsilon_{g})\cap{\frak M}'_{g}
\end{equation*}
of $\frak{M}_{g}'$, where $\Upsilon_{g}$ is the Siegel modular form introduced in Section \ref{sect:4.1}. 
Geometrically this locus parametrizes curves having \textit{at least two} effective even theta characteristics. 
(See Lemma~\ref{lemma:locus:vanishing:two:theta:char}.)

\begin{proposition}
\label{thetanull (r,a)=(2,0)}  
When $(r, \delta)=(2, 0)$ and $g=10$, the Heegner divisor 
$\overline{\mathcal H}_{\Lambda}$ is irreducible and equal to $\overline{\mu}_{\Lambda}^{-1}(\frak{M}_{10}'')$. 
In particular, the genus $10$ component of $X^{\iota}$ has a unique effective even theta characteristic 
if the period of $(X,\iota)$ lies outside $\overline{\mathcal H}_{\Lambda}$.
\end{proposition}

In case $(r, l, \delta)=(2, 2, 0)$, 
the Heegner divisor $\overline{\mathcal H}_{\Lambda}$ is reducible, 
reflecting the fact that norm $-4$ vectors $l\in\Lambda$ are divided into two classes 
according to whether ${\rm div}(l)=1$ or $2$. 
We accordingly write 
$\overline{\mathcal H}_{\Lambda} = \overline{\mathcal H}_1 + \overline{\mathcal H}_2$. 
 
\begin{proposition}
\label{thetanull (r,a)=(2,2)} 
When $(r, \delta)=(2, 0)$ and $g=9$, the component $\overline{\mathcal H}_1$ is irreducible 
and equal to $\overline{\mu}_{\Lambda}^{-1}(\frak{M}_{9}'')$.  
In particular, $X^{\iota}$ has a unique effective even theta characteristic if the period of $(X,\iota)$ lies outside $\overline{\mathcal H}_{\Lambda}$. 
\end{proposition}

For the second series, we will prove the following. 

\begin{lemma}
\label{thetanull r=10, k=3,4}
If $(r, \delta)=(10, 0)$ and $g=4, 5$, 
then $\overline{\mu}_{\Lambda}(\mathcal{M}_{\Lambda}^0)$ is not contained in the hyperelliptic locus. 
\end{lemma}

\begin{lemma}
\label{thetanull r=10, k=5}
If $(r, \delta)=(10, 0)$ and $g=6$,   
then a general member of $\overline{\mu}_{\Lambda}(\mathcal{M}_{\Lambda}^0)$ has exactly one effective even theta characteristic. 
\end{lemma}

In Section \ref{sect:9}, these results will be used to prove Theorem~\ref{thm:main:theorem:1} (2), (3). 
Lemmas \ref{thetanull r=10, k=3,4} and \ref{thetanull r=10, k=5} will be strengthened in Corollary~\ref{cor:completion:Prop.7.3:7.4} 
by an argument of modular form (a geometric proof is also possible).

\subsection{Proof of Proposition \ref{thetanull (r,a)=(2,0)}}
\label{ssec:(r,a)=(2,0)}

The line of the proof of Proposition \ref{thetanull (r,a)=(2,0)} (and \ref{thetanull (r,a)=(2,2)}) is similar to Theorem \ref{thetanull=Heegner}. 
We begin by checking the irreducibility of the Heegner divisor $\overline{\mathcal H}_{\Lambda}$, which is defined by norm $-4$ vectors in $\Lambda$.

\begin{lemma}
\label{lemma:irreducibility:char:Heegner:div:(2,0,0)}
When $(r, l, \delta)=(2, 0, 0)$, $\overline{\mathcal H}_{\Lambda}$ is irreducible. 
\end{lemma}

\begin{proof}
Since $\Lambda \simeq \mathbb{U}^{\oplus2}\oplus\mathbb{E}_8^{\oplus2}$ is unimodular, 
by the Eichler criterion the $O^{+}(\Lambda)$-equivalence class of a primitive vector of $\Lambda$ 
is determined by its norm. 
\end{proof}

We describe the members of $\mathcal{M}_{\Lambda}^{0}$ following the construction in Section 6.1 of \cite{Ma}. 
We consider curves on the Hirzebruch surface $\mathbf{F}_4$. 
Let $U\subset|L_{3,0}|$ be the open locus of smooth curves. 
We have a morphism $p\colon U \to \mathcal{M}_{\Lambda}^{0}$ 
by associating to $C\in U$ the double cover of $\mathbf{F}_4$ branched over $C+\Sigma$.

\begin{lemma}\label{lemma:GIT model (2,0,0)}
There exists a geometric quotient $U/{\rm Aut}(\mathbf{F}_4)$ 
and the period map $p$ descends to a biregular isomorphism 
$\mathcal{P}\colon U/{\rm Aut}(\mathbf{F}_4) \to \mathcal{M}_{\Lambda}^{0}$. 
In particular, the Torelli map $\mathcal{M}_{\Lambda}^{0}\to \frak{M}_{g}$ is given by $\mathcal{P}^{-1}$ and is injective, 
with the image the trigonal locus of Maroni invariant $2$. 
\end{lemma} 

\begin{proof}
The $p$-fibers are the ${\rm Aut}(\mathbf{F}_4)$-orbits because 
the ${\rm Aut}(\mathbf{F}_4)$-orbits correspond to the isomorphism classes of 
trigonal curves of Maroni invariant $2$, and the Torelli map recovers these isomorphism classes. 
Since $\mathcal{M}_{\Lambda}^{0}$ is normal, then \cite{GIT} Proposition 0.2 tells that 
the image of $p$ is identified with the geometric quotient of $U$ by ${\rm Aut}(\mathbf{F}_4)$. 
It remains to show the surjectivity of $p$. 

Let $(X, \iota)$ be an arbitrary member of $\mathcal{M}_{\Lambda}^{0}$, 
and let $\{{\frak e},{\frak f}\}$ be the natural hyperbolic basis of its invariant lattice $H_+\simeq \mathbb{U}$.  
The vector $2({\frak e}+{\frak f})$ satisfies the arithmetic conditions in Lemma \ref{HE LB} 
and hence gives an $\iota$-invariant hyperelliptic bundle of degree $8$. 
This defines a generically two-to-one morphism $X\to Y$ where $Y=\mathbf{P}^{2}$ or $\mathbf{F}_{2n}$ with $n\leq2$, 
on which $\iota$ acts by the covering transformation by Lemma \ref{cover trans}. 
Among these possibilities of $Y$, 
the branch $-2K_{Y}$-curve can contain a component of genus $10$ only when $Y=\mathbf{P}^{2}$ or $\mathbf{F}_{4}$. 
The case $Y=\mathbf{P}^{2}$ cannot happen because it would be $(r, l, \delta)=(1, 1, 1)$ in that case. 
Hence $Y=\mathbf{F}_{4}$. 
Since $|-2K_{\mathbf{F}_{4}}|=\Sigma+|L_{3,0}|$ and since $L_{3,0}$ has arithmetic genus $10$, 
the branch curve should be of the form $\Sigma+C$ with smooth $C\in|L_{3,0}|$. 
Thus $(X, \iota)=p(C)$. 
\end{proof}

Let $C$ be a member of $U$. 
Since $C$ is disjoint from $\Sigma$, the bundle $L_{1,-4}|_{C}$ is trivial so that $K_C\simeq L_{1,2}|_C\simeq L_{0,6}|_C$. 
Hence $L_{0,3}|_C$ is a theta characteristic with $h^{0}(L_{0,3}|_C)=4$. 
This shows that $\overline{\mu}_{\Lambda}(\mathcal{M}_{\Lambda}^{0})\subset\frak{M}_{10}'$. 
We are interested in the locus $\overline{\mu}_{\Lambda}(\mathcal{M}_{\Lambda}^{0})\cap \frak{M}_{10}''$ 
where $C$ has another effective even theta characteristic.

\begin{lemma}\label{lem: theta chara criterion (2,0,0)}
The curve $C$ has an effective even theta characteristic different from $L_{0,3}|_C$ 
if and only if
there exists a smooth member $H$ of $|L_{1,0}|$ such that $H|_C=2D$ for some divisor $D$ of degree $6$ on $C$. 
\end{lemma}

\begin{proof}
By Lemma \ref{even theta-chara Hirze}, $C$ has a theta characteristic $L$ with $h^0(L)\geq2$ 
if and only if there exists $H\in|L_{1,0}|$ with $H|_C$ divisible by $2$, in which case $L$ is given by $(H|_C)/2+L_{0,1}|_C$. 
If $H$ is singular, it contains $\Sigma$ as a component and hence can be written as $H=\Sigma+\sum_{i=1}^{4}F_i$ for some $L_{0,1}$-fibers $F_1, \cdots, F_4$. 
Since $H|_C=\sum_{i=1}^{4}F_i|_C$, after renumbering we must have $F_1=F_2$ and $F_3=F_4$. 
Thus $L=L_{0,3}|_C$ in this case. 
\end{proof}

We can now complete the proof of Proposition \ref{thetanull (r,a)=(2,0)}. 
Since $\overline{\mu}_{\Lambda}({\mathcal M}_{\Lambda}^{0})\subset{\frak M}'_{10}$, 
the inverse image $\overline{\mu}_{\Lambda}^{-1}({\frak M}''_{10})$ is a divisor in ${\mathcal M}_{\Lambda}^{0}$. 
It can be easily checked with Lemma \ref{lem: theta chara criterion (2,0,0)} that 
$\overline{\mu}_{\Lambda}({\mathcal M}_{\Lambda}^{0})\cap{\frak M}''_{10}$ is non-empty. 
Hence it remains to prove the inclusion $\overline{\mu}_{\Lambda}^{-1}({\frak M}''_{10})\subset\overline{\mathcal H}_{\Lambda}$. 
Let $(X,\iota)$ be a $2$-elementary $K3$ surface with $(g, k)=(10, 1)$ such that $\overline{\mu}_{\Lambda}(X,\iota)\in{\frak M}''_{10}$. 
By Lemmas \ref{lemma:GIT model (2,0,0)} and \ref{lem: theta chara criterion (2,0,0)}, 
we have $(X, \iota)=p(C)$ for a curve $C$ as in Lemma \ref{lem: theta chara criterion (2,0,0)}. 
Let $H$ be a smooth $L_{1,0}$-curve with $H|_C$ divisible by $2$. 
The pullback of $H$ by the covering map $\pi: X\to\mathbf{F}_4$ splits into two $(-2)$-curves: 
$\pi^{\ast}H=E+\iota(E)$. 
Therefore $X$ has the $\iota$-anti-invariant cycle 
\begin{equation*}
D_{-} = E-\iota(E) 
\end{equation*}
of norm $-16$. 
Let $H_-\subset H^{2}(X,{\bf Z})$ be the anti-invariant lattice of $(X, \iota)$. 
Since 
\begin{equation*}
(D_{-}, H_{-}) = (D_{-}+E+\iota(E), H_{-}) = 2(E, H_{-}) \subset 2\mathbf{Z} 
\end{equation*}
and since $H_{-}$ is unimodular, 
$D_{-}$ is divisible by $2$ in $H_-$. 
Therefore ${\rm Pic}(X)$ contains the anti-invariant cycle $D_{-}/2$ of norm $-4$, 
which implies that the period of $(X, \iota)$ lies in $\overline{{\mathcal H}}_{\Lambda}$. 
This proves Proposition \ref{thetanull (r,a)=(2,0)}. 
\qed

\subsection{Proof of Proposition \ref{thetanull (r,a)=(2,2)}}
\label{ssec:(r,a)=(2,2)}
In this subsection we prove Proposition \ref{thetanull (r,a)=(2,2)}. 
We first explain the decomposition of the Heegner divisor $\overline{\mathcal H}_{\Lambda}$. 
Recall that $\overline{\mathcal H}_{\Lambda}$ is defined by norm $-4$ vectors $l$ in 
$\Lambda \simeq \mathbb{U}\oplus\mathbb{U}(2)\oplus\mathbb{E}_8^{\oplus2}$. 
Since $\Lambda$ cannot contain $\langle -4 \rangle$ as an orthogonal direct summand, 
we have either ${\rm div}(l)=1$ or $2$. 
By the Eichler criterion, each type of norm $-4$ vectors consist of a single $O^{+}_{0}(\Lambda)$-orbit 
(in case ${\rm div}(l)=2$, $[l/2]\in A_{\Lambda}$ is the unique element of norm $\equiv 1$ mod $2\mathbf{Z}$).  
We accordingly obtain the decomposition 
\begin{equation*}
\overline{\mathcal H}_{\Lambda} = \overline{\mathcal H}_1 + \overline{\mathcal H}_2
\end{equation*}
where $\overline{\mathcal H}_{i}$ is defined by those $l$ with ${\rm div}(l)=i$, 
and each $\overline{\mathcal H}_i$ is irreducible. 

We next recall a (well-known) construction of members of $\mathcal{M}_{\Lambda}^0$. 
Let $\tilde{U}$ be the parameter space of smooth $(2, 4)$ complete intersections in $\mathbf{P}^3$. 
For each $C\in\tilde{U}$ the quadric containing it is unique; 
$\tilde{U}$ is thus stratified according to whether the quadric is smooth or a quadratic cone. 
In the first case $C$ is a smooth bidegree $(4, 4)$ curve on $Y=\mathbf{P}^1 \times \mathbf{P}^1$, 
and in the latter case $C$ is a smooth $L_{4,0}$-curve on $Y=\mathbf{F}_2$.  
We have a period map $p:\tilde{U}\to \mathcal{M}_{\Lambda}^0$ by associating to $C$ the double cover of $Y$ branched over $C$.

\begin{lemma}\label{lemma:GIT model (2,2,0)}
The period map $p$ descends to a biregular isomorphism 
$\tilde{U}/{\rm PGL}_4 \to \mathcal{M}_{\Lambda}^0$ from the geometric quotient $\tilde{U}/{\rm PGL}_4$. 
\end{lemma}

\begin{proof}
This is similar to Lemma \ref{lemma:GIT model (2,0,0)}, so we only indicate minimal ingredients of the argument: 
(1) using the natural norm $4$ vector in the invariant lattice $\mathbb{U}(2)$, 
we can realize a given $(X, \iota)$ as a double cover of a quadric so that $p$ is surjective; 
(2) the $p$-fibers are ${\rm PGL}_4$-orbits either
by an argument modeled in Section 4.3 of \cite{Ma} or 
by observing that the ${\rm PGL}_4$-orbits correspond to the isomorphism classes of curves 
(cf.~\cite{A-C-G-H} Exercise IV.~F-2). 
\end{proof}

In the proof of Proposition \ref{thetanull (r,a)=(2,2)} we restrict ourselves to the generic case, namely the smooth quadric case. 
The quadratic cone case can be dealt with similarly. 
So let $Y=\mathbf{P}^1 \times \mathbf{P}^1$ and 
consider the open locus $U\subset|\mathcal{O}_{Y}(4, 4)|$ of smooth bidegree $(4, 4)$ curves. 
If $C\in U$, then $K_C\simeq \mathcal{O}_{Y}(2, 2)|_C$
and the restriction map $|\mathcal{O}_{Y}(2, 2)| \to |K_C|$ is isomorphic. 
In particular, an effective divisor $D$ of degree $8$ on $C$ satisfies $2D \sim K_C$ 
if and only if there exists a bidegree $(2, 2)$ curve $H$ on $Y$ with $H|_C=2D$.
We have one apparent theta characteristic, $\mathcal{O}_{Y}(1, 1)|_C$, which has $h^{0}=4$. 
This is the case where $H$ is a double bidegree $(1, 1)$ curve. 
By \cite{A-C-G-H} Exercise IV.~F-2, any other effective even theta characteristic of $C$, if exist, must satisfy $h^{0}=2$. 

Now let $V\subset U$ be the locus where $C$ has a theta characteristic with $h^0=2$. 
By the same argument as in Section \ref{ssec:(r,a)=(2,0)}, the proof of Proposition \ref{thetanull (r,a)=(2,2)} 
is reduced to showing the inclusion $p(V)\subset \overline{\mathcal H}_1$. 
So let $C\in V$ and $(X, \iota)=p(C)$. 
We write $\pi\colon X\to Y$ for the covering map. 
The curve $C$ admits a $1$-dimensional family $\{ H_{t} \}_{t\in\mathbf{P}^1}$ of 
``totally tangent'' curves of bidegree $(2, 2)$, i.e., $H_t|_C=2D_t$ for some divisors $D_t$ of degree $8$ on $C$, 
which are not double bidegree $(1, 1)$ curves. 
Then $\{ D_{t} \}_{t\in\mathbf{P}^1}$ is a (complete) half-canonical pencil of $C$. 
Note that this pencil can also be obtained by picking up $t=0$ and 
considering the linear system of bidegree $(2, 2)$ curves passing through $D_0$, 
which intersect $C$ at $D_0+D_t$.

\begin{claim}
A general member of $\{ H_{t} \}_{t\in\mathbf{P}^1}$ is smooth. 
\end{claim} 

\begin{proof}
If $H_t$ is reducible, its irreducible components are smooth rational curves intersecting $C$ transversely 
at at most two points and tangent to $C$ elsewhere. 
Their pullback to $X$ split into two $(-2)$-curves. 
So if a general member is reducible, 
then the $K3$ surface $X$ would be covered by rational curves, which is absurd. 
By the same reason, a general (irreducible) member cannot be singular. 
\end{proof}

Let $H$ be a general member of $\{ H_{t} \}_{t\in\mathbf{P}^1}$. Since $H$ is totally tangent to $C$ at $8$ points,
its pullback to $X$ splits into two smooth elliptic curves: 
$\pi^{\ast} H = E + \iota(E)$. 
Hence $(X, \iota)$ possesses the $\iota$-anti-invariant cycle of norm $-16$:
\begin{equation*}
D_{-} = E-\iota(E).
\end{equation*}

\begin{claim}
$D_{-}$ is divisible by $2$ in the anti-invariant lattice $H_-$ 
and satisfies $(D_{-}/2, H_-)=\mathbf{Z}$. 
\end{claim}

\begin{proof}
Since $D_{-}/2 = E - \pi^{\ast}\mathcal{O}_Y(1, 1)$ is contained in $H^{2}(X, \mathbf{Z})$, 
we have $D_{-}/2\in H_{-}$. 
If $(D_{-}/2, H_-)\ne\mathbf{Z}$, then $(D_{-}/2, H_-)\subset 2\mathbf{Z}$ so that 
$D_{-}/4$ would be contained in the dual lattice $H_{-}^{\vee}$. 
Recall that the discriminant forms $A_{H_{+}}$, $A_{H_{-}}$ are isometric to $A_{\mathbb{U}(2)}$.  
As elements of $A_{H_-}$ and $A_{H_+}$, 
$[D_{-}/4]$ and $[\pi^{\ast}\mathcal{O}_Y(1, 1)/2]$ are respectively 
the unique elements of norm $\equiv 1$ mod $2\mathbf{Z}$.  
By Nikulin \cite{Nikulin80a}, then $E/2 = D_{-}/4 + \pi^{\ast}\mathcal{O}_Y(1, 1)/2$ would be contained in $H^{2}(X, \mathbf{Z})$. 
This contradicts the well-known fact that the class of a smooth elliptic curve is primitive in ${\rm Pic}(X)$. 
\end{proof}

To sum up, if a smooth bidegree $(4, 4)$ curve $C\subset Y$ has a theta characteristic with $h^{0}=2$, 
then the associated 2-elementary $K3$ surface $(X, \iota)$ has an anti-invariant cycle $D_{-}/2$ 
of norm $-4$ and with $(D_{-}/2, H_{-})=\mathbf{Z}$ in its Picard lattice. 
Hence the period of $(X, \iota)$ lies in the component $\overline{\mathcal H}_1$ of $\overline{\mathcal H}_{\Lambda}$. 
This finishes the proof of Proposition \ref{thetanull (r,a)=(2,2)}. 
\qed

\subsection{Proof of Lemmas \ref{thetanull r=10, k=3,4} and \ref{thetanull r=10, k=5}}\label{ssec:r=10}
\par
Lemma \ref{thetanull r=10, k=3,4} is an immediate consequence of the following known description of general members of 
$\overline{\mu}_{\Lambda}(\mathcal{M}_{\Lambda}^{0})$. 
When $g=4$, they are general curves in ${\frak M}'_{4}$ by \cite[Cor.\,9.10]{Ma}; 
when $g=5$, they are general trigonal trigonal curves with vanishing thetanull by Kond\=o \cite{Kondo94}. 
Since these curves are not hyperelliptic, Lemma \ref{thetanull r=10, k=3,4} follows. 
\qed
\par
For the proof of Lemma \ref{thetanull r=10, k=5} we use the generic description given in \cite[Cor.\,7.11]{Ma}. 
Let $C$ be a smooth curve on ${\bf F}_{2}$ belonging to the linear system $|L_{3,1}|$ such that 
the $L_{0,1}$-fiber $F$ through the point $C\cap\Sigma$ intersects $C$ with multiplicity $3$ there. 
By taking the resolution of the double cover of $\mathbf{F}_2$ branched over $C+F+\Sigma$, we obtain a 2-elementary $K3$ surface with $(r, l, \delta)=(10, 0, 0)$. 
This construction covers general members of $\mathcal{M}_{\Lambda}^{0}$, 
so a general member of $\overline{\mu}_{\Lambda}(\mathcal{M}_{\Lambda}^{0})$ is a curve $C$ as above. 
\par
Denote $p:=C\cap\Sigma$ and $T:=L_{0,1}|_C$. 
Since 
\begin{equation*}
K_C \simeq L_{1,1}|_C \sim 3T +\Sigma|_C \sim 2T + 4p, 
\end{equation*}
the divisor $T+2p$ gives a theta characteristic of $C$ with $h^0(T+2p)\geq2$. 
By Lemma~\ref{even theta-chara Hirze}, we have $h^0(T+2p)=2$.
Conversely, suppose we have a theta characteristic $L$ on $C$ with $h^0(L)\geq2$. 
By Lemma \ref{even theta-chara Hirze} 
we can find a curve $H\in|L_{1,-1}|$ with $H|_C=2D$ for some divisor $D$ satisfying $L\sim T+D$. 
Since $|L_{1,-1}|=\Sigma+|L_{0,1}|$, 
$H$ is of the form $H=\Sigma+F'$ for some $F'\in|L_{0,1}|$. 
The condition $(\Sigma+F')|_C=2D$ forces $F'$ to pass through $p=\Sigma\cap C$. 
Therefore $F=F'$ and $2D=4p$. Thus $L$ is uniquely determined as $L\sim T+2p$. 
This proves Lemma \ref{thetanull r=10, k=5}.  
\qed

\section{The structure of $\tau_{M}$: the case $\delta=1$}
\label{sect:8}
In Section~\ref{sect:8}, we determine the structure of $\Phi_{M}$ when $\delta=1$.

\subsection{Borcherds products for $2$-elementary lattices}
\label{sect:8.1}
\par
Recall that the Dedekind $\eta$-function and the Jacobi theta series $\theta_{{\Bbb A}_{1}^{+}+\epsilon/2}(\tau)$, $(\epsilon=0,1)$ are
the holomorphic functions on the complex upper half-plane ${\frak H}$
$$
\eta(\tau):=q^{1/24}\prod_{n=1}^{\infty}(1-q^{n}),
\qquad
\theta_{{\Bbb A}_{1}^{+}+\epsilon/2}(\tau):=\sum_{m\in{\bf Z}+\epsilon/2}q^{m^{2}},
\qquad
q:=e^{2\pi i\tau}.
$$
\par
Let $\Lambda\subset{\Bbb L}_{K3}$ be a primitive $2$-elementary sublattice of signature $(2,r(\Lambda)-2)$.
We set
$$
\phi_{\Lambda}(\tau):=\eta(\tau)^{-8}\eta(2\tau)^{8}\eta(4\tau)^{-8}\,\theta_{{\Bbb A}_{1}^{+}}(\tau)^{12-r(\Lambda)},
$$
$$
\psi_{\Lambda}(\tau):=-16\eta(2\tau)^{-16}\eta(4\tau)^{8}\theta_{{\Bbb A}_{1}^{+}+\frac{1}{2}}(\tau)^{12-r(\Lambda)}.
$$
Let $\{{\bf e}_{\gamma}\}_{\gamma\in A_{\Lambda}}$ be the standard basis of the group ring ${\bf C}[A_{\Lambda}]$.
For $0\leq j\leq3$, set ${\bf v}_{j}:=\sum_{\gamma\in A_{\Lambda},\,q_{\Lambda}(\gamma)\equiv j/2}{\bf e}_{\gamma}$.
By \cite[Def.\,7.6, Th.\,7.7]{Yoshikawa13}, the ${\bf C}[A_{\Lambda}]$-valued function
$$
F_{\Lambda}(\tau)
:=
\phi_{\Lambda}(\tau)\,{\bf e}_{0}
+
2^{g(\Lambda)-2}\sum_{j=0}^{3}\sum_{k=0}^{3}\phi_{\Lambda}\left(\frac{\tau+k}{4}\right)i^{-jk}\,{\bf v}_{j}
+
\psi_{\Lambda}(\tau)\,{\bf e}_{{\bf 1}_{\Lambda}}
$$
is a modular form for ${\rm Mp}_{2}({\bf Z})$ of weight $1-b^{-}(\Lambda)/2$ with respect to the Weil representation
$\rho_{\Lambda}\colon{\rm Mp}_{2}({\bf Z})\to{\rm GL}({\bf C}[A_{\Lambda}])$ attached to $\Lambda$ (\cite{Borcherds98}),
where ${\rm Mp}_{2}({\bf Z})$ is the metaplectic cover of ${\rm SL}_{2}({\bf Z})$.
By \cite[Eq.(7.9)]{Yoshikawa13}, the principal part of $F_{\Lambda}$ is given by
\begin{equation}
\label{eqn:principal:part:F}
\begin{aligned}
{\mathcal P}_{\leq0}[F_{\Lambda}]:=
&
\{q^{-1}+2(16-r(\Lambda))\}\,{\bf e}_{0}
+
2^{g(\Lambda)+1}\{16-r(\Lambda)\}\,{\bf v}_{0}
\\
&
+2^{g(\Lambda)}q^{-\frac{1}{4}}{\bf v}_{3}
-
2^{16-r(\Lambda)}q^{\frac{12-r(\Lambda)}{4}}\{1+(28-r(\Lambda))q^{2}\}\,{\bf e}_{{\bf 1}_{\Lambda}}.
\end{aligned}
\end{equation}
By \eqref{eqn:principal:part:F}, we easily see that ${\mathcal P}_{\leq0}[F_{\Lambda}]=0$ if and only if $(r(\Lambda),\delta(\Lambda))=(16,0)$.
\par
Let $\ell\in{\bf Z}_{>0}$ be such that $2^{r(\Lambda)-16}|\ell$ for all $\Lambda$.
Then $\ell\,F_{\Lambda}(\tau)$ has integral Fourier expansion at $+i\infty$.
Define $\Psi_{\Lambda}^{\ell}$ as the Borcherds lift of $\ell\,F_{\Lambda}(\tau)$ (cf. \cite[Th.\,13.3]{Borcherds98}):
$$
\Psi_{\Lambda}^{\ell}:=\Psi_{\Lambda}(\cdot,\ell\,F_{\Lambda}).
$$
If $r(\Lambda)\leq16$, then $\Psi_{\Lambda}(\cdot,F_{\Lambda})$ is well defined and 
$\Psi_{\Lambda}^{\ell}=\Psi_{\Lambda}(\cdot,\ell\,F_{\Lambda})=\Psi_{\Lambda}(\cdot,F_{\Lambda})^{\ell}$ in the ordinary sense.
Since $O(\Lambda)$ (equivalently $O(q_{\Lambda})$) acts trivially on $\ell\,F_{\Lambda}$ by \cite[Th.\,7.7 (2)]{Yoshikawa13}, 
$\Psi_{\Lambda}^{\ell}$ is an automorphic form on $\Omega_{\Lambda}$ for $O^{+}(\Lambda)$ by \cite[Th.\,13.3]{Borcherds98}.
Recall that the divisors ${\mathcal D}_{\Lambda}^{+}$, ${\mathcal D}_{\Lambda}^{-}$ and ${\mathcal H}_{\Lambda}$ 
were introduced in Sections~\ref{sect:2.2} and \ref{sect:2.3}.

\begin{theorem}
\label{thm:Borcherds:lift}
The weight and the divisor of $\Psi_{\Lambda}^{\ell}$ are given as follows:
\begin{itemize}
\item[(1)]
If $r(\Lambda)\leq20$, then
$$
{\rm wt}(\Psi_{\Lambda}^{\ell})
=
\begin{cases}
\begin{array}{ll}
(16-r(\Lambda))(2^{g(\Lambda)}+1)\ell
&
(r(\Lambda)\not=12,20),
\\
(16-r(\Lambda))(2^{g(\Lambda)}+1)\ell-8(1-\delta(\Lambda))\ell
&
(r(\Lambda)=12),
\\
(16-r(\Lambda))(2^{g(\Lambda)}+1)\ell-(28-r(\Lambda))2^{15-r(\Lambda)}(1-\delta(\Lambda))\ell
&
(r(\Lambda)=20),
\end{array}
\end{cases}
$$
$$
{\rm div}(\Psi_{\Lambda}^{\ell})
=
\ell\,
\{
{\mathcal D}_{\Lambda}^{-}
+
(2^{g(\Lambda)}+1)\,{\mathcal D}_{\Lambda}^{+}
-
2^{16-r(\Lambda)}\,{\mathcal H}_{\Lambda}
\}.
$$
\item[(2)]
If $r(\Lambda)=21$, then
$$
{\rm wt}(\Psi_{\Lambda}^{\ell})
=
(16-r(\Lambda))(2^{g(\Lambda)}+1)\ell
=
-5^{3}\cdot 41\ell,
$$
$$
\begin{aligned}
{\rm div}(\Psi_{\Lambda}^{\ell})
&=
\ell\,
[
{\mathcal D}_{\Lambda}^{-}
+
(2^{g(\Lambda)}+1)\,{\mathcal D}_{\Lambda}^{+}
-
2^{16-r(\Lambda)}\{
{\mathcal H}_{\Lambda}
+
(28-r(\Lambda))\,{\mathcal D}_{\Lambda}^{+}\}
]
\\
&=
2^{-5}\ell\cdot
\{
32{\mathcal D}_{\Lambda}^{-}
+
3\cdot17\cdot643\cdot{\mathcal D}_{\Lambda}^{+}
-
{\mathcal H}_{\Lambda}
\}.
\end{aligned}
$$
\end{itemize}
\end{theorem}

\begin{pf}
By using \eqref{eqn:principal:part:F}, the result follows from \cite[Th.\,13.3]{Borcherds98}.
\end{pf}

The Petersson norm $\|\Psi_{\Lambda}^{\ell}\|=\|\Psi_{\Lambda}(\cdot,\ell\,F_{\Lambda})\|$ is an $O(\Lambda)$-invariant function on $\Omega_{\Lambda}$. 
We identify $\|\Psi_{\Lambda}^{\ell}\|$ with the corresponding function on ${\mathcal M}_{\Lambda}$ and set
$$
\|\Psi_{\Lambda}(\cdot,F_{\Lambda})\|
:=
\|\Psi(\cdot,\ell\,F_{\Lambda})\|^{1/\ell}.
$$
Then $\|\Psi_{\Lambda}(\cdot,F_{\Lambda})\|$ is independent of the choice of $\ell\in{\bf Z}_{>0}$ with $2^{r(\Lambda)-16}|\ell$.
If $r(\Lambda)\leq16$, then $\|\Psi_{\Lambda}(\cdot,F_{\Lambda})\|$ is the ordinary Petersson norm of $\Psi_{\Lambda}(\cdot,F_{\Lambda})$.

\subsection{The structure of $\Phi_{M}$: the case $\delta=1$}
\label{sect:8.2}
\par
Write $M_{g,k}$ for a primitive $2$-elementary Lorentzian sublattice of ${\Bbb L}_{K3}$ such that 
$$
(g(M_{g,k}),k(M_{g,k}),\delta(M_{g,k}))=(g,k,1).
$$
Then $M_{g,0}\cong{\Bbb A}_{1}^{+}\oplus{\Bbb A}_{1}^{\oplus10-g}$ and $(r,l,\delta)=(11-g,11-g,1)$ for $M_{g,0}$. Set
$$
\Lambda_{g,k}:=M_{g,k}^{\perp}.
$$

\begin{lemma}
\label{lemma:structure:lattice:g<2:r<10}
There exist mutually perpendicular roots $d_{1},\ldots,d_{k}\in\Delta_{\Lambda_{g,0}}^{+}$ with
$$
\Lambda_{g,0}=\Lambda_{g,k}\oplus{\bf Z}d_{1}\oplus\cdots\oplus{\bf Z}d_{k}.
$$
In particular, if $M_{g,k}$ exists, one has $\Omega_{\Lambda_{g,k}}=\Omega_{\Lambda_{g,k-1}}\cap H_{d_{k}}$.
\end{lemma}

\begin{pf}
The result follows from the classification in Table~\ref{table:list:sublattice:sign(2,r-2):K3} in 
Proposition~\ref{prop:classification:sublattice:sign(2,r-2):K3}.
\end{pf}

By Theorem~\ref{thetanull=Heegner}, 
there exist integers $a_{g},b_{g},c_{g}\in{\bf Z}_{\geq0}$ for $3\leq g\leq10$ with
\begin{equation}
\label{eqn:divisor:pullback:Igusa:-1}
{\rm div}(J_{M_{g,0}}^{*}\chi_{g}^{8})
=
a_{g}\,{\mathcal D}_{\Lambda_{g,0}}^{-}
+
b_{g}\,{\mathcal H}_{\Lambda_{g,0}}
+
c_{g}\,{\mathcal D}_{\Lambda_{g,0}}^{+}.
\end{equation}

\begin{lemma}
\label{eqn:c:(g,0):g<10}
The following inequalities and equality hold:
$$
a_{g}>0,
\qquad
b_{g}>0
\quad
(3\leq g\leq10),
\qquad
c_{g}=0
\qquad
(3\leq g\leq9).
$$
\end{lemma}

\begin{pf}
We get $a_{g}>0$ by \cite[Prop.\,4.2 (2)]{Yoshikawa13} and $b_{g}>0$ by Theorem~\ref{thetanull=Heegner}.
Let $6\leq g\leq9$.
Recall that the Zariski open subset ${\mathcal D}_{\Lambda}^{0,+}$ (resp. ${\mathcal D}_{\Lambda}^{0,-}$) of ${\mathcal D}_{\Lambda}^{+}$
(resp. ${\mathcal D}_{\Lambda}^{-}$) was defined in Section~\ref{sect:2.2}.
By \eqref{eqn:functoriality:Torelli:map} and Proposition~\ref{thetanull k=1}, we get
$J_{M_{g,0}}({\mathcal D}_{\Lambda_{g,0}}^{0,+})=J_{M_{g,1}}(\Omega_{\Lambda_{g,1}}^{0})\not\subset\theta_{{\rm null},g}$,
which implies $c_{g}=0$ for $6\leq g\leq9$.
Let $1\leq g\leq 5$.
By \cite[Prop.\,4.2 (1)]{Yoshikawa13}, $J_{M_{g,g}}(\Omega_{\Lambda_{g,g}}^{0})\not\subset\theta_{{\rm null},g}$.
Since $J_{M_{g,g-1}}({\mathcal D}_{\Lambda_{g,g-1}}^{0,+})=J_{M_{g,g}}(\Omega_{\Lambda_{g,g}}^{0})$
by \eqref{eqn:functoriality:Torelli:map} and Lemma~\ref{lemma:structure:lattice:g<2:r<10}, 
we get $J_{M_{g,g-1}}(\Omega_{\Lambda_{g,g-1}}^{0})\not\subset\theta_{{\rm null},g}$ 
because ${\mathcal D}_{\Lambda_{g,g-1}}^{0,+}\subset\Omega_{\Lambda_{g,g-1}}$.
By \eqref{eqn:functoriality:Torelli:map} and Lemma~\ref{lemma:structure:lattice:g<2:r<10} again, we get
$J_{M_{g,g-2}}({\mathcal D}_{\Lambda_{g,g-2}}^{0,+})=J_{M_{g,g-1}}(\Omega_{\Lambda_{g,g-1}}^{0})\not\subset\theta_{{\rm null},g}$.
In the same way, we get inductively $J_{M_{g,k}}({\mathcal D}_{\Lambda_{g,k}}^{0,+})\not\subset\theta_{{\rm null},g}$ for $k\leq g-1$.
This proves $c_{g}=0$ for $3\leq g\leq5$.
\end{pf}

\begin{proposition}
\label{prop:divisor:pull:back:Igusa}
If $3\leq g\leq10$, the following equality of divisors on $\Omega_{\Lambda_{g,k}}$ holds:
\begin{equation}
\label{eqn:divisor:pullback:Igusa:1}
{\rm div}(J_{M_{g,k}}^{*}\chi_{g}^{8})
=
a_{g}\,{\mathcal D}_{\Lambda_{g,k}}^{-}
+
2^{k}b_{g}\,{\mathcal H}_{\Lambda_{g,k}}
+
c_{g}\,{\mathcal D}_{\Lambda_{g,k}}^{+}.
\end{equation}
\end{proposition}

\begin{pf}
When $g=10$, $M_{g,k}$ makes sense only for $k=0$ by Proposition~\ref{prop:classification:sublattice:sign(2,r-2):K3} .
Hence the assertion is obvious by \eqref{eqn:divisor:pullback:Igusa:-1} when $g=10$. 
We must prove \eqref{eqn:divisor:pullback:Igusa:1} when $g\leq9$.
We will prove by induction the existence of integers $a_{g,k},b_{g,k}\in{\bf Z}$ with
\begin{equation}
\label{eqn:divisor:pullback:Igusa:3}
{\rm div}(J_{M_{g,k}}^{*}\chi_{g}^{8})
=
a_{g,k}\,{\mathcal D}_{\Lambda_{g,k}}^{-}
+
b_{g,k}\,{\mathcal H}_{\Lambda_{g,k}}.
\end{equation}

When $k=0$, the assertion follows from \eqref{eqn:divisor:pullback:Igusa:-1} and Lemma~\ref{eqn:c:(g,0):g<10}. 
We assume \eqref{eqn:divisor:pullback:Igusa:3} for $M_{g,k}$.
By Lemma~\ref{lemma:structure:lattice:g<2:r<10}, there exists $d\in\Delta^{+}_{\Lambda_{g,k}}$ such that $\Lambda_{g,k}=\Lambda_{g,k+1}\oplus{\bf Z}d$.
Then $\Omega_{\Lambda_{g,k+1}}=H_{d}\cap\Omega_{\Lambda_{g,k}}$. 
Let $i\colon\Omega_{\Lambda_{g,k+1}}=H_{d}\cap\Omega_{\Lambda_{g,k}}\hookrightarrow\Omega_{\Lambda_{g,k}}$ be the inclusion.
By the compatibility of Torelli maps \eqref{eqn:functoriality:Torelli:map}, we get the equality of holomorphic maps
\begin{equation}
\label{eqn:functoriality:Torelli:map:sect:8}
J_{M_{g,k+1}}|_{\Omega_{\Lambda_{g,k+1}}^{0}\cup{\mathcal D}_{\Lambda_{g,k+1}}^{0}}
=
J_{M_{g,k}}\circ i|_{\Omega_{\Lambda_{g,k+1}}^{0}\cup{\mathcal D}_{\Lambda_{g,k+1}}^{0}}.
\end{equation}
Hence we get the equality of holomorphic sections on $\Omega_{\Lambda_{g,k+1}}^{0}\cup{\mathcal D}_{\Lambda_{g,k+1}}^{0}$
\begin{equation}
\label{eqn:functoriality:Igusa:cusp:form}
J_{M_{g,k+1}}^{*}\chi_{g}^{8}
=
(J_{M_{g,k}}\circ i)^{*}\chi_{g}^{8}
=
i^{*}J_{M_{g,k}}^{*}\chi_{g}^{8}.
\end{equation}
Since $\Omega_{\Lambda_{g,k+1}}\setminus(\Omega_{\Lambda_{g,k+1}}^{0}\cup{\mathcal D}_{\Lambda_{g,k+1}}^{0})$ has codimension $\geq2$
in $\Omega_{\Lambda_{g,k+1}}$, we deduce from \eqref{eqn:divisor:pullback:Igusa:3} and \eqref{eqn:functoriality:Igusa:cusp:form} 
the equation of divisors on $\Omega_{\Lambda_{g,k+1}}$
\begin{equation}
\label{eqn:divisor:pullback:Igusa:4}
\begin{aligned}
{\rm div}(J_{M_{g,k+1}}^{*}\chi_{g}^{8})
&=
{\rm div}(i^{*}J_{M_{g,k}}^{*}\chi_{g}^{8})
=
i^{*}{\rm div}(J_{M_{g,k}}^{*}\chi_{g}^{8})
=
i^{*}(
a_{g,k}{\mathcal D}_{\Lambda_{g,k}}^{-}
+
b_{g,k}{\mathcal H}_{\Lambda_{g,k}}
)
\\
&=
a_{g,k}{\mathcal D}_{\Lambda_{g,k+1}}^{-}
+
2b_{g,k}{\mathcal H}_{\Lambda_{g,k+1}}.
\end{aligned}
\end{equation}
Here the last equality follows from Propositions~\ref{prop:quasi:pull:back:discriminant:divisor} and \ref{prop:quasi:pull:back:characteristic:divisor}.
This proves \eqref{eqn:divisor:pullback:Igusa:3} for $M_{g,k+1}$. By induction, \eqref{eqn:divisor:pullback:Igusa:3} holds for all $M_{g,k}$.

Since $a_{g,k+1}=a_{g,k}$ and $b_{g,k+1}=2b_{g,k}$ for $k\geq0$ by \eqref{eqn:divisor:pullback:Igusa:3}, \eqref{eqn:divisor:pullback:Igusa:4}, 
we get
\begin{equation}
\label{eqn:recursion:relation}
a_{g,k}=a_{g},
\qquad
b_{g,k}=2^{k}b_{g}
\end{equation}
for $g\leq9$ and $k\geq0$.
This proves the result.
\end{pf}

\begin{proposition}
\label{prop:value:a_{g,1}}
If $3\leq g\leq10$, then the following inequalities hold:
$$
a_{g}\geq2^{2g-1},
\qquad
b_{g}\geq2^{4}.
$$
\end{proposition}

\begin{pf}
The inequality $a_{g}\geq2^{2g-1}$ follows from \cite[Prop.\,4.2 (2)]{Yoshikawa13}.
(To confirm $J_{M}(\Omega_{\Lambda}^{0})\not\subset\theta_{{\rm null},g}$, it is assumed either $r>10$ or $(r,\delta)=(10,1)$ there.
Since the same proof of \cite[Prop.\,4.2 (2)]{Yoshikawa13} works under the same assumption $J_{M}(\Omega_{\Lambda}^{0})\not\subset\theta_{{\rm null},g}$
and since $J_{M_{g,0}}(\Omega_{\Lambda_{g,0}}^{0})\not\subset\theta_{{\rm null},g}$ for $3\leq g\leq10$ by Theorem~\ref{thetanull=Heegner},
the same conclusion as in \cite[Prop.\,4.2 (2)]{Yoshikawa13} still holds for $M_{g,0}$.)
Let us prove the second inequality.
\par
Let $j\colon{\frak M}_{g}\ni C\to{\rm Jac}(C)\in{\mathcal A}_{g}$ be the Torelli map.
By \cite[Prop.\,3.1]{TeixidoriBigas88} or \cite[p.542 Proof of Th.\,1]{Tsuyumine91}, we get the equality of divisors on ${\frak M}_{g}$
\begin{equation}
\label{eqn:divisor:thetanull:TeixidoriBigas}
{\rm div}(j^{*}\chi_{g})|_{{\frak M}_{g}}
=
2\,{\frak M}'_{g}.
\end{equation}
\par
Set $\Lambda:=\Lambda_{g,0}$ and ${\mu}_{\Lambda}^{0}:=\overline{\mu}_{\Lambda}\circ\varPi_{\Lambda}|_{\Omega_{\Lambda}^{0}}$,
where $\varPi_{\Lambda}\colon\Omega_{\Lambda}\to{\mathcal M}_{\Lambda}$ is the projection.
Since $J_{M_{g,0}}|_{\Omega_{\Lambda}^{0}}=j\circ{\mu}_{\Lambda}^{0}$, 
we get by \eqref{eqn:divisor:thetanull:TeixidoriBigas} the following equality of divisors on $\Omega_{\Lambda}^{0}$
\begin{equation}
\label{eqn:divisor:pullback:Igusa:5}
{\rm div}(J_{M_{g,0}}^{*}\chi_{g}^{8})|_{\Omega_{\Lambda}^{0}}
=
({\mu}_{\Lambda}^{0})^{*}{\rm div}(j^{*}\chi_{g}^{8})|_{{\frak M}_{g}}
=
2^{4}({\mu}_{\Lambda}^{0})^{*}{\frak M}'_{g}.
\end{equation}
Since ${\rm Supp}(({\mu}_{\Lambda}^{0})^{*}{\frak M}'_{g})\subset{\mathcal H}_{\Lambda}\cap\Omega_{\Lambda}^{0}$ 
and since $\overline{\mathcal H}_{\Lambda}$ is irreducible by Theorem~\ref{thetanull=Heegner},
there exists $\beta_{g}\in{\bf Z}_{>0}$ such that the following equality of divisors on $\Omega_{\Lambda}^{0}$ holds
\begin{equation}
\label{eqn:pullbacl:mu:Lambda:2}
({\mu}_{\Lambda}^{0})^{*}\overline{\frak M}'_{g}
=
\beta_{g}\cdot{\mathcal H}_{\Lambda}|_{\Omega_{\Lambda}^{0}}.
\end{equation}
By \eqref{eqn:divisor:pullback:Igusa:5}, \eqref{eqn:pullbacl:mu:Lambda:2}, we get the equality of divisors on $\Omega_{\Lambda}^{0}$
\begin{equation}
\label{eqn:divisor:pullback:Igusa:6}
{\rm div}(J_{\Lambda}^{*}\chi_{g}^{8})|_{\Omega_{\Lambda}^{0}}
=
2^{4}\beta_{g}\cdot{\mathcal H}_{\Lambda}|_{\Omega_{\Lambda}^{0}}.
\end{equation}
Comparing \eqref{eqn:divisor:pullback:Igusa:-1} with \eqref{eqn:divisor:pullback:Igusa:6}, 
we get $b_{g}=2^{4}\beta_{g}\geq2^{4}$.
This completes the proof.
\end{pf}

\begin{theorem}
\label{thm:structure:g<6:r<10}
There exists a constant $C_{g,k,\ell}$ depending only on $g$, $k$, $\ell$ 
such that the following equality of automorphic forms on $\Omega_{\Lambda_{g,k}}$ holds:
$$
\Phi_{M_{g,k}}^{2^{g-1}(2^{g}+1)}
=
C_{g,k,\ell}\,
\Psi_{\Lambda_{g,k}}^{2^{g-1}\ell}
\otimes
J_{M_{g,k}}^{*}\chi_{g}^{8\ell}.
$$
In particular, there exists a constant $C_{g,k}$ depending only on $g$, $k$ such that
$$
\tau_{M_{g,k}}^{-2^{g}(2^{g}+1)}
=
C_{g,k}\,
\left\|
\Psi_{\Lambda_{g,k}}(\cdot,2^{g-1}F_{\Lambda_{g,k}})
\right\|
\cdot 
J_{M_{g,k}}^{*}
\left\|
\chi_{g}^{8}
\right\|.
$$
\end{theorem}

\begin{pf}
Since the result was proved in \cite[Th.\,9.1]{Yoshikawa13} in the case $g\leq2$, we assume $g\geq3$ in what follows.
Set $M:=M_{g,k}$ and $\Lambda:=\Lambda_{g,k}$. 
Then $r(\Lambda)\geq5$ and hence ${\mathcal M}_{\Lambda}^{*}\setminus{\mathcal M}_{\Lambda}$ has codimension $\geq2$ in ${\mathcal M}^{*}$. 
By definition, $M=M_{g,k}$ has invariants $r(M)=11+k-g$, $l(M)=11-k-g$, so that $r(\Lambda)=11-k+g$ and $l(\Lambda)=11-k-g$.
\par
Since ${\rm wt}(\chi_{g})=2^{g-2}(2^{g}+1)$ and ${\rm wt}(\Psi_{\Lambda}^{\ell})=\{16-r(\Lambda)\}(2^{g}+1)\ell$ by Theorem~\ref{thm:Borcherds:lift}, 
we get
\begin{equation}
\label{eqn:weight:product:1}
{\rm wt}(
\Psi_{\Lambda}^{2^{g-1}\ell}
\otimes
J_{M}^{*}\chi_{g}^{8\ell}
)
=
2^{g-1}(2^{g}+1)\ell\cdot(16-r(\Lambda),4).
\end{equation}
Since ${\rm wt}(\Phi_{\Lambda})=\ell(16-r(\Lambda),4)$ by Theorem~\ref{thm:automorphic:property:Phi:M}, we get by \eqref{eqn:weight:product:1}
\begin{equation}
\label{eqn:weight:product:2}
{\rm wt}(
(\Psi_{\Lambda}^{2^{g-1}}
\otimes
J_{M}^{*}\chi_{g}^{8})^{\ell}
/
\Phi_{\Lambda}^{2^{g-1}(2^{g}+1)}
)
=
(0,0).
\end{equation}
By \eqref{eqn:weight:product:2},  
$$
\varphi_{\Lambda}
:=
(\Psi_{\Lambda}^{2^{g-1}\ell}\otimes J_{M}^{*}\chi_{g}^{8\ell})/\Phi_{\Lambda}^{2^{g-1}(2^{g}+1)}.
$$
descends to a meromorphic function on ${\mathcal M}_{\Lambda}$. Since ${\mathcal M}_{\Lambda}^{*}$ is normal and 
since $\dim{\mathcal M}_{\Lambda}^{*}\setminus{\mathcal M}_{\Lambda}\leq\dim{\mathcal M}_{\Lambda}^{*}-2$, 
$\varphi_{\Lambda}$ extends to a meromorphic function on ${\mathcal M}_{\Lambda}^{*}$.
\par{\bf (1) }
Let $r(M)\geq2$. Hence $r(\Lambda)\leq20$. 
By Theorem~\ref{thm:Borcherds:lift} and Proposition~\ref{prop:divisor:pull:back:Igusa}, 
\begin{equation}
\label{eqn:divisor:product:1}
\begin{aligned}
{\rm div}(
\Psi_{\Lambda}^{2^{g-1}}
\otimes
J_{M}^{*}\chi_{g}^{8}
)
&=
2^{g-1}
\{
{\mathcal D}_{\Lambda}^{-}
+
(2^{g}+1){\mathcal D}_{\Lambda}^{+}
-
2^{16-r(\Lambda)}{\mathcal H}_{\Lambda}
\}
+
a_{g}{\mathcal D}_{\Lambda}^{-}
+
2^{k}b_{g}{\mathcal H}_{\Lambda}
\\
&=
(2^{g-1}+a_{g}){\mathcal D}_{\Lambda}^{-}
+
2^{g-1}(2^{g}+1){\mathcal D}_{\Lambda}^{+}
+
2^{k}(b_{g}-2^{4}){\mathcal H}_{\Lambda}.
\end{aligned}
\end{equation}
Since ${\rm div}(\Phi_{\Lambda})=\ell\,{\mathcal D}_{\Lambda}$,
we deduce from \eqref{eqn:divisor:product:1} that
\begin{equation}
\label{eqn:divisor:product:2}
{\rm div}
(
\varphi_{\Lambda}
)
=
(a_{g}-2^{2g-1})\ell\,{\mathcal D}_{\Lambda}^{-}
+
2^{k}(b_{g}-2^{4})\ell\,{\mathcal H}_{\Lambda}.
\end{equation}
Since $a_{g}\geq 2^{2g-1}$ and $b_{g}\geq2^{4}$ by Proposition~\ref{prop:value:a_{g,1}},
the divisor of $\varphi_{\Lambda}$ is effective by \eqref{eqn:divisor:product:2}.
Since $\varphi_{\Lambda}$ is a {\em holomorphic} function on ${\mathcal M}_{\Lambda}^{*}$,
$\varphi_{\Lambda}$ must be a constant function on ${\mathcal M}_{\Lambda}^{*}$, which implies that
\begin{equation}
\label{eqn:value:b_g}
a_{g}=2^{2g-1},
\qquad
b_{g}=2^{4}.
\end{equation}
This completes the proof when $r(M)\geq 2$.
\par{\bf (2) }
Assume $r(M)=1$. Then $M={\Bbb A}_{1}^{+}$ and $r(\Lambda)=21$, $g=10$, $k=0$, $\delta=1$.
By Theorem~\ref{thm:Borcherds:lift} (2), Proposition~\ref{prop:divisor:pull:back:Igusa} and the equality $(g-1)+16-r(\Lambda)=4$, we get
\begin{equation}
\label{eqn:divisor:product:1:+:rk1}
\begin{aligned}
\,&
{\rm div}
(
\Psi_{\Lambda}^{2^{g-1}\ell}
\otimes
J_{M}^{*}\chi_{g}^{8\ell}
)
\\
&=
2^{g-1}\ell
\{
{\mathcal D}_{\Lambda}^{-}
+
(2^{g}+1)\,{\mathcal D}_{\Lambda}^{+}
-
2^{16-r(\Lambda)}\,{\mathcal H}_{\Lambda}
-
2^{16-r(\Lambda)}(28-r(\Lambda))\,{\mathcal D}_{\Lambda}^{+}
\}
\\
&\quad
+
\ell\,
\{
a_{g}\,{\mathcal D}_{\Lambda}^{-}
+
b_{g}\,{\mathcal H}_{\Lambda}
+
c_{g}\,{\mathcal D}_{\Lambda}^{+}
\}
\\
&=
\ell
\{(2^{g-1}+a_{g})\,{\mathcal D}_{\Lambda}^{-}
+
(b_{g}-2^{4}){\mathcal H}_{\Lambda}
\}
+
\ell\,\{(2^{2g-1}+2^{g-1}+c_{g})-2^{4}\cdot7)\}{\mathcal D}_{\Lambda}^{+}.
\end{aligned}
\end{equation}
Since ${\rm div}(\Phi_{\Lambda})=\ell\,{\mathcal D}_{\Lambda}$, we get by \eqref{eqn:divisor:product:1:+:rk1}
\begin{equation}
\label{eqn:divisor:product:2:+:rk1}
\begin{aligned}
{\rm div}(\varphi_{\Lambda})
&=
\ell\,
\{
(a_{g}-2^{2g-1})\,{\mathcal D}_{\Lambda}^{-}
+
(b_{g}-2^{4})\,{\mathcal H}_{\Lambda}
\}
+
\ell(c_{g}-2^{4}\cdot7)\,{\mathcal D}_{\Lambda}^{+}.
\end{aligned}
\end{equation}
Since $a_{g}\geq 2^{2g-1}$ and $b_{g}\geq2^{4}$ by Proposition~\ref{prop:value:a_{g,1}},
it follows from \eqref{eqn:divisor:product:2:+:rk1} that $\varphi_{\Lambda}$ is a non-zero {\em holomorphic} function on 
${\mathcal M}_{\Lambda}\setminus\overline{\mathcal D}_{\Lambda}^{+}$.
By Lemma~\ref{cor:meromorphic:function:modular:variety:A1} below, $\varphi_{\Lambda}$ is a non-zero constant.
This completes the proof of Theorem~\ref{thm:structure:g<6:r<10}.
\end{pf}

\begin{lemma}
\label{cor:meromorphic:function:modular:variety:A1}
When $M\simeq{\Bbb A}_{1}^{+}$, 
any holomorphic function on ${\mathcal M}_{\Lambda}\setminus\overline{\mathcal D}_{\Lambda}^{+}$ 
is a constant. 
\end{lemma}

\begin{proof}
Let $U\subset |\mathcal{O}_{\mathbf{P}^2}(6)|$ be the space of smooth plane sextics, 
and let $V\subset |\mathcal{O}_{\mathbf{P}^2}(6)|$ be that of sextics with at most one node. 
By the stability criterion for plane sextics \cite{Shah80}, we have a geometric quotient $V/{\rm PGL}_3$ of $V$ by ${\rm PGL}_3$,
which contains $U/{\rm PGL}_3$ as an open set. 
It is well-known that ${\mathcal M}_{\Lambda}^{0}$ is isomorphic to $U/{\rm PGL}_3$ 
by associating to a smooth plane sextic the double covers of $\mathbf{P}^2$ branched over it. 
Shah \cite{Shah80} has shown that 
this isomorphism extends to an open embedding $V/{\rm PGL}_3 \hookrightarrow \mathcal{M}_{\Lambda}$ 
and that its image is contained in ${\mathcal M}_{\Lambda}\setminus\overline{\mathcal D}_{\Lambda}^{+}$. 
Hence a holomorphic function on ${\mathcal M}_{\Lambda}\setminus\overline{\mathcal D}_{\Lambda}^{+}$ 
gives that on $V/{\rm PGL}_3$, which in turn is pulled-back to $V$. 
Since the complement of $V$ in $|\mathcal{O}_{\mathbf{P}^2}(6)|$ is of codimension $2$ in $|\mathcal{O}_{\mathbf{P}^2}(6)|$, 
a holomorphic function on $V$ extends to $|\mathcal{O}_{\mathbf{P}^2}(6)|$ holomorphically 
and so is a constant. 
\end{proof}

\begin{remark}
\label{remark:thetanull:(g,k)=(10,1)}
Let $(g,k)=(10,0)$. Since $\varphi_{\Lambda}$ is constant when $\Lambda=({\Bbb A}_{1}^{+})^{\perp}$, we get  
\begin{equation}
\label{eqn:c:(10,0)}
a_{10}=2^{2g-1}=2^{19},
\qquad
b_{10}=2^{4},
\qquad
c_{10}=2^{4}\cdot7
\end{equation}
by \eqref{eqn:divisor:product:2:+:rk1}. In particular, $J_{{\Bbb A}_{1}^{+}}^{*}\chi_{10}$ vanishes on ${\mathcal D}_{\Lambda}^{+}$.
Since ${\Bbb U}=[{\Bbb A}_{1}^{+}\oplus{\bf Z}d]$ for any $d\in\Delta_{\Lambda}^{+}$, this, together with \eqref{eqn:functoriality:Torelli:map}, 
implies that $J_{\Bbb U}^{*}\chi_{10}$ vanishes identically on $\Omega_{{\Bbb U}^{\perp}}$.
\end{remark}

\section
{The structure of $\tau_{M}$: the case $\delta=0$}
\label{sect:9}
\par
In Section~\ref{sect:9}, $M\subset{\Bbb L}_{K3}$ is assumed to be a primitive $2$-elementary Lorentzian sublattice with $\delta=0$. 
As before, we set $\Lambda=M^{\perp}$.

\subsection{The structure of $\Phi_{M}$: the case $r\not=2,10$ and $\delta=0$}
\label{sect:9.1}
\par

\begin{lemma}
\label{lemma:div:Igusa:(gk,delta)=(6,1,0)}
Let $\Lambda\cong{\Bbb U}\oplus{\Bbb U}(k)\oplus{\Bbb D}_{4}\oplus{\Bbb E}_{8}$ with $k=1,2$.
Then the following equality of divisors on $\Omega_{\Lambda}$ holds:
$$
{\rm div}(J_{M}^{*}\chi_{g}^{8})
=
2^{g-1}(2^{g}+1)\,{\mathcal D}_{\Lambda}.
$$
\end{lemma}

\begin{pf}
Recall that the lattice $\Lambda_{g,k}$ was defined in Section~\ref{sect:8}.
Since $\Lambda_{g,k-1}\cong\Lambda\oplus{\Bbb A}_{1}$ by comparing the invariants $(r,l,\delta)$,
there is a root $d\in\Delta_{\Lambda_{g,k-1}}^{+}$ with $\Lambda_{g,k-1}\cap d^{\perp}\cong\Lambda$.
Since $d\in\Delta_{\Lambda_{g,k-1}}^{+}$, we get by \eqref{eqn:divisor:pullback:Igusa:1}, \eqref{eqn:recursion:relation}, \eqref{eqn:value:b_g}
and Propositions~\ref{prop:quasi:pull:back:discriminant:divisor} and \ref{prop:quasi:pull:back:characteristic:divisor}
$$
\begin{aligned}
{\rm div}(J_{M}^{*}\chi_{g}^{8})
&=
{\rm div}(J_{M_{g,k-1}}^{*}\chi_{g}^{8})|_{\Omega_{\Lambda}}
=
\{
2^{2g-1}{\mathcal D}_{\Lambda_{g,k-1}}^{-}
+
2^{k+3}{\mathcal H}_{\Lambda_{g,k-1}}
\}|_{\Omega_{\Lambda}}
\\
&=
2^{2g-1}{\mathcal D}_{\Lambda}^{-}
+
2^{k+4}{\mathcal H}_{\Lambda}
=
2^{g-1}(2^{g}+1){\mathcal D}_{\Lambda}.
\end{aligned}
$$
To get the last equality, we used $g-1=k+4$, ${\mathcal D}_{\Lambda}^{+}=0$ and
${\mathcal H}_{\Lambda}={\mathcal D}_{\Lambda}$ for $\Lambda$, 
where the last two equalities follow from $\varepsilon_{\Lambda}=-2$, $\delta(\Lambda)=0$ and ${\bf 1}_{\Lambda}=0$.
\end{pf}

\begin{theorem}
\label{thm:structure:r=6:delta=0}
If $r\not=2,10$ and $\delta=0$, there is a constant $C_{M,\ell}>0$ depending only on $M$ and $\ell$ 
such that the following equality of automorphic forms on $\Omega_{\Lambda}$ holds
$$
\Phi_{M}^{2^{g-1}(2^{g}+1)}
=
C_{M,\ell}\,
\Psi_{\Lambda}^{2^{g-1}\ell}
\otimes
J_{M}^{*}\chi_{g}^{8\ell}.
$$
In particular, there is a constant $C_{M}>0$ depending only on $M$ 
such that the following equality of automorphic forms on $\Omega_{\Lambda}$ holds:
$$
\tau_{M}^{-2^{g}(2^{g}+1)}
=
C_{M}\,
\left\|
\Psi_{\Lambda}(\cdot,2^{g-1}F_{\Lambda})
\right\|
\cdot 
J_{M}^{*}
\left\|
\chi_{g}^{8}
\right\|.
$$
\end{theorem}

\begin{pf}
When $r>10$, the result was proved in \cite[Th.\,9.1]{Yoshikawa13}. We may assume $2<r<10$ and $\delta=0$.
By Proposition~\ref{prop:classification:sublattice:sign(2,r-2):K3}, we get $r=6$
and $\Lambda\cong{\Bbb U}\oplus{\Bbb U}(k)\oplus{\Bbb D}_{4}\oplus{\Bbb E}_{8}$, $k=1,2$.
Since $\delta(\Lambda)=0$, we have ${\mathcal D}_{\Lambda}^{+}=0$ and ${\bf 1}_{\Lambda}=0$. 
Since $r(\Lambda)=16$, we have $\varepsilon_{\Lambda}=-2$, which, together with ${\bf 1}_{\Lambda}=0$, yields that 
${\mathcal H}_{\Lambda}={\mathcal D}_{\Lambda}$. 
By Theorem~\ref{thm:Borcherds:lift}, we get ${\rm wt}(\Psi_{\Lambda}^{\ell})=0$ and
$$
{\rm div}(\Psi_{\Lambda}^{\ell})
=
\ell\,({\mathcal D}_{\Lambda}^{-}-2^{16-r(\Lambda)}{\mathcal H}_{\Lambda})
=
\ell\,({\mathcal D}_{\Lambda}-{\mathcal D}_{\Lambda})
=
0,
$$
which implies that $\Psi_{\Lambda}$ is a non-zero constant function on $\Omega_{\Lambda}$. 

Set
$\varphi_{\Lambda}:=\Psi_{\Lambda}^{2^{g-1}\ell}\otimes J_{M}^{*}\chi_{g}^{8\ell}/\Phi_{M}^{2^{g-1}(2^{g}+1)}$.
Since $\Psi_{\Lambda}$ is a non-zero constant,
we deduce from Lemma~\ref{lemma:div:Igusa:(gk,delta)=(6,1,0)} and
${\rm wt}(\Phi_{M})=(0,4\ell)$, ${\rm div}(\Phi_{M})=\ell\,{\mathcal D}_{\Lambda}$ that
\begin{equation}
\label{eqn:weight:product:3}
{\rm wt}(\varphi_{\Lambda})=(0,0),
\qquad
{\rm div}(\varphi_{\Lambda})=0.
\end{equation}
Hence $\varphi_{\Lambda}$ is a {\em holomorphic} function on ${\mathcal M}_{\Lambda}$ by \eqref{eqn:weight:product:3}
and extends holomorphically to ${\mathcal M}_{\Lambda}^{*}$.
Thus $\varphi_{\Lambda}$ is a non-zero constant.
\end{pf}

\subsection
{The structure of $\Phi_{M}$: the case $(r,\delta)=(10,0)$}
\label{sect:9.2}
\par
In this subsection, we assume that $M$ is non-exceptional and
$$
(r,\delta)=(10,0).
$$ 
Then $0\leq l\leq 8$ and $2\leq g\leq6$.
Since $J_{M}^{*}\chi_{g}$ vanishes identically on $\Omega_{\Lambda}$ (e.g. \cite[Prop.\,9.3]{Yoshikawa13}), 
Theorem~\ref{thm:structure:g<6:r<10} does not hold in this case.
Identify ${\frak M}_{g}$ with its image by the Torelli map $j\colon{\frak M}_{g}\hookrightarrow{\mathcal A}_{g}$.
Then $J_{M}\colon\Omega_{\Lambda}^{0}\to{\mathcal A}_{g}$ is identified with the map
$\mu_{\Lambda}=j^{-1}\circ J_{M}\colon\Omega_{\Lambda}^{0}\to{\frak M}_{g}$.
Write ${\frak H}_{{\rm hyp},g}\subset{\frak M}_{g}$ for the hyperelliptic locus.

\begin{proposition}
\label{prop:zero:Upsilon}
Let $M$ be non-exceptional with $(r,\delta)=(10,0)$. Then $J_{M}^{*}\Upsilon_{g}$ does not vanish identically on $\Omega_{\Lambda}^{0}$.
Moreover, for any $d\in\Delta_{\Lambda}$, $J_{[M\perp d]}^{*}\chi_{g-1}$ is nowhere vanishing on $\Omega_{\Lambda\cap d^{\perp}}^{0}$.
\end{proposition}

\begin{pf}
For the first assertion,
it suffices to prove $\mu_{\Lambda}(\Omega_{\Lambda}^{0})\not\subset{\rm div}(\Upsilon_{g})\cap{\frak M}_{g}$.
Since $\Upsilon_{2}$ is nowhere vanishing on the diagonal locus of ${\frak S}_{2}$ 
and since $\mu_{\Lambda}(\Omega_{\Lambda}^{0})$ is the image of the diagonal locus by the projection ${\frak S}_{2}\to{\mathcal A}_{2}$,
we get ${\rm div}(\Upsilon_{2})\cap\mu_{\Lambda}(\Omega_{\Lambda}^{0})=\emptyset$.
Similarly, we have ${\rm div}(\Upsilon_{3})\cap\mu_{\Lambda}(\Omega_{\Lambda}^{0})=\emptyset$ by \cite[Lemma 11]{Igusa67}.
Let $g=4$. 
Since $\mu_{\Lambda}(\Omega_{\Lambda}^{0})\subset{\rm div}(\chi_{4})$, 
the inclusion $\mu_{\Lambda}(\Omega_{\Lambda}^{0})\subset{\rm div}(\Upsilon_{4})\cap{\frak M}_{4}$ would imply
$\mu_{\Lambda}(\Omega_{\Lambda}^{0})\subset{\rm div}(\chi_{4})\cap{\rm div}(\Upsilon_{4})\cap{\frak M}_{4}$.
Since the right hand side coincides with ${\frak H}_{\rm hyp,4}$ by \cite[p.544 Cor.]{Igusa82},
this last inclusion contradicts Lemma~\ref{thetanull r=10, k=3,4}. Thus $\mu_{\Lambda}(\Omega_{\Lambda}^{0})\not\subset{\rm div}(\Upsilon_{4})\cap{\frak M}_{4}$.
Let $g=5$. 
Let $F_{5}$ be the Schottky form in genus $5$ (cf. \cite[p.1018]{GrushevskySalvatiManni11}), 
whose zero divisor characterizes the (closure of) trigonal locus of ${\frak M}_{5}$ (cf. \cite[Cor.\,18]{GrushevskySalvatiManni11}).
By \cite{Kondo94}, a general point of $\mu_{\Lambda}(\Omega_{\Lambda}^{0})$ is contained in the intersection of the thetanull divisor and
the trigonal locus.
Then the inclusion $\mu_{\Lambda}(\Omega_{\Lambda}^{0})\subset{\rm div}(\Upsilon_{5})\cap{\frak M}_{5}$ would imply
$$
\mu_{\Lambda}(\Omega_{\Lambda}^{0})
\subset
{\rm div}(F_{5})\cap{\rm div}(\chi_{5})\cap{\rm div}(\Upsilon_{5})\cap{\frak M}_{5}.
$$
Since the right hand side coincides with ${\frak H}_{\rm hyp,5}$ by \cite[p.67]{FontanariPascolutti12},
this last inclusion contradicts Lemma~\ref{thetanull r=10, k=3,4}. 
Thus $\mu_{\Lambda}(\Omega_{\Lambda}^{0})\not\subset{\rm div}(\Upsilon_{5})\cap{\frak M}_{5}$.
When $g=6$, we get $\mu_{\Lambda}(\Omega_{\Lambda}^{0})\not\subset{\rm div}(\Upsilon_{6})\cap{\frak M}_{6}$
by Lemma~\ref{lemma:locus:vanishing:two:theta:char} and Lemma~\ref{thetanull r=10, k=5}. This proves the first assertion. 
Since $r([M\perp d])>10$, the second assertion follows from \cite[Prop.\,4.2 (1)]{Yoshikawa13}.
This completes the proof.
\end{pf}

\begin{theorem}
\label{thm:structure:r=10:delta=0}
Let $M$ be non-exceptional with $(r,\delta)=(10,0)$.
Then there is a constant $C_{M,\ell}>0$ depending only on $M$ and $\ell$ 
such that the following equality of automorphic forms on $\Omega_{\Lambda}$ holds
$$
\Phi_{M}^{(2^{g-1}+1)(2^{g}-1)}
=
C_{M,\ell}\,
\Psi_{\Lambda}^{(2^{g-1}+1)\ell}
\otimes
J_{M}^{*}\Upsilon_{g}^{\ell}.
$$
In particular, there is a constant $C_{M}>0$ depending only on $M$ such that
$$
\tau_{M}^{-(2^{g}+2)(2^{g}-1)}
=
C_{M}\,
\left\|
\Psi_{\Lambda}(\cdot,(2^{g-1}+1)\,F_{\Lambda})
\right\|
\cdot 
J_{M}^{*}
\left\|
\Upsilon_{g}
\right\|.
$$
\end{theorem}

\begin{pf}
Since $(r,\delta)=(10,0)$, we get $(r(\Lambda),\delta(\Lambda))=(12,0)$.
By \cite[Th.\,8.1]{Yoshikawa13}, 
\begin{equation}
\label{eqn:weight:zero:Psi:r=10:delta=0}
{\rm wt}(\Psi_{\Lambda})=(4(2^{g}-1),0),
\qquad
{\rm div}(\Psi_{\Lambda})
=
{\mathcal D}_{\Lambda}.
\end{equation}
Since $J_{M}^{*}\Upsilon_{g}$ does not vanish identically on $\Omega_{\Lambda}^{0}$ by Proposition~\ref{prop:zero:Upsilon}
and since $\overline{\mathcal D}_{\Lambda}$ is irreducible, 
there exists $a\in{\bf Z}_{\geq0}$ and an effective divisor ${\mathcal E}_{\Lambda}$ on $\Omega_{\Lambda}$ such that
\begin{equation}
\label{eqn:weight:zero:pullback:Upsilon:r=10:delta=0}
{\rm wt}(J_{M}^{*}\Upsilon_{g})=(0,4(2^{g-1}+1)(2^{g}-1)),
\qquad
{\rm div}(J_{M}^{*}\Upsilon_{g})=a\,{\mathcal D}_{\Lambda}+{\mathcal E}_{\Lambda}.
\end{equation}
Set $\varphi_{\Lambda}:=\Psi_{\Lambda}^{(2^{g-1}+1)\ell}\otimes J_{M}^{*}\Upsilon_{g}^{\ell}/\Phi_{M}^{(2^{g-1}+1)(2^{g}-1)}$.
Comparing \eqref{eqn:weight:zero:Psi:r=10:delta=0}, \eqref{eqn:weight:zero:pullback:Upsilon:r=10:delta=0}
and ${\rm wt}(\Phi_{M})=(4\ell,4\ell)$, ${\rm div}(\Phi_{M})=\ell\,{\mathcal D}_{\Lambda}$,
we get
\begin{equation}
\label{eqn:weight:product:Upsilon}
{\rm wt}(\varphi_{\Lambda})=(0,0),
\qquad
{\rm div}(\varphi_{\Lambda})
=
\ell\,\{
a-2(2^{2(g-1)}-1)
\}
{\mathcal D}_{\Lambda}
+
\ell\,{\mathcal E}_{\Lambda}.
\end{equation}
By Proposition~\ref{prop:zero:Upsilon}, we can apply Lemma~\ref{lemma:estimate:alpha:U(2)} 
to a general curve $\gamma\colon\varDelta\to{\mathcal M}_{\Lambda}$ intersecting $\overline{\mathcal D}_{\Lambda}^{0}$ transversally. 
Since $a\geq2(2^{2(g-1)}-1)$ by Lemma~\ref{lemma:estimate:alpha:U(2)}, we get ${\rm div}(\varphi_{\Lambda})\geq0$.
By the Koecher principle, $\varphi_{\Lambda}$ is a non-zero constant. 
\end{pf}

In the rest of this section, we determine $\Phi_{M}$ for the remaining $M$, i.e., those $M$ with $(r,\delta)=(2,0)$.
Then, either $M\cong{\Bbb U}$ or ${\Bbb U}(2)$.

\subsection
{The structure of $\Phi_{\Bbb U}$}
\label{sect:9.3}
\par
In Section~\ref{sect:9.3}, we set 
$$
M:={\Bbb U},
\qquad
\Lambda=M^{\perp}:={\Bbb U}^{\oplus2}\oplus{\Bbb E}_{8}^{\oplus2}.
$$ 
Then $g=10$ and $J_{M}^{*}\chi_{10}$ vanishes identically on $\Omega_{\Lambda}^{0}$. 
\par
Let $E_{4}(\tau)=\theta_{{\Bbb E}_{8}^{+}}(\tau)=1+240q+\cdots$ be the Eisenstein series of weight $4$ (or equivalently
the theta series of ${\Bbb E}_{8}^{+}$) and set
\begin{equation}
\label{eqn:correcting:modular:form:(r,a)=(2,0)}
f_{\Lambda}(\tau):=E_{4}(\tau)/\eta(\tau)^{24}=q^{-1}+264+O(q).
\end{equation}
Then $f_{\Lambda}(\tau)$ is a modular form of weight  $-8$.
In Section~\ref{sect:9.3}, we prove the following:

\begin{theorem}
\label{thm:formula:Phi:U}
There exists a constant $C_{M,\ell}>0$ such that
$$
\Phi_{M}^{(2^{g-1}+1)(2^{g}-1)}
=
C_{M,\ell}\,\Psi_{\Lambda}(\cdot,2^{g-1}F_{\Lambda}+f_{\Lambda})^{\ell}\otimes J_{M}^{*}\Upsilon_{g}^{\ell}.
$$
In particular, there is a constant $C_{M}>0$ depending only on $M=\Bbb U$ such that 
$$
\tau_{M}^{-(2^{g}-1)(2^{g}+2)}
=
C_{M}\,
\left\|
\Psi_{\Lambda}(\cdot,2^{g-1}F_{\Lambda}+f_{\Lambda})
\right\|
\cdot 
J_{M}^{*}
\left\|
\Upsilon_{g}
\right\|.
$$
\end{theorem}

For the proof of Theorem~\ref{thm:formula:Phi:U}, we first prove the following:

\begin{lemma}
\label{lemma:non-vanishing:Upsilon:(r,a)=(2,0)}
$J_{M}^{*}\Upsilon_{10}$ is nowhere vanishing on $\Omega_{\Lambda}^{0}\setminus{\mathcal H}_{\Lambda}$.
\end{lemma}

\begin{pf}
Let $(X,\iota,\alpha)$ be an arbitrary marked $2$-elementary $K3$ surface of type $M={\Bbb U}$ with period in 
$\Omega_{\Lambda}\setminus{\mathcal H}_{\Lambda}$.
Let $C$ be the component of genus $10$ of $X^{\iota}$. Fix a symplectic basis of $H_{1}(C,{\bf Z})$, so that $\varOmega(C)\in{\frak S}_{10}$, 
where $\varOmega(C)$ is the period of $C$ with respect to the symplectic basis.
By Proposition~\ref{thetanull (r,a)=(2,0)}, there is a unique even pair $(a_{0},b_{0})$, $a_{0},b_{0}\in\{0,\frac{1}{2}\}^{10}$ such that
$\theta_{a_{0},b_{0}}(\varOmega(C))=0$ and $\theta_{a,b}(\varOmega(C))\not=0$ for all even pair $(a,b)$ with $(a,b)\not=(a_{0},b_{0})$.
Hence we get
$$
J_{M}^{*}\Upsilon_{10}(X,\iota)
=
\Upsilon_{10}(\varOmega(C))
=
\prod_{(a,b)\not=(a_{0},b_{0})}\theta_{a,b}(\varOmega(C))^{8}\not=0.
$$
This proves the lemma.
\end{pf}

By Lemma~\ref{lemma:non-vanishing:Upsilon:(r,a)=(2,0)}, there exist $\alpha,\beta\in{\bf Z}_{>0}$ such that 
\begin{equation}
\label{eqn:divisor:Upsilon:U:1}
{\rm div}(J_{M}^{*}\Upsilon_{10})
=
\alpha\,{\mathcal D}_{\Lambda}+\beta\,{\mathcal H}_{\Lambda}.
\end{equation}

To prove Theorem~\ref{thm:formula:Phi:U}, we must determine $\beta$. We use the following notation: Set
$$
L:={\Bbb U}\oplus{\Bbb U}\oplus{\Bbb E}_{8}\oplus{\Bbb E}_{8}\oplus{\Bbb A}_{1}=\Lambda\oplus{\Bbb A}_{1}.
$$
Let $d\in\Delta^{+}_{L}$ be a generator of ${\Bbb A}_{1}$. 
Then $\Lambda=L\cap d^{\perp}$, $L^{\perp}={\Bbb A}_{1}^{+}$ and $A_{L}=\{0,{\bf 1}_{L}\}$, where ${\bf 1}_{L}=[d/2]$.
As before, we make the identification
$$
H_{d}=\Omega_{L\cap d^{\perp}}=\Omega_{\Lambda}.
$$
\par
Let $[\eta]\in H_{d}^{0}$ be an arbitrary point. Let $U\cong\varDelta^{19}$ be a small neighborhood of $[\eta]$ in $\Omega_{L}$ 
equipped with a system of coordinates $(t,s_{1},\ldots,s_{18})$ such that
$U\cap{\mathcal D}_{L}=U\cap H_{d}=U\cap\Omega_{\Lambda}=\{(t,s)\in U;\,t=0\}\cong\varDelta^{18}$.
Since $d\in\Delta^{+}_{L}$ and hence $g(L)=g(\Lambda)=10$ by \cite[Lemma 11.5]{Yoshikawa13}, 
the Torelli map $J_{{\Bbb A}_{1}^{+}}$ is a holomorphic map from $U$ to ${\mathcal A}_{10}$ by Theorem~\ref{thm:functoriality:Torelli:map}.
Let $\varPi\colon{\frak S}_{10}\to{\mathcal A}_{10}$ be the projection. 
Since $U$ is contractible, $J_{{\Bbb A}_{1}^{+}}\colon U\to{\mathcal A}_{10}$ is liftable. 
Namely, there exists a holomorphic map $\widetilde{J}_{{\Bbb A}_{1}^{+}}\colon U\to{\frak S}_{10}$ such that 
$$
J_{{\Bbb A}_{1}^{+}}=\varPi\circ\widetilde{J}_{{\Bbb A}_{1}^{+}}.
$$
Since $\widetilde{J}_{{\Bbb A}_{1}^{+}}$ takes its values in ${\frak S}_{10}$, 
the value $\theta_{a,b}(\widetilde{J}_{{\Bbb A}_{1}^{+}}(t,s))$ makes sense for all $(s,t)\in U$ and even $(a,b)$, $a,b\in\{0,1/2\}^{10}$. 
Since $\widetilde{J}_{{\Bbb A}_{1}^{+}}^{*}\chi_{10}=\prod_{(a,b)\,{\rm even}}\widetilde{J}_{{\Bbb A}_{1}^{+}}^{*}\theta_{a,b}$ is nowhere vanishing 
on $U\setminus({\mathcal D}_{L}\cup{\mathcal H}_{L})$
and since there exists by Proposition~\ref{thetanull (r,a)=(2,0)} a unique even theta constant $\theta_{a_{0},b_{0}}$ vanishing identically on $H_{d}$,
we get the following:
\begin{itemize}
\item
$U\cap H_{d}$ is a component of $U\cap{\rm div}(\widetilde{J}_{{\Bbb A}_{1}^{+}}^{*}\theta_{a_{0},b_{0}})$;
\item
$U\cap H_{d}$ is not a component of $U\cap{\rm div}(\widetilde{J}_{{\Bbb A}_{1}^{+}}^{*}\theta_{a,b})$ for any even $(a,b)\not=(a_{0},b_{0})$.
\end{itemize}
Thus there exist $c\in{\bf Z}_{>0}$, $c_{a,b}(\lambda)\in{\bf Z}_{\geq0}$ such that for $(a,b)=(a_{0},b_{0})$
\begin{equation}
\label{eqn:divisor:theta:constant:1}
{\rm div}(\widetilde{J}_{{\Bbb A}_{1}^{+}}^{*}\theta_{a_{0},b_{0}}^{8})|_{U}
=
c\,H_{d}
+
\sum_{\lambda\in L^{\lor}/\pm1,\,\lambda^{2}=-9/2,\,|\langle\lambda,d\rangle|=1,\,[\lambda]={\bf 1}_{L}}
c_{a_{0},b_{0}}(\lambda)\,H_{\lambda}
\end{equation}
and such that for all $(a,b)\not=(a_{0},b_{0})$
\begin{equation}
\label{eqn:divisor:theta:constant:2}
{\rm div}(\widetilde{J}_{{\Bbb A}_{1}^{+}}^{*}\theta_{a,b}^{8})|_{U}
=
\sum_{\lambda\in L^{\lor}/\pm1,\,\lambda^{2}=-9/2,\,|\langle\lambda,d\rangle|=1,\,[\lambda]={\bf 1}_{L}}
c_{a,b}(\lambda)\,H_{\lambda}.
\end{equation}
Let $s_{d}\in O(L)$ be the reflection with respect to $d$. 
Since $J_{{\Bbb A}_{1}^{+}}\circ s_{d}=J_{{\Bbb A}_{1}^{+}}$, we have $s_{d}^{*}\circ(\widetilde{J}_{{\Bbb A}_{1}^{+}})^{*}=(\widetilde{J}_{{\Bbb A}_{1}^{+}})^{*}$,
which implies the following equality for every even pair $(a,b)$
\begin{equation}
\label{eqn:invariance:coeff:reflection}
c_{a,b}(s_{d}(\lambda))=c_{a,b}(\lambda).
\end{equation}

\begin{lemma}
\label{lemma:concentration:nonzero:c_(a,b)}
Let $\lambda\in L^{\lor}$ be such that $\lambda^{2}=-9/2$, $|\langle\lambda,d\rangle|=1$, $[\lambda]={\bf 1}_{L}$.
If $c_{a_{0},b_{0}}(\lambda)>0$, then $c_{a,b}(\lambda)=0$ for all $(a,b)\not=(a_{0},b_{0})$.
\end{lemma}

\begin{pf}
Assume $c_{a_{0},b_{0}}(\lambda)>0$ and $c_{a',b'}(\lambda)>0$ for some even $(a',b')\not=(a_{0},b_{0})$.
Then, for every $2$-elementary $K3$ surface $(X,\iota)$ of type ${\Bbb A}_{1}^{+}$ whose period lies in $U\cap H_{\lambda}\setminus H_{d}$, 
$X^{\iota}$ has two distinct effective even half canonical bundles corresponding to $(a_{0},b_{0})$ and $(a',b')$. 
This contradicts Remark~\ref{remark:uniqueness theta chara k=0}. 
\end{pf}

Recall that $i\colon\Omega_{\Lambda}\hookrightarrow\Omega_{L}$ is the inclusion induced by that of lattices $\Lambda=L\cap d^{\perp}\subset L$.
On $H_{d}^{0}\cap U$, set 
\begin{equation}
\label{eqn:functoriality:Torelli:map:2}
\widetilde{J}_{M}:=\widetilde{J}_{{\Bbb A}_{1}^{+}}|_{U\cap H_{d}^{0}}.
\end{equation}
By \eqref{eqn:functoriality:Torelli:map}, we have $J_{M}=\varPi\circ\widetilde{J}_{M}$.
\par
Since $U\cap H_{d}\subset H_{d}^{0}=\Omega_{\Lambda}^{0}$, we have the following equality of divisors on 
$U\cap\Omega_{\Lambda}^{0}$ by \eqref{eqn:divisor:Upsilon:U:1}
\begin{equation}
\label{eqn:divisor:Upsilon:U:2}
{\rm div}(\widetilde{J}_{M}^{*}\Upsilon_{10})|_{U\cap H_{d}}
=
\beta\,\sum_{\mu\in\Lambda/\pm1,\,\mu^{2}=-4}H_{\mu}.
\end{equation}

\begin{lemma}
\label{lemma:coeff:div:theta:const}
For every $\lambda\in L^{\lor}$ with 
$\lambda^{2}=-9/2$, $\langle\lambda,d\rangle=\pm1$, $\lambda\equiv{\bf 1}_{L}\mod L$,
the following equalities hold
$$
c_{a_{0},b_{0}}(\lambda)=0,
\qquad
\sum_{(a,b)\not=(a_{0},b_{0})}c_{a,b}(\lambda)=\beta/2.
$$
\end{lemma}

\begin{pf}
Since
$
\widetilde{J}_{{\Bbb A}_{1}^{+}}^{*}\left(\chi_{10}/\theta_{a_{0},b_{0}}\right)^{8}|_{U\cap H_{d}}
=
\widetilde{J}_{{\Bbb A}_{1}^{+}}^{*}\Upsilon_{10}|_{U\cap H_{d}}
$
by the definitions of $\Upsilon_{g}$ and $\theta_{a_{0},b_{0}}$, we get the equality of divisors on $U\cap\Omega_{\Lambda}^{0}$
\begin{equation}
\label{eqn:divisor:chi10/theta:1}
i^{*}{\rm div}(\widetilde{J}_{{\Bbb A}_{1}^{+}}^{*}\left(\chi_{10}/\theta_{a_{0},b_{0}}\right)^{8}|_{U})
=
i^{*}{\rm div}(\widetilde{J}_{{\Bbb A}_{1}^{+}}^{*}\Upsilon_{10}|_{U})
=
{\rm div}(\widetilde{J}_{M}^{*}\Upsilon_{10}),
\end{equation}
where the second equality follows from \eqref{eqn:functoriality:Torelli:map:2}.
On the other hand, we get by \eqref{eqn:divisor:theta:constant:2}
$$
{\rm div}(\widetilde{J}_{{\Bbb A}_{1}^{+}}^{*}\left(\chi_{10}/\theta_{a_{0},b_{0}}\right)^{8}|_{U})
=
\sum_{\lambda\in L^{\lor}/\pm1,\,\lambda^{2}=-9/2,\,|\langle\lambda,d\rangle|=1,\,[\lambda]={\bf 1}_{L}}
(\sum_{(a,b)\not=(a_{0},b_{0})}c_{a,b}(\lambda))\,H_{\lambda},
$$
which, together with \eqref{eqn:fiber:projection}, yields the following equality of divisors on $U$
\begin{equation}
\label{eqn:divisor:chi10/theta:2}
i^{*}{\rm div}\left(\widetilde{J}_{{\Bbb A}_{1}^{+}}^{*}\left(\frac{\chi_{10}}{\theta_{a_{0},b_{0}}}\right)^{8}|_{U}\right)
=
\sum_{\mu\in\Lambda/\pm1,\,\mu^{2}=-4}
\sum_{(a,b)\not=(a_{0},b_{0})}\{c_{a,b}(\mu+\frac{d}{2})+c_{a,b}(\mu-\frac{d}{2})\}\,H_{\mu}.
\end{equation}
Substituting \eqref{eqn:divisor:Upsilon:U:2}, \eqref{eqn:divisor:chi10/theta:2} into \eqref{eqn:divisor:chi10/theta:1}
and comparing the coefficients of $H_{\mu}$, we get
\begin{equation}
\label{eqn:sum:c_(a,b)}
\sum_{(a,b)\not=(a_{0},b_{0})}\{c_{a,b}(\mu+\frac{d}{2})+c_{a,b}(\mu-\frac{d}{2})\}=\beta.
\end{equation}
Since $c_{a,b}(\mu+\frac{d}{2})=c_{a,b}(\mu-\frac{d}{2})$ by \eqref{eqn:invariance:coeff:reflection}, 
we get $\sum_{(a,b)\not=(a_{0},b_{0})}c_{a,b}(\lambda)=\beta/2$ by \eqref{eqn:sum:c_(a,b)}.
If $c_{a_{0},b_{0}}(\lambda)>0$, then $\sum_{(a,b)\not=(a_{0},b_{0})}c_{a,b}(\lambda)=0$ by Lemma~\ref{lemma:concentration:nonzero:c_(a,b)}.
Since $\beta>0$, this contradicts the equality $\sum_{(a,b)\not=(a_{0},b_{0})}c_{a,b}(\lambda)=\beta/2$.
Thus $c_{a_{0},b_{0}}(\lambda)=0$.
\end{pf}

\begin{lemma}
\label{lemma:value:beta:U}
One has the equality $\beta=2^{5}$.
\end{lemma}

\begin{pf}
By \eqref{eqn:divisor:theta:constant:1}, \eqref{eqn:divisor:theta:constant:2} and Lemma~\ref{lemma:coeff:div:theta:const}, we get on $U$
$$
{\rm div}(\widetilde{J}_{{\Bbb A}_{1}^{+}}^{*}\theta_{a_{0},b_{0}}^{8})=c\,H_{d},
\qquad
{\rm div}(\widetilde{J}_{{\Bbb A}_{1}^{+}}^{*}(\chi_{10}/\theta_{a_{0},b_{0}})^{8})
=
(\beta/2)\cdot{\mathcal H}_{L},
$$
which yields the equality of divisors 
${\rm div}(J_{{\Bbb A}_{1}^{+}}^{*}\chi_{10}^{8})=c\,{\mathcal D}_{L}^{+}+(\beta/2)\cdot{\mathcal H}_{L}$
on $U$.
By Remark~\ref{remark:thetanull:(g,k)=(10,1)}, we get $c=c_{10}=2^{4}\cdot7$ and $\beta/2=b_{10}=2^{4}$.
Thus $\beta=2^{5}$.
\end{pf}

{\bf Proof of Theorem~\ref{thm:formula:Phi:U} }
By Theorem~\ref{thm:Borcherds:lift} (1), we have
\begin{equation}
\label{eqn:weight:Psi:II_{2,18}}
{\rm wt}(\Psi_{\Lambda}^{2^{4}})=-4(2^{5}+1)(2^{9}-2^{4}+1),
\qquad
{\rm div}(\Psi_{\Lambda}^{2^{4}})=2^{4}\,{\mathcal D}_{\Lambda}-{\mathcal H}_{\Lambda}.
\end{equation}
By \cite[Th.\,13.3]{Borcherds98}, we get
\begin{equation}
\label{eqn:weight:divisor:lift:E_4/eta^24}
{\rm wt}(\Psi_{\Lambda}(\cdot,f_{\Lambda}))=2^{2}(2^{5}+1),
\qquad
{\rm div}(\Psi_{\Lambda}(\cdot,f_{\Lambda}))={\mathcal D}_{\Lambda}.
\end{equation}
By \eqref{eqn:weight:Psi:II_{2,18}}, \eqref{eqn:weight:divisor:lift:E_4/eta^24}, we get
\begin{equation}
\label{eqn:weight:Psi:Lambda:U}
{\rm wt}(\Psi_{\Lambda}(\cdot,2^{g-1}F_{\Lambda}+f_{\Lambda}))
=
-2^{2}(2^{9}+1)(2^{10}-1),
\end{equation}
\begin{equation}
\label{eqn:divisor:Psi:Lambda:U}
{\rm div}(\Psi_{\Lambda}(\cdot,2^{g-1}F_{\Lambda}+f_{\Lambda}))
=
(2^{9}+1){\mathcal D}_{\Lambda}-2^{5}{\mathcal H}_{\Lambda}.
\end{equation}
Set
$\varphi_{\Lambda}:=\Psi_{\Lambda}(\cdot,2^{g-1}F_{\Lambda}+f_{\Lambda})^{\ell}\otimes\Upsilon_{g}^{\ell}/\Phi_{M}^{(2^{g-1}+1)(2^{g}-1)}$.
Since ${\rm wt}(\Phi_{M})=(-4\ell,4\ell)$ and ${\rm div}(\Phi_{M})=\ell\,{\mathcal D}_{\Lambda}$, we deduce from
\eqref{eqn:weight:Psi:Lambda:U}, \eqref{eqn:divisor:Psi:Lambda:U} that
${\rm wt}(\varphi_{\Lambda})=(0,0)$ 
and
$$
{\rm div}(\varphi_{\Lambda})
=
\ell\{\alpha-2(2^{18}-1)\}\,{\mathcal D}_{\Lambda}+(\beta-2^{5})\,{\mathcal H}_{\Lambda}
=
\ell\{\alpha-2(2^{18}-1)\}\,{\mathcal D}_{\Lambda},
$$
where we used Lemma~\ref{lemma:value:beta:U} to get the second equality.
By Lemma~\ref{lemma:non-vanishing:Upsilon:(r,a)=(2,0)} and Theorem~\ref{thm:structure:g<6:r<10} for $[M\perp d]$, $d\in\Delta_{\Lambda}$, 
Lemma~\ref{lemma:estimate:alpha:U(2)} applies to a general curve $\gamma\colon\varDelta\to{\mathcal M}_{\Lambda}$
intersecting $\overline{\mathcal D}_{\Lambda}^{0}$ transversally. 
Since $\alpha\geq 2(2^{18}-1)$ by Lemma~\ref{lemma:estimate:alpha:U(2)}, we get ${\rm div}(\varphi_{\Lambda})\geq0$. 
As before, this implies that $\varphi_{\Lambda}$ is a constant.
\qed

\subsection
{The structure of $\Phi_{{\Bbb U}(2)}$}
\label{sect:9.4}
\par
In Section~\ref{sect:9.4}, we set
$$
M:={\Bbb U}(2),
\qquad
\Lambda=M^{\perp}:={\Bbb U}(2)\oplus{\Bbb U}\oplus{\Bbb E}_{8}\oplus{\Bbb E}_{8}.
$$
Then $g=9$ and $J_{M}^{*}\chi_{9}$ vanishes identically.
Let $\{{\frak e},{\frak f}\}$ be a basis of $M={\Bbb U}(2)$ with ${\frak e}^{2}={\frak f}^{2}=0$ and $\langle{\frak e},{\frak f}\rangle=2$.
Hence $({\frak e}+{\frak f})/2\in A_{\Lambda}$ is the unique element with non-zero norm.
Let ${\bf e}_{00},{\bf e}_{01},{\bf e}_{10},{\bf e}_{11}$ be the standard basis of ${\bf C}[A_{\Lambda}]={\bf C}[A_{{\Bbb U}(2)}]$,
where ${\bf e}_{\alpha\beta}$ corresponds to $(\alpha\,{\frak e}+\beta\,{\frak f})/2\in A_{\Lambda}$. 
Applying the construction \cite[Proof of Lemma 11.1]{Borcherds00}, \cite[Th.\,6.2]{Scheithauer06}
to the modular form $\eta(\tau)^{-8}\eta(2\tau)^{-8}$, we define
\begin{equation}
\label{eqn:correcting:modular:form:(r,a)=(2,2)}
\begin{aligned}
f_{\Lambda}(\tau)
&:=
8
\sum_{\gamma\in A_{\Lambda}}
\{
\eta\left(\frac{\tau}{2}\right)^{-8}\eta(\tau)^{-8}
+
(-1)^{q_{\Lambda}(\gamma)}\eta\left(\frac{\tau+1}{2}\right)^{-8}\eta(\tau+1)^{-8}
\}
\,{\bf e}_{\gamma}
\\
&\quad
+\eta(\tau)^{-8}\eta(2\tau)^{-8}\,{\bf e}_{00}.
\end{aligned}
\end{equation}
Then $f_{\Lambda}(\tau)$ is an $O(A_{\Lambda})$-invariant modular form of weight $-8$ and of type $\rho_{\Lambda}$ with principal part 
\begin{equation}
\label{eqn:principal:part:f:(r,a)=(2,2)}
{\mathcal P}_{\leq0}[f_{\Lambda}]=(q^{-1}+136)\,{\bf e}_{00}+16q^{-1/2}\,{\bf e}_{11}.
\end{equation}
Since $O(q_{\Lambda})$ preserves $({\frak e}+{\frak f})/2$, the Heegner divisor of $\Omega_{\Lambda}$
$$
{\mathcal H}_{\Lambda}(-1,{\bf e}_{11})
:=
\sum_{\lambda\in\Lambda^{\lor}/\pm1;\,\lambda^{2}=-1,\,[\lambda]=({\frak e}+{\frak f})/2\mod\Lambda}H_{\lambda}
$$ 
is $O^{+}(\Lambda)$-invariant. 
By \cite[Th.\,13.3]{Borcherds98}, the Borcherds lift $\Psi_{\Lambda}(\cdot,f_{\Lambda})$ is an automorphic form on $\Omega_{\Lambda}$ 
for $O^{+}(\Lambda)$ such that
\begin{equation}
\label{eqn:zero:divisor:Borcherds:product:f}
{\rm wt}\,\Psi_{\Lambda}(\cdot,f_{\Lambda})
=
68=2^{6}+2^{2},
\qquad
{\rm div}\,\Psi_{\Lambda}(\cdot,f_{\Lambda})
=
{\mathcal D}_{\Lambda}+2^{4}\,{\mathcal H}_{\Lambda}(-1,{\bf e}_{11}).
\end{equation}

In this subsection, we prove the following:

\begin{theorem}
\label{thm:formula:Phi:U(2)}
There exists a constant $C_{M,\ell}>0$ such that
$$
\Phi_{M}^{(2^{g-1}+1)(2^{g}-1)}
=
C_{M,\ell}\,\Psi_{\Lambda}(\cdot,2^{g-1}F_{\Lambda}+f_{\Lambda})^{\ell}\otimes J_{M}^{*}\Upsilon_{g}^{\ell}.
$$
In particular, there is a constant $C_{M}>0$ depending only on $M={\Bbb U}(2)$ such that 
$$
\tau_{M}^{-(2^{g}+2)(2^{g}-1)}
=
C_{M}\,
\left\|
\Psi_{\Lambda}(\cdot,2^{g-1}F_{\Lambda}+f_{\Lambda})
\right\|
\cdot 
J_{M}^{*}
\left\|
\Upsilon_{g}
\right\|.
$$
\end{theorem}

Define the reduced divisor ${\mathcal H}_{1}$ on $\Omega_{\Lambda}$ as
$$
{\mathcal H}_{1}
:=
{\mathcal H}_{\Lambda}-{\mathcal H}_{\Lambda}(-1,{\bf e}_{11})
=
\sum_{\lambda\in\Lambda/\pm1,\,\lambda^{2}=-4,\,{\rm div}(\lambda)=1}H_{\lambda}.
$$
The divisor $\overline{\mathcal H}_{1}\subset{\mathcal M}_{\Lambda}$ in Section~\ref{ssec:(r,a)=(2,2)} is obtained as the quotient 
$\overline{\mathcal H}_{1}={\mathcal H}_{1}/O(\Lambda)$.

\begin{lemma}
\label{lemma:non-vanishing:Upsilon:(r,a)=(2,2)}
$J_{M}^{*}\Upsilon_{9}$ is nowhere vanishing on $\Omega_{\Lambda}^{0}\setminus{\mathcal H}_{1}$.
\end{lemma}

\begin{pf}
By using Proposition~\ref{thetanull (r,a)=(2,2)} instead of Proposition~\ref{thetanull (r,a)=(2,0)},
the proof is parallel to that of Lemma~\ref{lemma:non-vanishing:Upsilon:(r,a)=(2,0)}.
\end{pf}

By Lemma~\ref{lemma:non-vanishing:Upsilon:(r,a)=(2,2)}, there exist $\alpha,\beta\in{\bf Z}_{>0}$ such that
\begin{equation}
\label{eqn:divisor:Upsilon:U(2)}
{\rm div}(J_{M}^{*}\Upsilon_{9})=\alpha\,{\mathcal D}_{\Lambda}+\beta\,{\mathcal H}_{1}.
\end{equation}

\begin{lemma}
\label{lemma:estimate:beta:U(2)}
One has the inequality $\beta\geq2^{4}$.
\end{lemma}

\begin{pf}
Let $\lambda\in\Lambda$ be an arbitrary vector such that $\lambda^{2}=-4$ and ${\rm div}(\lambda)=1$. We set
$$
H_{\lambda}^{0}
:=
H_{\lambda}
\setminus
(
{\mathcal D}_{\Lambda}\cup\bigcup_{\lambda'\in\Lambda,\,(\lambda')^{2}=-4,\,{\rm div}(\lambda')=1}H_{\lambda'}
).
$$
Then $H_{\lambda}^{0}$ is a non-empty Zariski open subset of $H_{\lambda}$.
Let $[\eta]\in H_{\lambda}^{0}$ be an arbitrary point. Let $U\cong\varDelta^{18}$ be a small neighborhood of $[\eta]$ in $\Omega_{\Lambda}$ 
such that $U\cap({\mathcal H}_{1}\cup{\mathcal D}_{\Lambda})=U\cap H_{\lambda}^{0}\cong\varDelta^{17}$.
Since $U$ is small enough, 
there is a marked family of $2$-elementary $K3$ surfaces $(p\colon({\mathcal X},\iota)\to U,\alpha)$ of type $M={\Bbb U}(2)$,
whose period map is the inclusion $U\hookrightarrow\Omega_{\Lambda}$. Set ${\mathcal C}:={\mathcal X}^{\iota}$.
Then $p\colon{\mathcal C}\to U$ is a family of smooth curves of genus $9$. 
Set $C_{t}:=p^{-1}(t)\cap{\mathcal C}$ for $t\in U$.
The period map $J_{M}|_{U}$ is a holomorphic map from $U$ to ${\mathcal A}_{9}$
such that $J_{M}(t)=\varOmega(C_{t})$.
Since $U$ is contractible, the local system $R^{1}(p|_{\mathcal C})_{*}{\bf Z}$ is trivial and admits a symplectic basis.
Hence $J_{M}\colon U\to{\mathcal A}_{9}$ lifts to a holomorphic map
$\widetilde{J}_{M}\colon U\to{\frak S}_{9}$ such that $J_{M}=\varPi\circ\widetilde{J}_{M}$,
where $\varPi\colon{\frak S}_{9}\to{\mathcal A}_{9}$ is the projection. 
\par
Since $\widetilde{J}_{M}$ takes its values in ${\frak S}_{9}$, 
the value of the theta constant  $\theta_{a,b}(\widetilde{J}_{M}(t))$ makes sense
for all $t\in U$ and for every even pair $(a,b)$, $a,b\in\{0,1/2\}^{9}$. 
Moreover, since the family $p\colon{\mathcal C}\to U$ admits a level $4l$-structure for any $l\in{\bf Z}_{>0}$,
the square root $\widetilde{J}_{M}^{*}\sqrt{\theta_{a,b}}$ is a well-defined holomorphic section of a holomorphic line bundle on $U$ 
for every even pair $(a,b)$ by \cite[Th.\,1]{Tsuyumine91}. Since any holomorphic line bundle on $U$ is trivial, 
we may regard $\widetilde{J}_{M}^{*}\sqrt{\theta_{a,b}}\in{\mathcal O}(U)$.
\par
We deduce from Proposition~\ref{thetanull (r,a)=(2,2)} the existence of a unique even pair $(a_{0},b_{0})$ and at least one even pair $(a_{1},b_{1})$
with the following properties:
\begin{itemize}
\item[(1)]
$\widetilde{J}_{M}^{*}\theta_{a_{0},b_{0}}$ vanishes identically on $U$;
\item[(2)]
Set-theoretically, $U\cap{\rm div}(\widetilde{J}_{M}^{*}\theta_{a_{1},b_{1}})=U\cap H_{\lambda}$.
\end{itemize}
By (2) and the fact $\widetilde{J}_{M}^{*}\sqrt{\theta_{a_{1},b_{1}}}\in{\mathcal O}(U)$,
there exists $c\in{\bf Z}_{>0}$ such that on $U$
$$
{\rm div}(\widetilde{J}_{M}^{*}\theta_{a_{1},b_{1}})=2c\,H_{\lambda}.
$$
By (1) and the definition of $\Upsilon_{9}$, we have
$$
\widetilde{J}_{M}^{*}\Upsilon_{9}|_{U}
=
\prod_{(a,b)\not=(a_{0},b_{0})}\widetilde{J}_{M}^{*}\theta_{a,b}^{8}
=
\widetilde{J}_{M}^{*}\theta_{a_{1},b_{1}}^{8}\cdot\prod_{(a,b)\not=(a_{0},b_{0}),(a_{1},b_{1})}\widetilde{J}_{M}^{*}\theta_{a,b}^{8}.
$$
Setting $E:={\rm div}(\prod_{(a,b)\not=(a_{0},b_{0}),(a_{1},b_{1})}\widetilde{J}_{M}^{*}\theta_{a,b}^{8})$,
we get the equality of divisors on $U$
$$
{\rm div}(\widetilde{J}_{M}^{*}\Upsilon_{9}|_{U})=16c\,H_{\lambda}+E.
$$
Thus we get the desired inequality $\beta\geq16c\geq2^{4}$.
\end{pf}

{\bf Proof of Theorem~\ref{thm:formula:Phi:U(2)} }
By Theorem~\ref{thm:Borcherds:lift} (2), we get
\begin{equation}
\label{eqn:weight:Psi:Lambda:U(2)}
{\rm wt}(\Psi_{\Lambda}^{\ell})=-2^{-2}\{2^{4}(2^{9}+1)+1\}\ell,
\qquad
{\rm div}(\Psi_{\Lambda}^{\ell})=\ell\left({\mathcal D}_{\Lambda}-2^{-4}{\mathcal H}_{\Lambda}\right).
\end{equation}
By \eqref{eqn:zero:divisor:Borcherds:product:f} and \eqref{eqn:weight:Psi:Lambda:U(2)}, we get
\begin{equation}
\label{eqn:wt:div:(r,a)=(2,2)}
{\rm wt}(\Psi_{\Lambda}^{2^{8}}\cdot\Psi_{\Lambda}(\cdot,f_{\Lambda}))=-2^{2}(2^{8}+1)(2^{9}-1),
\qquad
{\rm div}(\Psi_{\Lambda}^{2^{8}}\cdot\Psi_{\Lambda}(\cdot,f_{\Lambda}))=(2^{8}+1){\mathcal D}_{\Lambda}-2^{4}{\mathcal H}_{1}.
\end{equation}
Since ${\rm wt}(J_{M}^{*}\Upsilon_{9})=(0,2^{2}(2^{8}+1)(2^{9}-1))$ and since
${\rm div}(J_{M}^{*}\Upsilon_{9})=\alpha\,{\mathcal D}_{\Lambda}+\beta\,{\mathcal H}_{1}$ by \eqref{eqn:divisor:Upsilon:U(2)}, we get
$$
{\rm wt}\left(
\Psi_{\Lambda}(\cdot,2^{8}F_{\Lambda}+f_{\Lambda})\otimes J_{M}^{*}\Upsilon_{9}
\right)
=
\left(
-4(2^{8}+1)(2^{9}-1),4(2^{8}+1)(2^{9}-1)
\right),
$$
$$
{\rm div}\left(
\Psi_{\Lambda}(\cdot,2^{8}F_{\Lambda}+f_{\Lambda})\otimes J_{M}^{*}\Upsilon_{9}
\right)
=
(\alpha+2^{8}+1)\,{\mathcal D}_{\Lambda}+(\beta-2^{4})\,{\mathcal H}_{1}.
$$
Since ${\rm wt}(\Phi_{M})=(-4\ell,4\ell)$ and ${\rm div}(\Phi_{M})=\ell\,{\mathcal D}_{\Lambda}$,
$$
\varphi_{\Lambda}
:=
\Psi_{\Lambda}(\cdot,2^{8}F_{\Lambda}+f_{\Lambda})^{\ell}\otimes J_{M}^{*}\Upsilon_{9}^{\ell}
/
\Phi_{M}^{(2^{8}+1)(2^{9}-1)}
$$
is a meromorphic function on ${\mathcal M}_{\Lambda}$ with divisor
$$
{\rm div}(\varphi_{\Lambda})
=
\ell\{\alpha-2(2^{16}-1)\}\,\overline{\mathcal D}_{\Lambda}+\ell(\beta-2^{4})\,\overline{\mathcal H}_{1}.
$$
By Lemma~\ref{lemma:non-vanishing:Upsilon:(r,a)=(2,2)} and Theorem~\ref{thm:structure:g<6:r<10} for $[M\perp d]$, $d\in\Delta_{\Lambda}$, 
Lemma~\ref{lemma:estimate:alpha:U(2)} applies to a general curve $\gamma\colon\varDelta\to{\mathcal M}_{\Lambda}$
intersecting $\overline{\mathcal D}_{\Lambda}^{0}$ transversally.
Since $\alpha\geq2(2^{16}-1)$ and $\beta\geq2^{4}$ by Lemmas~\ref{lemma:estimate:alpha:U(2)} and \ref{lemma:estimate:beta:U(2)},
we get ${\rm div}(\varphi_{\Lambda})\geq0$. Thus $\varphi_{\Lambda}$ is a constant.
\qed

\subsection
{The divisors of $J_{M}^{*}\chi_{g}^{8}$ and $J_{M}^{*}\Upsilon_{g}$}
\label{sect:9.5}
\par

\begin{theorem}
\label{thm:divisor:pullback:chi:upsilon}
Let $M$ be a {\em non-exceptional} primitive $2$-elementary Lorentzian sublattice of ${\Bbb L}_{K3}$ with $\Lambda=M^{\perp}$
and invariants $(r,l,\delta)$.
Then the following holds:
\newline{$(1)$}
If $g=0$, i.e., $r+l=22$, then ${\rm div}(J_{M}^{*}\chi_{g}^{8})=0$.
\newline{$(2)$}
If $r\geq2$, $\delta=1$ and $1\leq g\leq9$, then
${\rm div}(J_{M}^{*}\chi_{g}^{8})=2^{2g-1}{\mathcal D}_{\Lambda}^{-}+2^{k+4}{\mathcal H}_{\Lambda}$.
\newline{$(3)$}
If $(r,\delta)=(1,1)$, then 
${\rm div}(J_{M}^{*}\chi_{g}^{8})=2^{19}{\mathcal D}_{\Lambda}^{-}+2^{4}\cdot7\,{\mathcal D}_{\Lambda}^{+}+2^{4}{\mathcal H}_{\Lambda}$.
\newline{$(4)$}
If $\delta=0$ and $r\not=2,10$, then
${\rm div}(J_{M}^{*}\chi_{g}^{8})=2^{g-1}(2^{g}+1)\,{\mathcal D}_{\Lambda}$.
\newline{$(5)$}
If $(r,\delta)=(2,0)$ or $(10,0)$, then $J_{M}^{*}\chi_{g}^{8}$ vanishes identically on $\Omega_{\Lambda}$.
\newline{$(6)$}
If $(r,\delta)=(10,0)$, then
${\rm div}(J_{M}^{*}\Upsilon_{g})=2(2^{2(g-1)}-1)\,{\mathcal D}_{\Lambda}$.
\newline{$(7)$}
If $(r,l,\delta)=(2,0,0)$, then 
${\rm div}(J_{M}^{*}\Upsilon_{10})=2(2^{18}-1)\,{\mathcal D}_{\Lambda}+2^{5}\,{\mathcal H}_{\Lambda}$.
\newline{$(8)$}
If $(r,l,\delta)=(2,2,0)$, then 
${\rm div}(J_{M}^{*}\Upsilon_{9})=2(2^{16}-1)\,{\mathcal D}_{\Lambda}+2^{4}\,\{{\mathcal H}_{\Lambda}-{\mathcal H}_{\Lambda}(-1,{\bf e}_{11})\}$.
\end{theorem}

\begin{pf}
The assertion (1) is obvious since $\chi_{g}=1$ for $g=0$. 
For $g=1,2$ (resp. $3\leq g\leq9$), we get (2) by \cite[Prop.\,4.2 (2), (3)]{Yoshikawa13}
(resp. Lemma~\ref{eqn:c:(g,0):g<10}, Proposition~\ref{prop:divisor:pull:back:Igusa}, \eqref{eqn:value:b_g}).
We get (3) by Proposition~\ref{prop:divisor:pull:back:Igusa}, \eqref{eqn:value:b_g}, \eqref{eqn:c:(10,0)}.
When $r>10$ (resp. $2<r<10$), we get (4) by \cite[Eq.\,(9.3)]{Yoshikawa13} and the equality $a=E=0$ in \cite[Proof of Th.\,9.1]{Yoshikawa13}
(resp. Lemma~\ref{lemma:div:Igusa:(gk,delta)=(6,1,0)}).
We get (5) by \cite[Prop.\,9.3]{Yoshikawa13} when $r=10$.
When $r=2$, $\delta=0$, there are two possible cases: $g=10$ and $g=9$. 
In case $g=10$, (5) was proved in Remark~\ref{remark:thetanull:(g,k)=(10,1)}.
In case $g=9$, (5) was proved in Section~\ref{ssec:(r,a)=(2,2)}. This proves (5).
We get (6) by \eqref{eqn:weight:zero:pullback:Upsilon:r=10:delta=0} and the equality
$a=2(2^{2(g-1)}-1)$ in the proof of Theorem~\ref{thm:structure:r=10:delta=0}.
We get (7) by \eqref{eqn:divisor:Upsilon:U:1},
since $\beta=2^{5}$ and $\alpha=2(2^{18}-1)$ in the proof of Theorem~\ref{thm:formula:Phi:U}.
We get (8) by \eqref{eqn:divisor:Upsilon:U(2)},
since the equalities $\alpha=2(2^{16}-1)$ and $\beta=2^{4}$ follow from the proof of Theorem~\ref{thm:formula:Phi:U(2)}.
This completes the proof.
\end{pf}

\begin{corollary}
\label{cor:completion:Prop.7.3:7.4}
For $(r,l,\delta)=(10,2,0),(10,4,0)$, $J_{M}(\Omega_{\Lambda}^{0})$ is disjoint from the hyperelliptic locus.  
For $(r,l,\delta)=(10,0,0)$, every member of $J_{M}(\Omega_{\Lambda}^{0})$ has exactly one effective even theta characteristic.
\end{corollary}

\begin{pf}
Set-theoretically, ${\frak H}_{{\rm hyp},4}$ (resp. ${\frak H}_{{\rm hyp},5}$) is given by ${\rm div}(\chi_{4})\cap{\rm div}(\Upsilon_{4})$ 
(resp. ${\rm div}(\chi_{5})\cap{\rm div}(\Upsilon_{5})\cap{\rm div}(F_{5})$) on ${\frak M}_{4}$ (resp. ${\frak M}_{5}$).
Since $J_{M}^{*}\Upsilon_{4}$ (resp. $J_{M}^{*}\Upsilon_{5}$) is nowhere vanishing on $\Omega_{\Lambda}^{0}$ by 
Theorem~\ref{thm:divisor:pullback:chi:upsilon} (6) for $(r,l,\delta)=(10,4,0)$ (resp. $(10,2,0)$), we get the first assertion.
When $(r,l,\delta)=(10,0,0)$, 
$J_{M}^{*}\chi_{6}$ vanishes identically on $\Omega_{\Lambda}^{0}$ and $J_{M}^{*}\Upsilon_{6}$ is nowhere vanishing on $\Omega_{\Lambda}^{0}$
by Theorem~\ref{thm:divisor:pullback:chi:upsilon} (5), (6). This, together with Lemma~\ref{lemma:locus:vanishing:two:theta:char}, implies the second assertion.
\end{pf}

Let us give a geometric interpretation of Theorem~\ref{thm:divisor:pullback:chi:upsilon}, 
where we use the notion of log Del Pezzo surfaces of index $\leq2$. We refer to \cite{AlexeevNikulin06} for this notion. 
Let $S$ be a log Del Pezzo surface of index $\leq2$. By \cite[Th.\,1.5]{AlexeevNikulin06}, the bi-anticanonical system of $S$ contains a smooth member.
For any smooth member $C\in|-2K_{S}|$, one can canonically associate a $2$-elementary $K3$ surface $(X_{(S,C)},\iota_{(S,C)})$
whose quotient $X_{(S,C)}/\iota_{(S,C)}$ is the right resolution of $S$. (See \cite[Sect.\,2.1]{AlexeevNikulin06}.)
We define the invariant $\delta(S)\in\{0,1\}$ as that of the $2$-elementary lattice $H^{2}(X_{(S,C)},{\bf Z})_{+}$.
Then $\delta(S)$ is independent of the choice of a smooth member $C\in|-2K_{S}|$.

\begin{corollary}
\label{cor:theta:char:-2K:curve}
Let $S$ be a log Del Pezzo surface of index $\leq2$ and let $C\in|-2K_{S}|$ be a smooth member.
If $S\not\cong{\bf F}_{0},{\bf P}(1,1,2)$, then the following hold:
\newline{$(1)$}
When $(\rho(S),\delta(S))\not=(2,0),(10,0)$,
$C$ has an effective even theta characteristic if and only if the period of $(X_{(S,C)},\iota_{(S,C)})$ lies in the characteristic Heegner divisor.
\newline{$(2)$}
When $(\rho(S),\delta(S))=(2,0)$ or $(10,0)$, $C$ always has an effective even theta characteristic. 
Moreover, $C$ has at least two effective even theta characteristics if and only if
 the period of $(X_{(S,C)},\iota_{(S,C)})$ lies in the characteristic Heegner divisor.
\end{corollary}

\begin{pf}
Since the period of $C$ is exactly the image of $(X_{(S,C)},\iota_{(S,C)})$ by the Torelli map, the result follows from
Theorem~\ref{thm:divisor:pullback:chi:upsilon}.
\end{pf}

\subsection
{The quasi-affinity of ${\mathcal M}_{\Lambda}^{0}$}
\label{sect:9.6}
\par
As an application of the results in Sections~\ref{sect:8} and \ref{sect:9}, we obtain the quasi-affinity of ${\mathcal M}_{\Lambda}^{0}$
for a wide range of $\Lambda$ as follows.

\begin{theorem}
\label{thm:quasiaffinity}
If $\Lambda$ is a primitive $2$-elementary sublattice of ${\Bbb L}_{K3}$ with $r(\Lambda)<16$ and signature $(2,r(\Lambda)-2)$, 
then ${\mathcal M}_{\Lambda}^{0}$ is quasi-affine.
\end{theorem}

\begin{pf}
By \cite[Prop.\,2.2]{Yoshikawa13}, $\overline{J}_{M}^{0}$ extends to a meromorphic map from ${\mathcal M}_{\Lambda}^{*}$ to ${\mathcal A}_{g}^{*}$. 
We regard ${\mathcal M}_{\Lambda}^{0}$ as a Zariski open subset of a subvariety of ${\mathcal M}_{\Lambda}^{*}\times{\mathcal A}_{g}^{*}$ via 
the embedding ${\rm id}_{{\mathcal M}_{\Lambda}^{0}}\times \overline{J}_{M}^{0}$. It suffices to prove the existence of a meromorphic section of an ample line bundle
on ${\mathcal M}_{\Lambda}^{*}\times{\mathcal A}_{g}^{*}$, which is nowhere vanishing on ${\mathcal M}_{\Lambda}^{0}$
(\cite[Prop.\,5.1.2]{Grothendieck61}). 
Let $\lambda_{\Lambda}$ be the Hodge bundle on ${\mathcal M}_{\Lambda}^{*}$. 
By Baily-Borel, the line bundle $\lambda_{\Lambda}^{\otimes a}\boxtimes{\mathcal F}_{g}^{\otimes b}$ on 
${\mathcal M}_{\Lambda}^{*}\times{\mathcal A}_{g}^{*}$ is ample if $a>0$ and $b>0$.
Under the assumption $r>6$, it follows from Theorems~\ref{thm:structure:g<6:r<10}, \ref{thm:structure:r=6:delta=0}, \ref{thm:structure:r=10:delta=0}
that $\Phi_{M}^{\nu}$ is a meromorphic section of $\lambda_{\Lambda}^{\otimes a}\boxtimes{\mathcal F}_{g}^{\otimes b}$
for some $a,b,\nu\in{\mathbf Z}_{>0}$. 
Since $\Phi_{M}^{\nu}$ is nowhere vanishing on ${\mathcal M}_{\Lambda}^{0}$ by Theorem~\ref{thm:automorphic:property:Phi:M}, 
$\Phi_{M}^{\nu}$ is a desired section.
\end{pf}

By \cite{BKPSB98}, it is known that ${\mathcal M}_{\Lambda}^{0}$ is quasi-affine when 
$\Lambda={\Bbb U}^{\oplus2}\oplus{\Bbb E}_{8}^{\oplus2}\oplus{\Bbb A}_{1}$, ${\Bbb U}^{\oplus2}\oplus{\Bbb E}_{8}^{\oplus2}$.
This, together with Theorem~\ref{thm:quasiaffinity}, implies that ${\mathcal M}_{\Lambda}^{0}$ is quasi-affine possibly except for $12$ isometry classes 
of primitive $2$-elementary sublattices of ${\Bbb L}_{K3}$.

\section{Spin-$1/2$ bosonization formula and a factorization of $\tau_{M}$}
\label{sect:10}
\par
In Section~\ref{sect:10}, we introduce a twisted version $\tau_{M}^{\rm spin}$ of $\tau_{M}$ and give its explicit formula
purely in terms of Borcherds products. The relations \eqref{eqn:relation:tau:spin:non-spin:1}, \eqref{eqn:relation:tau:spin:non-spin:2}
below provide a factorization of $\tau_{M}$ at the level of holomorphic torsion invariants.

\subsection
{Spin-$1/2$ bosonization formula}
\label{sect:10.1}
\par
Let $C$ be a smooth projective curve of genus $g$ and let $\Sigma$ be its theta characteristic.
The pair $(C,\Sigma)$ is called a spin curve. A theta characteristic $\Sigma$ is ineffective if $h^{0}(\Sigma)=0$. 
Let $\omega$ be a K\"ahler form on $C$. Then $\Sigma$ is equipped with the Hermitian metric induced by $\omega$.
Let $\tau(C,\Sigma;\omega)$ be the analytic torsion of $\Sigma$ with respect to $\omega$. 
Recall that ${\rm vol}(C,\omega)=\int_{C}\omega/2\pi$. We set
$$
\widetilde{\tau}_{g}(C,\Sigma)
:=
{\rm Vol}(C,\omega)\tau(C,\omega)\tau(C,\Sigma;\omega)^{2}.
$$
By the anomaly formula for Quillen metrics \cite{BGS88}, if $\Sigma$ is ineffective, 
$\widetilde{\tau}_{g}(C,\Sigma)$ is independent of the choice of a K\"ahler form $\omega$ on $C$.
Thus we get an invariant $\widetilde{\tau}_{g}$ of ineffective spin curves of genus $g$.
In Section~\ref{sect:10.1}, we recall the spin-$1/2$ bosonization formula \cite{AGMV86}, \cite{BostNelson86}, \cite{Fay92},
which gives an explicit formula for $\widetilde{\tau}_{g}$ viewed as a function on the moduli space of ineffective spin curves of genus $g$ 
with level $2$-structure.
\par
Let $V$ be a fixed symplectic vector space of rank $2g$ over ${\bf F}_{2}$ equipped with
a fixed symplectic basis $\{{\frak e}_{1},\ldots,{\frak e}_{g},{\frak f}_{1},\ldots,{\frak f}_{g}\}$.
A level $2$-structure on $C\in{\frak M}_{g}$ is defined as an isomorphism of symplectic vector spaces
$\alpha\colon V\cong H_{1}(C,{\bf F}_{2})$, where $H_{1}(C,{\bf F}_{2})$ is equipped with the intersection pairing.
Let ${\rm Alb}(C)[2]$ (resp. ${\rm Pic}^{0}(C)[2]$) be the $2$-division points of the Albanese variety ${\rm Alb}(C)$ 
(resp. Picard variety ${\rm Pic}^{0}(C)$). By the canonical isomorphism 
$H_{1}(C,{\bf F}_{2})\cong\frac{1}{2}H_{1}(C,{\bf Z})/H_{1}(C,{\bf Z})\cong{\rm Alb}(C)[2]$ and the Abel-Jacobi isomorphism
${\rm Pic}^{0}(C)\cong{\rm Alb}(C)$, a level $2$-structure on $C$ is identified with a symplectic basis of 
${\rm Pic}^{0}(C)[2]$ with respect to the Weil pairing.
\par
Let ${\frak M}_{g}(2)$ be the moduli space of projective curves of genus $g$ with level $2$-structure
and let $p\colon{\frak M}_{g}(2)\to{\frak M}_{g}$ be the natural projection. 
Let ${\mathcal S}_{g}^{+}$ be the moduli space of even spin curves of genus $g$ and 
let $\pi\colon{\mathcal S}_{g}^{+}\to{\frak M}_{g}$ be the natural projection.
We define ${\mathcal S}_{g}^{+}(2)$ as the fiber product ${\mathcal S}_{g}^{+}\times_{{\frak M}_{g}}{\frak M}_{g}(2)$. 
The projection from ${\mathcal S}_{g}^{+}(2)$ to ${\frak M}_{g}(2)$ (resp. ${\mathcal S}_{g}^{+}$)
is denoted again by $p$ (resp. $\pi$).
The covering $p\colon{\mathcal S}_{g}^{+}(2)\to{\frak M}_{g}(2)$ of degree $2^{g-1}(2^{g}+1)$ is trivial as follows.
On $(C,\alpha)\in{\frak M}_{g}(2)$, there is a distinguished even theta characteristic $\kappa\in{\rm Pic}^{g-1}(C)$ 
called Riemann's constant (e.g. \cite[p.6 and Lemma 1.5]{Fay92}). 
For every even pair $(a,b)$, $a,b\in\{0,1/2\}^{g}$, we define a section $\sigma_{a,b}\colon{\frak M}_{g}(2)\to{\mathcal S}_{g}^{+}(2)$ by
$\sigma_{a,b}(C,\alpha):=(C,\kappa\otimes\chi_{a,b},\alpha)$, where $\chi_{a,b}\in{\rm Pic}^{0}(C)[2]$ is the point 
corresponding to $\sum_{i}2a_{i}{\frak e}_{i}+\sum_{j}2b_{j}{\frak f}_{j}\in{\bf F}_{2}^{2g}$ via the isomorphism
${\bf F}_{2}^{2g}\cong{\rm Pic}^{0}(C)[2]$ induced by $\alpha$.
In this way, we get a decomposition ${\mathcal S}_{g}^{+}(2)=\amalg_{(a,b)\,{\rm even}}\sigma_{a,b}({\frak M}_{g}(2))$.
We set
${\mathcal S}_{g}^{+,0}(2):=\amalg_{(a,b)\,{\rm even}}\sigma_{a,b}({\frak M}_{g}(2)\setminus{\rm div}(\theta_{a,b}))$
and ${\mathcal S}_{g}^{+,0}:=\pi({\mathcal S}_{g}^{+,0}(2))=\bigcup_{(a,b)\,{\rm even}}
\pi(\sigma_{a,b}({\frak M}_{g}(2)\setminus{\rm div}(\theta_{a,b})))$. 
Since $h^{0}(\kappa\otimes\chi_{a,b})=0$ if and only if $\theta_{a,b}(\varOmega(C))\not=0$ for $(C,\alpha)\in{\frak M}_{g}(2)$, 
$\widetilde{\tau}_{g}$ is a function on ${\mathcal S}_{g}^{+,0}$.
\par
On the other hand, for every even pair $(a,b)$, the theta constant $\theta_{a,b}$ is a section of a certain line bundle 
on ${\frak M}_{g}(2)$ and its Petersson norm $\|\theta_{a,b}\|$ is a function on ${\frak M}_{g}(2)$. 
For $g=0$, we define $\|\theta_{a,b}\|:=1$. Let $\zeta_{\bf Q}(s)$ be the Riemann zeta function.
By the spin-$1/2$ bosonization formula \cite{AGMV86}, \cite{BostNelson86}, \cite[Th.\,4.9 (i), p.94 Eq.\,(4.58), p.97 Eq.\,(5.7)]{Fay92},
the following equality of functions on ${\frak M}_{g}(2)\setminus{\rm div}(\theta_{a,b})$ holds:
\begin{equation}
\label{eqn:spin:1/2:bosonization:0}
\sigma_{a,b}^{*}\pi^{*}\widetilde{\tau}_{g}
=
{\frak c}_{g}\,\|\theta_{a,b}\|^{-4},
\qquad
{\frak c}_{g}:=(4\pi)^{-g}e^{6(1-g)(2\zeta_{\bf Q}'(-1)+\zeta_{\bf Q}(-1))},
\end{equation}
where ${\frak c}_{g}$ is evaluated by the arithmetic Riemann-Roch theorem \cite{GilletSoule92} for $g=0$,
by Kronecker's limit formula and Ray-Singer's formula \cite[Th.\,4.1]{RaySinger73} for $g=1$,
and by Wentworth's formula \cite[Eq.\,(1.1) and Cor.\,1.1]{Wentworth12} for $g\geq2$.
In other words,
\begin{equation}
\label{eqn:spin:1/2:bosonization:1}
{\rm Vol}(C,\omega)\tau(C,\omega)\tau(C,\kappa\otimes\chi_{a,b};\omega)^{2}
=
{\frak c}_{g}\,\|\theta_{a,b}(\varOmega(C))\|^{-4}
\end{equation}
for all $(C,\alpha)\in{\frak M}_{g}(2)\setminus{\rm div}(\theta_{a,b})$ and even pairs $(a,b)$. 
Notice that the Laplacians (resp. volume) in \cite{Wentworth12} differ from ours by the scaling factor $2$ (resp. $2\pi$).
Hence Wentworth's formula \cite[Eq.\,(1.1) and Cor.\,1.1]{Wentworth12} reads
\begin{equation}
\label{eqn:spin:1/2:bosonization:2}
\begin{aligned}
\,&
{\rm Area}(C,\omega)2^{-\zeta_{{\mathcal O}_{C}}(0)}\tau(C,\omega)2^{-\zeta_{\Sigma}(0)}\tau(C,\kappa\otimes\chi_{a,b};\omega)^{2}
\\
&=
(4\pi e^{6(2\zeta_{\bf Q}'(-1)+\zeta_{\bf Q}(-1))})^{1-g}
\|\theta_{a,b}(\varOmega(C))\|^{-4},
\end{aligned}
\end{equation}
where $\zeta_{{\mathcal O}_{C}}(s)$ (resp. $\zeta_{\Sigma}(s)$) is the spectral zeta function of the Laplacian
$(\bar{\partial}+\bar{\partial}^{*})^{2}$ acting on the smooth sections of ${\mathcal O}_{C}$ (resp. $\Sigma$)
and ${\rm Area}(C,\omega):=\int_{C}\omega$.
Since
$$
\zeta_{{\mathcal O}_{C}}(0)=\frac{\chi(C)}{6}-h^{0}({\mathcal O}_{C})=-\frac{g+2}{3},
\qquad
\zeta_{\Sigma}(0)=\frac{\chi(C)}{6}+\frac{1}{2}\deg\Sigma-h^{0}(\Sigma)=\frac{g-1}{6}
$$
by \cite[p.37 l.16]{Fay92} and since ${\rm Vol}(C,\omega)={\rm Area}(C,\omega)/2\pi$, 
we get the value ${\frak c}_{g}$ in \eqref{eqn:spin:1/2:bosonization:0}.
\par
Let us extend the definition of $\widetilde{\tau}_{g}$ to disconnected curves as follows.
A line bundle on a disconnected curve is a theta characteristic if it is a componentwise theta characteristic. 
Similarly, a theta characteristic on a disconnected curve is ineffective if it is componentwise ineffective.
In what follows, for a disjoint union of smooth projective curves $C=\amalg_{i\in I}C_{i}$ with $g(C_{i}):=g_{i}$
and an ineffective theta characteristic $\Sigma=\{\Sigma_{i}\}_{i\in I}$ on $C$, we define 
$$
\widetilde{\tau}_{g}(C,\Sigma):=\prod_{i\in I}\widetilde{\tau}_{g_{i}}(C_{i},\Sigma_{i}),
$$
where $g:=\sum_{i\in I}g_{i}$ is the total genus of $C$.

\subsection
{A factorization of $\tau_{M}$}
\label{sect:10.2}
\par
We introduce the following twisted version of $\tau_{M}$.

\begin{definition}
\label{def:tau:spin}
Let $(X,\iota)$ be a $2$-elementary $K3$ surface  of type $M$ and let $\gamma$ be an $\iota$-invariant K\"ahler form on $X$.
If $M\not\cong{\Bbb U}(2)\oplus{\Bbb E}_{8}(2)$, define
$$
\begin{aligned}
\tau_{M}^{\rm spin}(X,\iota)
&:=
\prod_{\Sigma^{2}=K_{X^{\iota}},\,h^{0}(\Sigma)=0}
\left\{
{\rm Vol}(X,\gamma)^{\frac{14-r}{4}}
\tau_{{\bf Z}_{2}}(X,\gamma)(\iota)\,
\tau(X^{\iota},\Sigma;\gamma|_{X^{\iota}})^{-2}
\right.
\\
&\qquad\qquad\qquad
\left.
\times
\exp\left[
\frac{1}{8}\int_{X^{\iota}}
\log\left.\left(\frac{\eta\wedge\bar{\eta}}
{\gamma^{2}/2!}\cdot\frac{{\rm Vol}(X,\gamma)}{\|\eta\|_{L^{2}}^{2}}
\right)\right|_{X^{\iota}}
c_{1}(X^{\iota},\gamma|_{X^{\iota}})
\right]
\right\},
\end{aligned}
$$
where $\Sigma$ runs over all {\em ineffective} theta characteristics on $X^{\iota}$. 
If $M\cong{\Bbb U}(2)\oplus{\Bbb E}_{8}(2)$ and hence $X^{\iota}=\emptyset$, define
$$
\tau_{M}^{\rm spin}(X,\iota)
:=
\tau_{M}(X,\iota)^{2}.
$$
\end{definition}

Recall that the vector-valued modular form $f_{\Lambda}$ of type $\rho_{\Lambda}$ was defined by 
\eqref{eqn:correcting:modular:form:(r,a)=(2,0)}, \eqref{eqn:correcting:modular:form:(r,a)=(2,2)} when $r=22-r(\Lambda)=2$.
We extend its definition to the case $r\not=2$ by setting
$$
f_{\Lambda}:=\delta_{r,10}\,F_{\Lambda}.
$$
Then Theorem~\ref{thm:main:theorem:1} is interpreted as the modularity of $\tau_{M}^{\rm spin}$ as follows.

\begin{theorem}
\label{thm:tau:spin}
There exists a constant $C'_{M}>0$ depending only on $M$ such that 
the following equality of functions on ${\mathcal M}_{\Lambda}^{0}\setminus{\mathcal H}_{\Lambda}$ holds:
\begin{equation}
\label{eqn:formula:tau:spin}
\tau_{M}^{\rm spin}
=
C'_{M}\,
\left\|
\Psi_{\Lambda}(\cdot,2^{g-1}F_{\Lambda}+f_{\Lambda})
\right\|^{-1/2}.
\end{equation}
\end{theorem}

\begin{pf} 
For a $2$-elementary $K3$ surface $(X,\iota)$ of type $M$ with $(r,l,\delta)\not=(10,10,0),(10,8,0)$, recall that
$X^{\iota}$ consists of a curve of genus $g=g(M)$ and $k=k(M)$ smooth rational curves (cf. Section~\ref{sect:3.2}).
\par{\em (Case 1) }
If $(r,\delta)\not=(2,0),(10,0)$ and $X^{\iota}$ has no effective even theta characteristics, we get
\begin{equation}
\label{eqn:relation:tau:spin:non-spin:1}
\begin{aligned}
\tau_{M}^{\rm spin}(X,\iota)
&=
\frac
{\tau_{M}(X,\iota)^{2^{g-1}(2^{g}+1)}}
{\prod_{\Sigma^{2}=K_{X^{\iota}},\,h^{0}(\Sigma)=0}\widetilde{\tau}_{g}(X^{\iota},\Sigma)}
=
\frac
{\tau_{M}(X,\iota)^{2^{g-1}(2^{g}+1)}}
{\prod_{(a,b)\,{\rm even}}{\frak c}_{g}{\frak c}_{0}^{k}\|\theta_{a,b}(\varOmega(X^{\iota}))\|^{-4}}
\\
&=
c_{M}
\tau_{M}(X,\iota)^{2^{g-1}(2^{g}+1)}\|\chi_{g}(\varOmega(X^{\iota}))^{8}\|^{1/2}
\end{aligned}
\end{equation}
with 
$$
c_{M}:=({\frak c}_{g}^{-1}{\frak c}_{0}^{-k})^{2^{g-1}(2^{g}+1)}=\{(4\pi)^{g}e^{6(10-r)(2\zeta_{\bf Q}'(-1)+\zeta_{\bf Q}(-1))}\}^{-2^{g-1}(2^{g}+1)}.
$$
Here the first equality of \eqref{eqn:relation:tau:spin:non-spin:1} follows from Definition~\ref{def:tau:spin} and the second follows from
\eqref{eqn:spin:1/2:bosonization:1}.
Comparing Theorem~\ref{thm:main:theorem:1} and \eqref{eqn:relation:tau:spin:non-spin:1}, 
we get \eqref{eqn:formula:tau:spin} with $(C'_{M})^{2}:=c_{M}^{2}/C_{M}$ in this case.
\par{\em (Case 2) }
If $(r,\delta)=(2,0),(10,0)$ with $(r,l,\delta)\not=(10,10,0),(10,8,0)$ and if $X^{\iota}$ has a unique effective even theta characteristic
corresponding to a theta constant $\theta_{a_{0},b_{0}}(\varOmega(X^{\iota}))$ with respect to 
a suitable level $2$ structure, we get in the same way
\begin{equation}
\label{eqn:relation:tau:spin:non-spin:2}
\begin{aligned}
\tau_{M}^{\rm spin}(X,\iota)
&=
\frac
{\tau_{M}(X,\iota)^{(2^{g}-1)(2^{g-1}+1)}}
{\prod_{\Sigma^{2}=K_{X^{\iota}},\,h^{0}(\Sigma)=0}\widetilde{\tau}(X^{\iota},\Sigma)}
=
\frac
{\tau_{M}(X,\iota)^{(2^{g}-1)(2^{g-1}+1)}}
{\prod_{(a,b)\not=(a_{0},b_{0}),{\rm even}}{\frak c}_{0}^{k}{\frak c}_{g}\|\theta_{a,b}(\varOmega(X^{\iota}))\|^{-4}}
\\
&=
c_{M}
\tau_{M}(X,\iota)^{(2^{g}-1)(2^{g-1}+1)}
\|\Upsilon_{g}(\varOmega(X^{\iota}))\|^{1/2}
\end{aligned}
\end{equation}
with 
$$
c_{M}:=({\frak c}_{0}^{-k}{\frak c}_{g}^{-1})^{(2^{g}-1)(2^{g-1}+1)}=\{(4\pi)^{g}e^{6(10-r)(2\zeta_{\bf Q}'(-1)+\zeta_{\bf Q}(-1))}\}^{-(2^{g}-1)(2^{g-1}+1)}.
$$
By Theorem~\ref{thm:main:theorem:1} and \eqref{eqn:relation:tau:spin:non-spin:2}, 
we get \eqref{eqn:formula:tau:spin} with $(C'_{M})^{2}:=c_{M}^{2}/C_{M}$ in this case.
\par{\em (Case 3) }
If $(r,l,\delta)=(10,10,0)$, then $X^{\iota}=\emptyset$. Since we defined $g=1$ in this case, 
we get \eqref{eqn:formula:tau:spin} with $(C'_{M})^{2}:=C_{M}^{-1}$ by Definition~\ref{def:tau:spin} and Theorem~\ref{thm:main:theorem:1}. 
If $(r,l,\delta)=(10,8,0)$, then $X^{\iota}$ consists of two disjoint elliptic curves. In the same way as in \eqref{eqn:relation:tau:spin:non-spin:2},
we get \eqref{eqn:formula:tau:spin} with $(C'_{M})^{2}:=c_{M}^{2}/C_{M}$, $c_{M}=({\frak c}_{1}^{-2})^{9}=(4\pi)^{-18}$ in this case.
This completes the proof.
\end{pf}

\begin{remark}
Assume ${\mathcal H}_{\Lambda}\not=\emptyset$.
As the period of a $2$-elementary $K3$ surface of type $M$ approaches to a point of ${\mathcal H}_{\Lambda}$,
one of the ineffective even theta characteristics on its fixed curve becomes effective in the limit and the value $\tau_{M}^{\rm spin}$ 
jumps there. Because of this jumping, $\tau_{M}^{\rm spin}$ is a discontinuous function on ${\mathcal M}_{\Lambda}^{0}$.
Since $\|\Psi_{\Lambda}(\cdot,2^{g-1}F_{\Lambda}+f_{\Lambda})\|$ is also discontinuous along ${\mathcal H}_{\Lambda}$
by \cite[Th.\,1.1 (i)]{Schofer09}, it is an interesting problem of comparing $\tau_{M}^{\rm spin}$ and 
$\|\Psi_{\Lambda}(\cdot,2^{g-1}F_{\Lambda}+f_{\Lambda})\|$ on the locus ${\mathcal H}_{\Lambda}$.
\end{remark}

\subsection
{A uniqueness of elliptic modular form corresponding to $\tau_{M}^{\rm spin}$}
\label{sect:10.3}
\par
Set
\begin{equation}
\label{eqn:Phi:spin}
\Phi_{M}^{\rm spin}:=\Psi_{\Lambda}(\cdot,2^{g-1}F_{\Lambda}+f_{\Lambda}).
\end{equation}
Since $\tau_{M}^{\rm spin}=C'_{M}\|\Phi_{M}^{\rm spin}\|^{-1/2}$, $\Phi_{M}^{\rm spin}$ can be identified with $\tau_{M}^{\rm spin}$.
In this subsection, we study the uniqueness of elliptic modular form whose Borcherds lift is $\Phi_{M}^{\rm spin}$. 
\par
For a modular form $\varphi(\tau)$ of type $\rho_{\Lambda}$ with weight $1-b^{-}(\Lambda)/2$, we write
$$
\varphi(\tau)=\sum_{\gamma\in A_{\Lambda}}{\bf e}_{\gamma}\sum_{n\in\gamma^{2}/2+{\bf Z}}c_{\gamma}(n;\varphi)q^{n}
$$
for its Fourier expansion. The principal part of $\varphi$ is the Laurent polynomial defined as
$$
{\mathcal P}[\varphi]:=\sum_{\gamma\in A_{\Lambda}}{\bf e}_{\gamma}\sum_{n\in\gamma^{2}/2+{\bf Z},\,n<0}c_{\gamma}(n;\varphi)q^{n}
\in
{\bf C}[q^{-1/4}]\otimes{\bf C}[A_{\Lambda}].
$$
Notice that we used the notation
${\mathcal P}_{\leq0}[\varphi]=\sum_{\gamma\in A_{\Lambda}}{\bf e}_{\gamma}\sum_{n\in\gamma^{2}/2+{\bf Z},\,n\leq0}c_{\gamma}(n;\varphi)q^{n}$
in the previous sections. Obviously, ${\mathcal P}_{\leq0}[\varphi]-{\mathcal P}[\varphi]\in{\bf C}[A_{\Lambda}]$ is the constant term of $\varphi$.
\par
Similarly, for a Heegner divisor 
$H=\sum_{\gamma\in A_{\Lambda}}\sum_{n\in\gamma^{2}/2+{\bf Z},\,n<0}a_{\gamma}(n)\,H(n,\gamma)$ on $\Omega_{\Lambda}$,
where $H(n,\gamma):=\sum_{\lambda\in(\gamma+\Lambda)/\pm1,\,\lambda^{2}=2n}H_{\lambda}$, $n\in{\bf Q}$, $\gamma\in A_{\Lambda}$,
we define
$$
{\frak P}[H]:=\sum_{\gamma\in A_{\Lambda}}{\bf e}_{\gamma}\sum_{n\in\gamma^{2}/2+{\bf Z},\,n<0}a_{\gamma}(n)\,q^{n}
\in
{\bf C}[q^{-1/4}]\otimes{\bf C}[A_{\Lambda}].
$$
Comparing \eqref{eqn:principal:part:F}, \eqref{eqn:correcting:modular:form:(r,a)=(2,0)}, \eqref{eqn:principal:part:f:(r,a)=(2,2)}
with Theorem~\ref{thm:Borcherds:lift}, \eqref{eqn:divisor:Psi:Lambda:U}, \eqref{eqn:wt:div:(r,a)=(2,2)},
we have the equality
\begin{equation}
\label{eqn:comparison:principal:part}
{\mathcal P}[2^{g-1}F_{\Lambda}+f_{\Lambda}]={\frak P}[{\rm div}(\Phi_{M}^{\rm spin})]
\end{equation}
if $\Lambda\not\cong{\Bbb U}(2)^{\oplus2},({\Bbb A}_{1}^{+})^{\oplus2}$ or equivalently ${\mathcal D}_{\Lambda}\not=0$.
When $\Lambda\cong{\Bbb U}(2)^{\oplus2}$ or $({\Bbb A}_{1}^{+})^{\oplus2}$, we have 
${\mathcal P}[2^{g-1}F_{\Lambda}+f_{\Lambda}]\not=0$ but ${\frak P}[{\rm div}(\Phi_{M}^{\rm spin})]=0$.
Thus \eqref{eqn:comparison:principal:part} does not hold in these two cases.
Except them, the elliptic modular form $2^{g-1}F_{\Lambda}+f_{\Lambda}$ is characterized uniquely by the holomorphic torsion invariant $\tau_{M}^{\rm spin}$ 
as follows.

\begin{theorem}
\label{thm:converse}
If $\Lambda\not\cong{\Bbb U}(2)^{\oplus2},({\Bbb A}_{1}^{+})^{\oplus2}$, then there exists a unique $O(q_{\Lambda})$-invariant elliptic modular form 
$\varphi_{\Lambda}$ of type $\rho_{\Lambda}$ with weight $1-b^{-}(\Lambda)/2$ such that
\begin{equation}
\label{eqn:div&wt}
{\mathcal P}[\varphi_{\Lambda}]={\frak P}[{\rm div}(\Phi_{M}^{\rm spin})],
\qquad
c_{0}(0;\varphi_{\Lambda})/2={\rm wt}(\Phi_{M}^{\rm spin}).
\end{equation}
In particular, the $O(q_{\Lambda})$-invariance and \eqref{eqn:div&wt} characterize $2^{g-1}F_{\Lambda}+f_{\Lambda}$ uniquely.
\end{theorem}

\begin{pf}
Let $\varphi_{\Lambda}$ be an $O(q_{\Lambda})$-invariant modular form satisfying \eqref{eqn:div&wt}.
Set $\psi:=\varphi_{\Lambda}-(2^{g-1}F_{\Lambda}+f_{\Lambda})$. 
This is a modular form of type $\rho_{\Lambda}$ and weight $1-b^{-}(\Lambda)/2$ 
which is holomorphic at the cusp, is $O(q_{\Lambda})$-invariant, and satisfies $c_{0}(0; \psi)=0$. 
We must prove $\psi=0$. 
When $b^{-}(\Lambda)>2$, $\psi$ has negative weight and hence $\psi=0$. 
\par
Let $b^{-}(\Lambda)=2$. 
Then $\psi$ has weight $0$, so it must be a constant vector of ${\bf C}[A_{\Lambda}]$. 
From \cite[Lemma 3.9.1]{Nikulin80a} and \cite[Th.\,1]{Skoruppa08}, we deduce that 
${\rm Mp}_{2}({\bf Z})\times O(q_{\Lambda})$-invariant vectors in ${\mathbf C}[A_{\Lambda}]$ are scalar multiple of  
$\nu_{\Lambda}{\bf e}_{0}+{\bf v}_{\Lambda}+\nu_{\Lambda}\delta(\Lambda){\bf e}_{{\bf 1}_{\Lambda}}$
for some $\nu_{\Lambda}\in{\bf Z}_{>0}$, where 
${\mathbf v}_{\Lambda} =\sum_{\gamma\ne 0,{\bf 1}_{\Lambda},\gamma^{2}\equiv 0}{\bf e}_{\gamma}$. 
Since $c_{0}(0; \psi)=0$, we have $\psi=0$. 
\par
Let $b^{-}(\Lambda)=1$. By \cite[Th.\,5 and Th.\,9]{Skoruppa08} there is a canonical isomorphism between 
the space of modular froms of type $\rho_{\Lambda}$ and weight $1/2$ 
with the space of ${\rm Mp}_{2}({\mathbf Z})$-invariant vectors in 
${\mathbf C}[A_{{\Bbb A}_{1}}]\otimes{\mathbf C}[A_{\Lambda}] \simeq {\mathbf C}[A_{{\Bbb A}_{1}\oplus\Lambda}]$. 
By \cite[Th.\,1]{Skoruppa08}, the latter is generated by the vectors $I_{U}=\sum_{\gamma\in U}{\bf e}_{\gamma}$ 
where $U$ runs over self-dual isotropic subgroups of $A_{{\Bbb A}_{1}\oplus\Lambda}$. 
In this isomorphism the modular form corresponding to $I_{U}$ is given by 
$\sum_{\gamma\in U}\theta_{{\Bbb A}_{1}^{+}+\gamma_{1}}{\bf e}_{\gamma_{2}}$ 
where $\gamma=(\gamma_{1}, \gamma_{2})\in A_{{\Bbb A}_{1}\oplus\Lambda}$ (see \cite[Th.\,8]{Skoruppa08}).  
Now we have $\Lambda\cong{\Bbb A}_{1}^{+}\oplus{\Bbb U}$ or $({\Bbb A}_{1}^{+})^{\oplus2}\oplus{\Bbb A}_{1}$. 
In the first case, 
$A_{{\Bbb A}_{1}\oplus\Lambda}$ contains a unique nonzero isotropic element, 
and the corresponding modular form $\varphi$ has $c_{0}(0; \varphi)\ne 0$. 
In the second case, 
$A_{{\Bbb A}_{1}\oplus\Lambda}$ contains exactly two isotropic subgroups of rank $2$, 
which can be switched by an element of $O(q_{\Lambda})$. 
If $\varphi_{1}$ and $\varphi_{2}$ are the corresponding modular forms, then 
$\psi$ must be a scalar multiple of $\varphi_{1}+\varphi_{2}$. 
Again we have $c_{0}(0; \varphi_{1}+\varphi_{2})\ne 0$, so $\psi=0$.
\end{pf}

\section
{An equivariant analogue of Borcherds' conjecture}
\label{sect:11}
\par
In this section, we study an equivariant analogue of Borcherds' conjecture \cite{Borcherds98}. Let us explain briefly this conjecture.
Let $X_{K3}$ be the oriented $4$-manifold underlying a $K3$ surface. Let ${\mathcal E}$ be the set of Ricci-flat Riemannian metrics on $X_{K3}$ 
with normalized volume $1$ (cf. \cite{Yau78}). 
For $\gamma\in{\mathcal E}$, let $\Delta_{\gamma}$ be the Laplacian of $(X_{K3},\gamma)$ acting on $C^{\infty}(X_{K3})$.
Let $\det\Delta_{\gamma}$ be the regularized determinant of $\Delta_{\gamma}$. Then the assignment 
$\det\Delta\colon{\mathcal E}\ni\gamma\to\det\Delta_{\gamma}\in{\mathbf R}$
is a function on ${\mathcal E}$. In \cite[Example 15.2]{Borcherds98}, Borcherds conjectured that $\det\Delta$ is given by the automorphic form 
$\Phi_{{\Bbb L}_{K3}}(\cdot,1,E_{4}/\eta^{24})$ on $G({\Bbb L}_{K3})$, the period space of ${\mathcal E}$, 
where $E_{4}(\tau)$ is the Eisenstein series of weight $4$. To our knowledge, this conjecture is still open.
In Section~\ref{sect:11}, instead of the original Borcherds' conjecture, we study its equivariant analogue. 
\par
Let $\iota\colon X_{K3}\to X_{K3}$ be a $C^{\infty}$ involution. We define the lattices $H^{2}(X_{K3},{\mathbf Z})_{\pm}$ as in the preceding sections. 
Then $\iota$ is called {\it hyperbolic} if $H^{2}(X_{K3},{\mathbf Z})_{+}$ is Lorentzian.
Let ${\mathcal E}^{\iota}$ be the set of $\iota$-invariant Ricci-flat Riemannian metrics on $X_{K3}$ with volume $1$. 
Since we are interested in an equivariant analogue of Borcherds' conjecture, throughout Section~\ref{sect:11}, we restrict our consideration to 
those involutions $\iota$ satisfying
\begin{equation}
\label{eqn:nonemptyness}
{\mathcal E}^{\iota}\not=\emptyset.
\end{equation}
By \cite[Props.\,3.4, 3.6]{Yoshikawa08}, if $\iota$ is hyperbolic, then \eqref{eqn:nonemptyness} is equivalent to the existence of 
a complex structure $I$ on $X_{K3}$ such that $\iota$ is an anti-symplectic holomorphic involution on $(X_{K3},I)$. 
In particular, if $\iota$ is hyperbolic with \eqref{eqn:nonemptyness}, then $X_{K3}^{\iota}$ is a disjoint union of (possibly empty) smooth compact real surfaces.
\par
Let $\iota$ be a hyperbolic involution on $X_{K3}$ with \eqref{eqn:nonemptyness}. Its type is defined as the isometry class of 
$H^{2}(X_{K3},{\mathbf Z})_{+}$. Let $M$ be the type of $\iota$ and set $\Lambda:=M^{\perp_{{\Bbb L}_{K3}}}$.  
Then $M$ and $\Lambda$ are primitive $2$-elementary sublattices of ${\Bbb L}_{K3}$.
To formulate an equivariant analogue of Borcherds' conjecture, we construct two functions on ${\mathcal E}^{\iota}$. 
\par
Let $\gamma\in{\mathcal E}^{\iota}$. Let $C^{\infty}(X_{K3})_{\pm}$ be the $\pm1$-eigenspace of the $\iota$-action on $C^{\infty}(X_{K3})$. 
Since $\Delta_{\gamma}$ preserves $C^{\infty}(X_{K3})_{\pm}$, we can define $\Delta_{\gamma,\pm}:=\Delta_{\gamma}|_{C^{\infty}(X_{K3})_{\pm}}$.
Let $\zeta_{\pm}(s)$ be the spectral zeta function of $\Delta_{\gamma,\pm}$. 
The equivariant determinant of $\Delta_{\gamma}$ is defined as (cf. \cite{Bismut95})
$$
\det{}_{{\mathbf Z}_{2}}\Delta_{\gamma}(\iota):=\exp[-\zeta'_{+}(0)+\zeta'_{-}(0)].
$$
\par
Assume $X_{K3}^{\iota}\not=\emptyset$.
Let $S_{\gamma}$ be a spinor bundle on the fixed point set $(X_{K3}^{\iota}, \gamma|_{X^{\iota}_{K3}})$. 
Let $D_{S_{\gamma}}$ be the Dirac operator acting on $C^{\infty}(S_{\gamma})$. 
Let $\det D_{S_{\gamma}}^{2}$ be the regularized determinant of $D_{S_{\gamma}}^{2}$.
If $\zeta_{D_{S_{\gamma}}^{2}}(s)$ denotes the spectral zeta function of $D_{S_{\gamma}}^{2}$, then 
$$
\det D_{S_{\gamma}}^{2}:=\exp(-\zeta'_{D_{S_{\gamma}}^{2}}(0)).
$$
When $X_{K3}^{\iota}=\emptyset$, we define $\det D_{S_{\gamma}}^{2}:=1$. As an equivariant analogue of the function $\det\Delta$ on ${\mathcal E}$,
we consider the following function on ${\mathcal E}^{\iota}$.

\begin{definition}
\label{def:twisted:equiv:det}
For $\gamma\in{\mathcal E}^{\iota}$, define
$$
\tau_{\iota}^{\rm spin}(\gamma)
=
\prod_{S_{\gamma}\,\,{\rm ineffective}} (\det{}_{{\mathbf Z}_{2}}\Delta_{\gamma}(\iota))^{-2}\,\det D_{S_{\gamma}}^{2},
$$
where $S_{\gamma}$ runs over the spinor bundles on $(X_{K3}^{\iota}, \gamma|_{X^{\iota}_{K3}})$ with $\ker D_{S_{\gamma}}=0$.
\end{definition}

\par
Let us construct an automorphic function on the period space of ${\mathcal E}^{\iota}$. 
By e.g. \cite[Prop.\,3.6]{Yoshikawa08}, there exists a hyperk\"ahler structure $(I,J,K)$ on $(X_{K3},\gamma)$ such that
\begin{equation}
\label{eqn:compatible:HK:str}
\iota_{*}I=I\iota_{*},
\qquad
\iota_{*}J=-J\iota_{*},
\qquad
\iota_{*}K=-K\iota_{*}.
\end{equation}
By \cite[Lemma 3.17]{Yoshikawa08}, the pair of conjugate points of ${\mathcal M}_{\Lambda}$ defined as
\begin{equation}
\label{eqn:period:Ricci:flat:metric}
\pi_{\iota}(\gamma):=[\alpha(\omega_{J}\pm\sqrt{-1}\omega_{K})],
\qquad
\omega_{J}:=\gamma(\cdot,J(\cdot)),
\quad
\omega_{K}:=\gamma(\cdot,K(\cdot))
\end{equation}
is independent of the choice of a triplet $(I,J,K)$ satisfying \eqref{eqn:compatible:HK:str} and a marking, i.e.,
an isometry $\alpha\colon H^{2}(X_{K3},{\mathbf Z})\to{\Bbb L}_{K3}$ satisfying $\alpha(H^{2}(X,{\mathbf Z})_{+})=M$.
The pair of conjugate points $\pi_{\iota}(\gamma)\in{\mathcal M}_{\Lambda}$ is called the period of $\gamma\in{\mathcal E}^{\iota}$.

\begin{lemma}
\label{lemma:complex:conjugation}
$\pi_{\iota}^{*}\|\Psi_{\Lambda}(\cdot,2^{g-1}F_{\Lambda}+f_{\Lambda})\|$ is a well defined function on ${\mathcal E}^{\iota}$.
\end{lemma}

\begin{pf}
Write $\Lambda={\Bbb U}(-N)\oplus L$, $N\in\{1,2\}$, where $L$ is a Lorentzian lattice. Set $C_{L}:=\{x\in L\otimes{\mathbf R};\,\langle x,x\rangle>0\}$. 
Since $L$ is Lorentzian, $C_{L}$ consists of two components $C_{L}^{\pm}$ with $C_{L}^{-}=-C_{L}^{+}$. 
Then $L\otimes{\mathbf R}+\sqrt{-1}C_{L}\subset L\otimes{\mathbf C}$ is isomorphic to $\Omega_{\Lambda}$ via the map
$$
\exp_{N}\colon L\otimes{\mathbf R}+\sqrt{-1}C_{L}\ni z \to \exp_{N}(z):=( (1/N, \langle z,z\rangle/2), z) \in\Omega_{\Lambda}.
$$
Since $\exp_{N}\circ(-1_{L})=(1_{{\Bbb U}(-N)}\oplus -1_{L})\circ\exp_{N}$ and since $-1_{L}$ exchanges the components of $C_{L}$, 
$1_{{\Bbb U}(-N)}\oplus -1_{L}\in O(\Lambda)$ exchanges the components of $\Omega_{\Lambda}$. 
\par
Set $\eta:=\alpha(\omega_{J}+\sqrt{-1}\omega_{K})$. Let $z\in L\otimes{\mathbf R}+\sqrt{-1}C_{L}$ be such that $[\eta]=\exp_{N}(z)$. 
Then $\bar{\eta}=\alpha(\omega_{J}-\sqrt{-1}\omega_{K})$ and $[\bar{\eta}]=\exp_{N}(\bar{z})$.
Let $\Omega_{\Lambda}^{+}$ be the component of $\Omega_{\Lambda}$ such that $[\eta]\in\Omega_{\Lambda}^{+}$.
Let $\Omega_{\Lambda}^{-}$ be the remaining component. Then $\Omega_{\Lambda}^{-}=\overline{\Omega_{\Lambda}^{+}}$ and 
$[\bar{\eta}]\in\Omega_{\Lambda}^{-}$. 
Since $(1_{{\Bbb U}(-N)}\oplus -1_{L})[\bar{\eta}] = [(1_{{\Bbb U}(-N)}\oplus -1_{L})(\exp_{N}(\bar{z}))] = [\exp_{N}(-\bar{z})] \in \Omega_{\Lambda}^{+}$,
the point of $\Omega_{\Lambda}^{+}/O^{+}(\Lambda)$ corresponding to $[\bar{\eta}]$ is represented by $[\exp_{N}(-\bar{z})]$.
\par
For simplicity, write $\Psi(\cdot)$ for $\Psi_{\Lambda}(\cdot,2^{g-1}F_{\Lambda}+f_{\Lambda})$. Let $w$ be its weight.
Then $\Psi\in{\mathcal O}(L\otimes{\mathbf R}+\sqrt{-1}C_{L}^{+})$. By the definition of Petersson norm (cf. Section~\ref{sect:4.2}),
\begin{equation}
\label{eqn:Petersson:norm}
\|\Psi([\eta])\|^{2}=\langle\Im z,\Im z\rangle^{w} |\Psi(z)|^{2},
\qquad
\|\Psi([\bar{\eta}])\|^{2}=\langle\Im(-\bar{z}),\Im(-\bar{z})\rangle^{w} |\Psi(-\bar{z})|^{2}.
\end{equation}
Since $\Psi$ is a Borcherds product, it is expressed as a Fourier series
\begin{equation}
\label{eqn:def:Borcherds:prod}
\Psi(z) = \sum_{\lambda\in L^{\lor}} a(\lambda) \,e^{2\pi i\langle\lambda,z\rangle}
\qquad
(z\in L\otimes{\mathbf R}+\sqrt{-1}C_{L}^{+}),
\end{equation}
with $a(\lambda)\in{\mathbf Z}$.
Since $\overline{e^{2\pi i\langle l,z\rangle}}=e^{2\pi i\langle l,-\bar{z}\rangle}$ and $a(l)\in{\mathbf Z}$ for all $l\in L^{\lor}$ and hence
$\overline{\Psi(z)}=\Psi(-\bar{z})$, 
we deduce from \eqref{eqn:Petersson:norm} and \eqref{eqn:def:Borcherds:prod} that $\|\Psi([\eta])\|^{2}=\|\Psi([\bar{\eta}])\|^{2}$.
\end{pf}

Now we can formulate an equivariant analogue of Borcherds' conjecture as the coincidence of the two functions $\tau_{\iota}^{\rm spin}$
and $\pi_{\iota}^{*}\|\Psi_{\Lambda}(\cdot,2^{g-1}F_{\Lambda}+f_{\Lambda})\|$ on ${\mathcal E}^{\iota}$.
By Theorem~\ref{thm:tau:spin}, we have an affirmative answer to this problem.

\begin{theorem}
\label{thm:main:theorem:3} 
Let $\iota$ be a hyperbolic involution on $X_{K3}$ with ${\mathcal E}^{\iota}\not=\emptyset$.
Then the following equality of functions on $\pi_{\iota}^{-1}({\mathcal M}_{\Lambda}^{0}\setminus{\mathcal H}_{\Lambda})\subset{\mathcal E}^{\iota}$ holds:
$$
\tau_{\iota}^{\rm spin} = C''_{M}\,\pi_{\iota}^{*}\|\Psi_{\Lambda}(\cdot,2^{g-1}F_{\Lambda}+f_{\Lambda})\|^{-1/2},
$$
where $C''_{M}$ is a constant depending only on $M$.
\end{theorem}

\begin{pf}
Let $\gamma\in{\mathcal E}^{\iota}$. 
Let $(I,J,K)$ be a hyperk\"ahler structure on $(X_{K3},\gamma)$ with \eqref{eqn:compatible:HK:str} and set $X_{I}:=(X_{K3},I)$. 
Then $(X_{I},\iota)$ is a $2$-elementary $K3$ surface of type $M$. By \cite[Lemma 4.3]{Yoshikawa08}, we have
\begin{equation}
\label{eqn:equiv:det}
(\det{}_{{\mathbf Z}_{2}}\Delta_{\gamma}(\iota))^{-2} = \tau_{{\mathbf Z}_{2}}(X_{I},\gamma)(\iota).
\end{equation}
\par
On the other hand, it is classical that $D_{S_{\gamma}}$ can be identified with the Dolbeault operator $\sqrt{2}(\bar{\partial}+\bar{\partial}^{*})$ 
acting on $A^{0,*}(X_{K3}^{\iota},\Sigma_{S_{\gamma}})$, where $\Sigma_{S_{\gamma}}$ is the theta characteristic on $X_{K3}^{\iota}$ 
corresponding to $S_{\gamma}$. Here $X_{K3}^{\iota}$ is endowed with the complex structure induced by $\gamma$. 
Since $D_{S_{\gamma}}^{2}=2(\bar{\partial}+\bar{\partial}^{*})^{2}$ by this identification, we have 
$\zeta_{D_{S_{\gamma}}^{2}}(s)=2^{-s+1}\zeta_{\Sigma_{S_{\gamma}}}(s)$, where $\zeta_{\Sigma_{S_{\gamma}}}(s)$ is the spectral zeta function of
$(\bar{\partial}+\bar{\partial}^{*})^{2}|_{A^{0,0}(\Sigma_{S_{\gamma}})}$ as in Section~\ref{sect:10.1}. Hence
\begin{equation}
\label{eqn:det:Dirac:op}
\det D_{S_{\gamma}}^{2}
=
2^{2\zeta_{\Sigma_{S_{\gamma}}}(0)}\tau(X_{I}^{\iota},\Sigma_{S_{\gamma}};\gamma|_{X_{I}^{\iota}})^{-2}
=
2^{\frac{g-1}{3}}\tau(X_{I}^{\iota},\Sigma_{S_{\gamma}};\gamma|_{X_{I}^{\iota}})^{-2}.
\end{equation}
By \eqref{eqn:equiv:det}, \eqref{eqn:det:Dirac:op} and the definitions of $\tau^{\rm spin}_{M}$ and $\tau^{\rm spin}_{\iota}$, we get
\begin{equation}
\label{eqn:interpretation}
\tau_{\iota}^{\rm spin}(\gamma) = 2^{\frac{(g-1)N(\gamma,\iota)}{3}} \tau_{M}^{\rm spin}(X_{I},\iota),
\end{equation}
where $N(\gamma,\iota)$ is the number of ineffective spinor bundles on $(X_{K3}^{\iota},\gamma)$.
\par
Let $\gamma\in \pi_{\iota}^{-1}({\mathcal M}_{\Lambda}^{0}\setminus{\mathcal H}_{\Lambda})$. 
Since $N(\gamma,\iota)$ is a constant function on $\pi_{\iota}^{-1}({\mathcal M}_{\Lambda}^{0}\setminus{\mathcal H}_{\Lambda})$
by Theorem~\ref{thm:divisor:pullback:chi:upsilon} and \eqref{eqn:period:Ricci:flat:metric}
and since $\pi_{\iota}(\gamma)$ is given by the pair of $\overline{\pi}_{M}(X_{I},\iota)$ and its conjugate point by \eqref{eqn:period:Ricci:flat:metric}, 
the result follows from Theorem~\ref{thm:tau:spin} and \eqref{eqn:interpretation}.
\end{pf}

As in \cite{Yoshikawa08}, \cite[Sect.\,10]{Yoshikawa13}, we obtain, as a corollary of Theorem~\ref{thm:main:theorem:3}, 
an interpretation of Theorem~\ref{thm:tau:spin} on the mirror side, i.e., in terms of {\it real $K3$ surfaces}.
Recall that a pair consisting of a $K3$ surface and an {\em anti-holomorphic} involution is called a real $K3$ surface.
The set of fixed points of the involution on a real $K3$ surface is the set of real points.  
A holomorphic $2$-form on a real $K3$ surface is said to be {\em defined over ${\mathbf R}$} 
if it is mapped to its complex conjugation by the involution.
In view of mirror symmetry for $K3$ surfaces with involution \cite[Sect.\,2]{GrossWilson97}, 
the following corollary is a counter part of Theorem~\ref{thm:tau:spin} in mirror symmetry.

\begin{corollary}
\label{cor:real:K3}
Let $(Y,\sigma)$ be a real $K3$ surface. Let $M$ be the type of $\sigma$ and let $\alpha$ be a marking with $\alpha(H^{2}(X,{\mathbf Z})_{+})=M$. 
Let $\gamma$ be a $\sigma$-invariant Ricci-flat K\"ahler metric on $Y$ with volume $1$. Let $\omega_{\gamma}$ be the K\"ahler form of $\gamma$ and
let $\eta_{\gamma}$ be a holomorphic $2$-form on $Y$ defined over ${\mathbf R}$ such that 
$\eta_{\gamma}\wedge\bar{\eta}_{\gamma}=2\omega_{\gamma}^{2}$. Then
$$
\tau^{\rm spin}_{\sigma}(\gamma)
=
C''_{M}\,\|\Psi_{\Lambda}(\alpha(\Im(\eta_{\gamma})+\sqrt{-1}\omega_{\gamma}),2^{g-1}F_{\Lambda}+f_{\Lambda})\|^{-1/2},
$$
where $\Im(\eta_{\gamma})+\sqrt{-1}\omega_{\gamma}\in H^{2}(Y,{\mathbf R})+\sqrt{-1}{\mathcal K}_{Y}$ is a point of the complexified K\"ahler cone 
of the K\"ahler surface $(Y,\omega_{\gamma})$ with $B$-field $\Im(\eta_{\gamma})$.  
\end{corollary}

\begin{pf}
Set $I'=K$, $J'=-J$, $K'=I$, where $(I,J,K)$ is a hyperk\"ahler structure on $(X_{K3},\gamma)$ satisfying \eqref{eqn:compatible:HK:str} 
for $\sigma$. Then $(I',J',K')$ is a hyperk\"ahler structure on $(X_{K3},\gamma)$ such that $\sigma_{*}I'=-I'\sigma_{*}$, $\sigma_{*}J'=-J'\sigma_{*}$,
$\sigma_{*}K'=K'\sigma_{*}$. Set $Y=(X_{K3},I')$. Then $\sigma$ is an anti-holomorphic involution on $Y$. On the other hand, $\sigma$ is an
anti-symplectic holomorphic involution on $(X_{K3},K')$. 
By \cite[p.514]{GrossWilson97}, we see that $\Re(\eta_{\gamma})$ and $\Im(\eta_{\gamma})+\sqrt{-1}\omega_{\gamma}$ are a K\"ahler form and 
a holomorphic $2$-form on $(X_{K3},K')$, respectively. From this interpretation and Theorem~\ref{thm:main:theorem:3}, the result follows.
\end{pf}

Is it possible to prove Corollary~\ref{cor:real:K3} without passing through algebraic geometry? 
Such a proof will provide a new understanding of Theorem~\ref{thm:tau:spin}.


\end{document}